\documentclass[10pt, reqno]{amsart}
\usepackage{graphicx, amssymb, amsmath, amsthm}
\usepackage{comment}%%将一大段话变成zhushi
\usepackage{epsfig}  
\usepackage{xcolor}
\usepackage{enumerate}
\usepackage{hyperref}
\usepackage{color}
\numberwithin{equation}{section}

\let\Re=\undefined\DeclareMathOperator*{\Re}{Re}
\let\Im=\undefined\DeclareMathOperator*{\Im}{Im}

\def\ov{\overline}

\newcommand{\pa}{\parallel}
\newcommand{\pe}{\perp}

\newtheorem{theorem}{Theorem}[section]
\newtheorem{lemma}[theorem]{Lemma}
\newtheorem{corollary}[theorem]{Corollary}
\newtheorem{proposition}[theorem]{Proposition}
\newtheorem{definition}[theorem]{Definition}
\newtheorem{remark}[theorem]{Remark}

\begin{document}

\title[Semi-Classical propagation of singularities for Stokes system]{Semi-Classical propagation of singularities for the Stokes system}
\author{Chenmin Sun}
\address{Laboratoire J.~A.~Dieudonn\'e, CNRS, UMR-7351, Université Côte d'azur}
\email{chenmin.sun1991@gmail.com}
\maketitle

\begin{abstract}
We study the quasi-mode of Stokes system posed on a smooth bounded domain $\Omega$ with Dirichlet boundary condition.  We prove that the semi-classical defect measure associated with a sequence of solutions concentrates on the bicharacteristics of Laplacian as a matrix-valued Radon measure. Moreover, we show that the support of the measure is invariant under the Melrose-Sj\"{o}strand flow.
\end{abstract}

\begin{center}
	\begin{minipage}{100mm}
	\end{minipage}
\end{center}
%%%%%%%%%%%%%%%%%%%%%%%%%%%%%%%%%%%%%%%%%%%%%%%%%%%%%%%%%%%%%%%%%%%%%%%%%%%%%
%% introduction
%%%%%%%%%%%%%%%%%%%%%%%%%%%%%%%%%%%%%%%%%%%%%%%%%%%%%%%%%%%%%%%%%%%%%%%%%%%%%
\section{Introduction}

Let $\Omega\subset \mathbb{R}^d$ be a smooth bounded domain. Consider the eigenvalue problem of Stokes equation
\begin{equation}\label{Stokeschapter1}
\left\{
\begin{aligned}
&-\Delta u_k+\nabla P_k=\lambda_k^2u_k, \;\textrm{ in }\; \Omega\\
& \textrm{ div } u_k=0,\; \textrm{ in }\; \Omega \\
& u_k|_{\partial\Omega}=0
\end{aligned}
\right.
\end{equation}
where $u_k\in (H^2(\Omega))^{d}\cap V$,$\|u_k\|_{L^2}=1$, are $\mathbb{R}^d$-valued normalized eigenfunctions and
$$ V=\{u\in(H_0^1(\Omega))^d:\textrm{div }u=0\}.
$$
We collect several facts which are well-known in functional analysis: 
\begin{itemize}
	\item $u_k$ forms a orthonormal basis of 
	$$ H=\{u\in (L^2(\Omega))^d:\textrm{ div }u=0,u\cdot\nu|_{\partial\Omega}=0\}
	$$
	The canonical projector $\Pi:(L^2(\Omega))^d\rightarrow H$ is called Leray projector.
	\item The pressure $P_k\in L^2(\Omega)/{\mathbb{R}}$ satisfies $\int_{\Omega}P_k=0$.
	\item $\|\nabla u_k\|_{L^2}^2=\lambda_k^2,\|u_k\|_{H^2}\leq C\lambda_k^2, \|\nabla P_k\|_{L^2}\leq C\lambda_k^2$,$\|P_k\|_{L^2}\leq C\lambda_k^2$.
\end{itemize}
We rephrase the system \eqref{Stokeschapter1} by semi-classical reduction. Taking $h_k=\lambda_k^{-1}$ and $q_k=\lambda_k^{-1}P_k$, dropping the sub-index, we obtain the following $h$-dependent quasi-mode system
\begin{equation*}
\left\{
\begin{aligned}
&-h^2\Delta u-u+h\nabla q=0,\; \textrm{ in } \Omega\\
&h \textrm{ div } u=0,\;  \textrm{ in } \Omega \\
& u|_{\partial\Omega}=0
\end{aligned}
\right.
\end{equation*}
In this article,  we will study the following generalization by adding a quasi-mode:
\begin{equation}\label{WStokesmainchapter1}
\left\{
\begin{aligned}
&-h^2\Delta u-u+h\nabla q=f,\;\textrm{ in } \Omega\\
&h \textrm{ div } u=0,\; \textrm{ in } \Omega \\
& u|_{\partial\Omega}=0
\end{aligned}
\right.
\end{equation}
with the following conditions:
$$ \|u\|_{L^2}=1,\|h\nabla u\|_{L^2}=O(1),\|h^2\nabla^2u\|_{L^2}=O(1), $$
$$\|h\nabla q\|_{L^2}=O(1),f\in H, \|f\|_{L^2}=o(h).
$$
When $h$ is small, the corresponding solution $u=u(h)$ can be interpreted as high-frequency quasi-mode as its mass, i.e., the $L^2$ norm, is essentially concentrated on the frequency scale $h^{-1}$.  

Before stating the main result, it is worth mentioning the eigenvalue problem of Laplace operator in semi-classical version:
\begin{equation}\label{Laplacianchapter1}
\left\{
\begin{aligned}
&-h^2\Delta u-u=0\textrm{ in } \Omega\\
& u|_{\partial\Omega}=0.
\end{aligned}
\right.
\end{equation}
One method to capture the high-frequency behavior of the solutions of \eqref{Laplacianchapter1} is to use semi-classical defect measure associated to a bounded sequence $(u_k)$ of $L^2(\Omega)$ and to a sequence of positive scales $h_k$ converging to zero. This measure is aimed to describe quantitatively the oscillations of $(u_k)$ at the frequency scale $ h_k^{-1}$. More precisely, for any bounded sequence $(w_k)$ of $L^2(\mathbb{R}^d)$, there exists a subsequence of $(w_k)$ and a non-negative Radon measure $\mu$ on $T^*\mathbb{R}^d$ such that for any $a(x,\xi)\in\mathcal{S}(\mathbb{R}^{2d})$,
$$ \lim_{k\rightarrow\infty}(a(x,h_kD_x)w_k|w_k)_{L^2(\mathbb{R}^d)}=\langle\mu,a\rangle.
$$
When $\Omega$ is a bounded domain, the precise definition of defect measure corresponding to the boundary value problem will be given later. 

Let us mention that a counterpart of semi-classical defect measure, micro-local defect measure, was introduced by P.~G\'{e}rard \cite{Gerard1991} and L.~Tartar \cite{Tartar1990} independently. These objects are widely used in the study of control and stabilization, scattering theory and quantum ergodicity, see for example \cite{Burq-Lebeau2001}, \cite{Burq2002}, \cite{Gerard-Leichtnam1993}. 

In the context of semi-classical defect measure, the classical theorem of Melrose-Sj\"{o}strand about propagation of singularities (\cite{Melrose-SjostrandI},\cite{Melrose-SjostrandII}) for hyperbolic equation  can be rephrased as follows:
\begin{theorem}[\cite{Gerard-Leichtnam1993}]\label{MSchapter1}
	Assume that $\Omega$ is a smooth, bounded domain with no infinite order of contact on the boundary. Suppose $\mu$ is the semi-classical defect measure associated to the pair $(u_k,h_k)$ where $(u_k)$ is a sequence of solutions to \eqref{Laplacianchapter1} (with $h=h_k$) which are bounded in $L^2(\Omega)$. Then $\mu$ is invariant under the Melrose-Sj\"{o}strand flow. 
\end{theorem}   
We will give the precise definition of the Melrose-Sj\"{o}strand flow and the associated concept of the order of contact in the second section. Intuitively, these flows are the generalization of geometric optics. No infinite order of contact means that the trajectory of the flow can not tangent to the boundary with an infinite order.

The main result of this paper is as follows.
\begin{theorem}\label{Stchapter1}
	Assume that $\Omega$ is a smooth, bounded domain with no infinite order of contact on the boundary. Suppose $(u_k)$ is a sequence of solutions to the quasi-mode problem \eqref{WStokesmainchapter1} with semi-classical parameters $h=h_k$. Assume that $f_k\in H$, $\|f_k\|_{L^2(\Omega)}=o(h_k)$ and  $u_k$ converges weakly to $0$ in $L^2(\Omega)$. Assume that $\mu$ is a semi-classical measure associated to some subsequence of $(u_k,h_k)$, then $\mathrm{supp }(\mu)$ is invariant under the Melrose-Sj\"{o}strand flow. 	
\end{theorem}
We make some comments about the result. Firstly, the measure $\mu$ is Hermitian matrix-valued, and we have no information so far on the precise propagation for $\mu$ except for supp$(\mu)$. Secondly, since the eigenfunctions of Stokes operator converge weakly to $0$ in $L^2(\Omega)$, our results includes this special case.

The Propagation theorem for a given quasi-mode has many applications, in particular, it leads to the stabilization of the associated damped evolution system. In the context of the damped wave equation, it was shown that (see \cite{Rauch-Taylor1974}, \cite{Bardo-Lebeau-Rauch1992}, \cite{Lebeau1996}) under the geometric control condition, the energy decays exponentially.
An application of Theorem \ref{Stchapter1} is the stabilization of a hyperbolic Stokes system, a model in the theory of linear elasticity introduced in \cite{Lions1992}, under the geometric control condition. More precisely, consider the damped hyperbolic-Stokes system:
\begin{equation}\label{dampedStokeschapter1}
\left \{   
\begin{array}{ll}
\partial^2_tu-\Delta u+\nabla p+a(x)\partial_t u=0 &  \mbox{in}  \   \mathbb{R}\times \Omega,  \\
\mathrm{div } \ u  = 0 & \mbox{in} \  \mathbb{R}\times \Omega, \\
u = 0 & \mbox{on} \  \mathbb{R}\times \partial \Omega , \\
(u(0, x),\partial_tu(0, x))=(u_0,v_0)\in V\times H,
\end{array}
\right. 
\end{equation}
The energy 
$$  E[u](t)=\frac{1}{2}\int_{\Omega}(|\partial_tu|^2+|\nabla u|^2)dx
$$
is dissipative. In \cite{Chaves-Sun2017}, we use propagation Theorem \ref{Stchapter1} to show that the energy decays exponentially in time.

Let us describe briefly our strategy for the proof of Theorem \ref{Stchapter1}. The pressure term $q$ is harmonic and in heuristic, it can only have the influence to the solution near the boundary. Hence we will prove that the measure $\mu_i$ is propagated in the same way as Laplace quasi-mode (semi-classical analogue of wave equation) along the rays inside the domain. When a ray reaches the boundary, we need a more careful analysis between the wave-like propagation phenomenon and the concentration phenomenon of the pressure. It is difficult to get a simple propagation formula near the boundary, comparing to the treatment of quasi-mode problem of Laplace operator as in \cite{Burq-Lebeau2001},\cite{Gerard-Leichtnam1993}. We partition the phase space into elliptic region $\mathcal{E}$, hyperbolic region $\mathcal{H}$ and glancing surface $\mathcal{G}$. It turns out that no singularity accumulates near the elliptic region. For the hyperbolic region, we prove the propagation by the standard energy estimate, with an additional treatment when the incidence of the ray is right.  Near the glancing surface, we will follow the arguments of Ivrii and Melrose-Sj\"{o}strand. The main difference is that we will encounter two new cross terms essentially of the form $(q|u)_{L^2}$ after certain micro-localization. 
To overcome this difficulty, we further micro-localize the solution according to the distance to the glancing surface $\mathcal{G}$ and treat them separately. For the part nearing $\mathcal{G}$, we use the fact that the pressure decays fast away from the boundary while the solution can not concentrate too much near the boundary, provided that it is micro-localized close enough to the glancing surface. For the part away from $\mathcal{G}$, it can be well-controlled by induction argument, using geometric properties of the generalized bicharacteristic flow.

\section{Preliminary}
\subsection{Notations}
We will sometimes drop the sub-index $k$ for a sequence of functions $(u_k)$ and semi-classical parameters $h_k$. In this circumstance, the notion $\|u\|_{X}=O(1),o(1)$ as $h\rightarrow 0 $ should be understood as $\|u_k\|_X=O(1),o(1)$ as $k\rightarrow\infty$ (thus $h_k\rightarrow 0$) up to certain subsequence.

As in the introduction, we follow the notation in the context of the analysis of Stokes system (see  \cite{bookTemam}) $$V=\{u\in H_{0}^1(\Omega)^N: \textrm{div } u=0\}$$ and
$$H=\{u\in L^2(\Omega)^N: \textrm{div } u=0,u\cdot\mathbb{\nu}|_{\partial\Omega}=0\}.
$$
In this paper we always use $\mathbb{\nu}$ to denote the outward normal vector on $\partial\Omega$. 

For a manifold $M$, we let $TM$ be its tangent bundle and $T^*M$ be the cotangent bundle with canonical projection
$$ \pi:TM( \textrm{ or }T^*M)\rightarrow M.
$$
We will identify system \eqref{WStokesmainchapter1} as a system on differential form
\begin{equation}\label{WStokesmain1formchapter1}
\left\{
\begin{aligned}
&h^2\Delta_H u-u+h\mathrm{d} q=f\textrm{ in } \Omega\\
&h \mathrm{d}^* u=0 \textrm{ in } \Omega \\
& u|_{\partial\Omega}=0
\end{aligned}
\right.
\end{equation}
where the unknown $u\in \Lambda^1(\Omega)$ is 1-form, and
$$ \mathrm{d}:\Lambda^p(\Omega)\rightarrow \Lambda^{p+1}(\Omega), \mathrm{d}^*:\Lambda^{p+1}(\Omega)\rightarrow \Lambda^p(\Omega)
$$  
are exterior differential and divergence operator on forms, with respectively. Recall also that he Hodge Laplace operator is defined by $$\Delta_H=\mathrm{d}\mathrm{d}^*+\mathrm{d}^*\mathrm{d}=(\mathrm{d}+\mathrm{d}^*)^2.$$

In the turbulent neighborhood of boundary, we can identify $\Omega$ locally as one side of the turbulent neighborhood denoted by $Y_+=[0,\epsilon_0)\times X$, $X=\{x'\in\mathbb{R}^{d-1}:|x'|<1\}$. We denote by $\underline{\partial}Y_+=Y_+|_{y=0}$ and $Y_+^0=Y_+|_{y>0}$. For $x\in\ov{\Omega}$, we note $x=(y,x')$, where $y\in [0,\epsilon_0),x'\in X$, and $x\in\partial\Omega$ if and only if $x=(0,x')$. In this coordinate system, the Euclidean metric $dx^2$ can be written as matrices
\begin{equation*}
\ov{g}=\left(
\begin{array}{cc}
1 & 0 \\
0 & g(y,x') \\
\end{array}
\right),
\ov{g}^{-1}=\left(
\begin{array}{cc}
1 & 0 \\
0 & g^{-1}(y,x') \\
\end{array}
\right),
\end{equation*}
with $|\xi'|_{g^{-1}(y,x'))}^2=\langle\xi',g^{-1}(y,x')\xi'\rangle_{\mathbb{R}^{d-1}}=g^{jk}\xi'_j\xi'_k$ be the induced metric on $T^*\partial\Omega$, parametrized by $y$. Note that
$|\xi'|_{g^{-1}(0,x')}^2=\langle\xi',g^{-1}(0,x')\xi'\rangle_{\mathbb{R}^{d-1}}=g^{jk}\xi'_j\xi'_k$ is the natural norm on $T^*\partial\Omega$, dual of the norm on $T\partial\Omega$, induced by the canonical metric on $\ov{\Omega}$. Write $(x,\xi)=(y,x',\eta,\xi')$ and denote by $|\xi|$ the Euclidean norm on $T^*\mathbb{R}^d$.
For $u,v\in \Lambda^1(Y_+)$ with support in the local chart of turbulence neighborhood, we define the $L^2$ norms and inner product on $[0,\epsilon_0)\times X$ via
$$ \|u\|_{L^2(Y_+)}^2:=\int_0^{\epsilon_0}\int_X \langle u|u\rangle \sqrt{\det(\ov{g})}dx'dy,$$  $$ (u|v)_{L^2(Y_+)}:=(u|v)_{Y_+}:=\int_0^{\epsilon_0}\int_X \langle u|v\rangle \sqrt{\det(\ov{g})}dx'dy,
$$
$$\|u(y,\cdot)\|_{L^2(\underline{\partial}Y_+)}^2:=\int_X\langle  u(y,\cdot)|u(y,\cdot)\rangle\sqrt{\det(g)}dx',$$  $$ (u|v)_{L^2(\underline{\partial}Y_+)}(y):=\int_X \langle  u(y,\cdot)|v(y,\cdot)\rangle\sqrt{\det(g)}dx',
$$
where for $u=u_0dy+u_jdx'^j,v=v_0dy+v_jdx'^j$,
$$ \langle u|v\rangle=u_0\ov{v}_0+u_j\ov{v_k}g^{jk}.
$$ 
In certain situations we also use global notation for $L^2$ inner product:
\begin{gather*}
(u|v)_{\Omega}:=\int_{\Omega}u\cdot\ov{v}dx,\quad
(f|g)_{\partial\Omega}:=\int_{\partial\Omega}f\cdot\ov{g}d\sigma(x).
\end{gather*}
We will identify the unknown vector fields $u,v,$ etc. with their dual 1-forms. Formulation of differential form will simplify some calculations. 
In the turbulence neighborhood, we write a vector field
$$L=L_{\pe}\frac{\partial}{\partial y}+L_{\pa},\quad 
L_{\pa}=\sum_{j=1}^{d-1}L_{\pa,j}\frac{\partial}{\partial x_j'} 
$$ and we write $L=(L_{\pe},L_{\pa}).$ The normal component obeys the following convention: $(a,0)=-a\nu$. 

Following \cite{Chaves-Lebeau2016}, we will write down system \eqref{WStokesmainchapter1} in the turbulent neighborhood. For
$u=(u_{\pe},u_{\pa})$, equation \eqref{WStokesmainchapter1} can be written as:
\begin{equation}
\left\{
\begin{aligned}
& (-h^2\Delta_{\pa}-1)u_{\pa}+h\nabla_{x'}q=f_{\pa},\\
& (-h^2\Delta_{g}-1)u_{\pe}+h\partial_y q=f_{\pe},\\
&h\textrm{ div }_{\pa}u_{\pa}+\frac{h}{\sqrt{\det g}}\partial_y(\sqrt{\mathrm{det }g} u_{\pe})=0
\end{aligned}
\right.
\end{equation}
where
$$ h^2\Delta_{\pa}=h^2\partial_y^2-\Lambda^2(y,x',hD_{x'})+hM_{\pa}(y,x',hD_x')+hM_1(y,x')h\partial_y,
$$
$$ h^2\Delta_{g}=h^2\partial_y^2-\Lambda^2(y,x',hD_{x'})+hM_{\pe}(y,x',hD_x')+
hN_1(y,x')h\partial_y,
$$
$$ h\textrm{ div }_{\pa}u_{\pa}=\frac{h}{\sqrt{\det g}}\sum_{j=1}^{N-1}\partial_{x'_j}(\sqrt{\det g}u_{\pa,j}).
$$
$h^2\Lambda^2(y,x',hD_{x'})$ has the symbol $\lambda^2=|\xi'|^2_{\alpha(y,\cdot)}$, and $M_{\pa,\pe}$ are both first-order matrix-valued semi-classical differential operators. 

\subsection{Geometric Preliminaries}
Denote by $^bT\ov{\Omega}$ the vector bundle whose sections are the vector fields $X(p)$ on $\ov{\Omega}$ with $X(p)\in T_p\partial\Omega$ if $p\in\partial\Omega$. Moreover, denote by $^bT^*\ov{\Omega}$ the Melrose's compressed cotangent bundle which is the dual bundle of $^bT\ov{\Omega}$. Let 
$$ j:T^*\ov{\Omega}\rightarrow ^bT^*\ov{\Omega}
$$
be the canonical map. In our geodesic coordinate system near $\partial\Omega$, $^bT\ov{\Omega}$ is generated by the vector fields $\frac{\partial}{\partial x'_1},\cdot\cdot\cdot, 
\frac{\partial}{\partial x'_{d-1}},y\frac{\partial}{\partial y}$ and thus $j$ is defined by
$$ j(y,x';\eta,\xi')=(y,x';v=y\eta,\xi').
$$

The principal symbol of operator $P_h=-(h^2\Delta+1)$ is $$p(y,x',\eta,\xi')=\eta^2+|\xi'|_{g^{-1}(y,x')}^2-1.$$
By Car($P$) we denote the characteristic variety of $p$:
$$\textrm{Car}(P):=\{(x,\xi)\in T^*\mathbb{R}^{d}|_{\ov{\Omega}}:p(x,\xi)=0\},\quad 
Z:=j(\textrm{Car}(P)).
$$
By writing in another way
$$p=\eta^2-r(y,x',\xi'),\quad  r(y,x',\xi')=1-|\xi'|_{g^{-1}(y,x')}^2,
$$
we have the decomposition
$$ T^*\partial\Omega=\mathcal{E}\cup\mathcal{H}\cup\mathcal{G},
$$
according to the value of $r_0:=r|_{y=0}$ where
$$\mathcal{E}=\{r_0<0\},
\mathcal{H}=\{r_0>0\},
\mathcal{G}=\{r_0=0\}.
$$
The sets $\mathcal{E},\mathcal{H},
\mathcal{G}$ are called elliptic, hyperbolic and glancing, with respectively. 

For a symplectic manifold $S$ with local coordinate $(z,\zeta)$, a Hamiltonian vector field associated with a real function $f$ is given by
$$ H_f=\frac{\partial f}{\partial \zeta}\frac{\partial}{\partial z}-\frac{\partial f}{\partial z}\frac{\partial}{\partial\zeta}.
$$
Now for $(x,\xi)\in\Omega$ far away from the boundary, the Hamiltonian vector field associated to the characteristic function $p$ is given by
$$ H_p=2\xi\frac{\partial}{\partial x}.
$$
We call the trajectory of the flow
$$ \phi_s:(x,\xi)\mapsto (x+s\xi,\xi)
$$
bicharacteristic or simply ray, provided that the point $x+s\xi$ is still in the interior.

To classify different situations as a ray reaching the boundary, we need more accurate decomposition of the glancing set $\mathcal{G}$. Let $r_1=\partial_yr|_{y=0}$ and define
$$ \mathcal{G}^{k+3}=\{(x',\xi'):r_0(x',\xi')=0,H_{r_0}^j(r_1)=0,\forall j\leq k;H_{r_0}^{k+1}(r_1)\neq 0\}, k\geq 0
$$
$$
\mathcal{G}^{2,\pm}:=\{(x',\xi'):r_0(x',\xi')=0,\pm r_1(x',\xi')>0\},\mathcal{G}^2:=\mathcal{G}^{2,+}\cup\mathcal{G}^{2,-}.
$$
Denote by $\mathcal{G}^{\infty}=\mathcal{G}\setminus\big(\cup_{j\geq 2}\mathcal{G}^j\big)$. We say that there is no infinite order of contact on the boundary if $\mathcal{G}^{\infty}=\emptyset$.

By setting
$$ H_p^G:=H_p+\frac{H_p^2y}{H_y^2p}H_y,
$$
$H_p^G$ is a well-defined vector field tangent to $\mathcal{G}$ which is called the gliding vector field.
Given a ray $\gamma(s)$ with $\pi(\gamma(0))\in \Omega$ and $\pi(\gamma(s_0))\in\partial\Omega$ be the first point that reaches the boundary.  If $\gamma(s_0)\in\mathcal{H}$, then $\eta_{\pm}(\gamma(s_0))=\pm\sqrt{r_0(\gamma(s_0))}$ are the two different roots of $\eta^2=r_0$ at this point. Notice that the ray starting with direction $\eta_-$ will leave $\Omega$,  while the ray with direction $\eta_+$ will enter the interior of $\Omega$. This motivates the following definition of broken bicharacteristic:
\begin{definition}[\cite{bookHormander(III)}]
	A broken bicharacteristic arc of $p$ is a map:
	$$ s\in I\setminus B\mapsto \gamma(s)\in T^*\Omega\setminus \{0\},
	$$
	where $I$ is an interval on $\mathbb{R}$ and $B$ is a discrete subset, such that
	\begin{itemize}
		\item If $J$ is an interval contained in $I\setminus B$, then 
		$ s\in J\mapsto \gamma(s)
		$
		is a bicharacteristic of $p$ over $\Omega$.
		\item If $s\in B$, then the limits $\gamma(s^+)$ and $\gamma(s^-)$ exist and belongs to $T_x^*\ov{\Omega}\setminus \{0\}$ for some $x\in\partial\Omega$, and the projections in $T_x^*\partial\Omega\setminus \{0\}$ are the same hyperbolic point.
	\end{itemize}
\end{definition}
When a ray $\gamma(s)$ reaches a point $\rho_0\in\mathcal{G}$, there are several situations. If $\rho_0\in\mathcal{G}^{2,+}$, then the ray passes transversally over $\rho_0$ and enters $T^*\Omega$ immediately. If $\rho_0\in\mathcal{G}^{2,-}$ or $\rho_0\in\mathcal{G}^k$ for some $k\geq3$, then we can continue it inside $T^*\partial\Omega$
as long as it can not leave the boundary along the trajectory of the Hamiltonian flow of $H_{-r_0}$. We now give the precise definition.

\begin{definition}[\cite{bookHormander(III)}]
	A generalized bicharacteristic ray of $p$ is a map:
	$$ s\in I\setminus B\mapsto \gamma(s)\in (T^{*}\ov{\Omega}\setminus T^*\partial\Omega) \cup \mathcal{G}
	$$
	where $I$ is an interval on $\mathbb{R}$ and $B$ is a discrete set of $I$ such that $p\circ \gamma=0$ and the following:
	\begin{enumerate}
		\item $\gamma(s)$ is differentiable and $\frac{d\gamma}{ds}=H_{p}(\gamma(s))$ if $\gamma(s)\in T^*\ov{\Omega}\setminus T^*\partial\Omega $ or $\gamma(s)\in\mathcal{G}^{2,+}$.
		\item Every $s\in B$ is isolated, $\gamma(s)\in T^*\ov{\Omega}\setminus T^*\partial\Omega$ if $s\neq t$ and $|s-t|$ is small enough, the limits $\gamma(s^{\pm})$ exist and are different points in the same fibre of $T^*\partial\Omega$.
		\item $\gamma(s)$ is differentiable and $\frac{d\gamma}{ds}=H_{p}^G(\gamma(s))$ if $\gamma(s)\in \mathcal{G}\setminus \mathcal{G}^{2,+}$.
	\end{enumerate}
\end{definition}
\begin{remark}
	The definition above does not depend on the choice of local coordinate, and in the geodesic coordinate system, the map
	$$ s\mapsto (y(s),\eta^2(s),x'(s),\xi'(s)) 
	$$
	is always continuous and 
	$$ s\mapsto (x'(s),\xi'(s))
	$$
	is always differentiable and satisfies the ordinary differential equations
	$$ \frac{dx'}{dt}=-\frac{\partial r}{\partial \xi'},\frac{d\xi'}{dt}=\frac{\partial r}{\partial x'},
	$$
	the map $s\mapsto y(s)$ is left and right differentiable with derivative $2\eta(s^{\pm})$ for any $s\in B$ (hyperbolic point).
	
	Moreover, there is also the continuous dependence with the initial data, namely the map
	$$ (s,\rho)\mapsto (y(s,\rho),\eta^2(s,\rho),x'(s,\rho),\xi'(s,\rho))
	$$
	is continuous. We denote the flow map by $\gamma(s,\rho)$.
\end{remark}

\begin{remark}
	Under the map $j:T^*\ov{\Omega}\rightarrow ^bT^*\ov{\Omega}$, one could regard $\gamma(s)$ as a continuous flow on the compressed cotangent bundle $^bT^*\ov{\Omega}$, and it is called the Melrose-Sj\"{o}strand flow. We will also call each trajectory generalized bicharacteristic or simply ray in the sequel. 
\end{remark}
From the classical result of Melrose-Sj\"ostrand, a generalized bicharacteristic that does not meet $\mathcal{G}^{\infty}$ is uniquely defined (see Corollary 24.3.10 in \cite{bookHormander(III)}). Then, having $\mathcal{G}^{\infty}=\emptyset$, meaning $\mathcal{G}^j=\emptyset$ for some $j\geq 2$, implies the uniqueness of all generalized bicharacteristics and thus the existence of the Melrose-Sj\"ostrand flow. We refer to Example 24.3.11 in \cite{bookHormander(III)} where nonuniqueness occurs precisely in the case of a point in $\mathcal{G}^{\infty}$.
\subsection{definition of defect measure}
We follow closely as in \cite{Burq2002} and the one can find in \cite{Gerard-Leichtnam1993} for a little different but comprehensive introduction.

We denote by $S^m$ the usual symbol class. Define the partial symbol class $S^m_{\partial}$ and the class of boundary $h$-pseudo-differential operators $ \mathcal{A}_h^m$ as follows 
\begin{gather*}
S^m_{\partial}:=\{a(y,x',\xi'):\sup_{\alpha,\beta,y\in[0,\epsilon_0]}|\partial_{x'}^{\alpha}
\partial_{\xi'}^{\beta}a(y,x',\xi')|\leq C_{m,\alpha,\beta}(1+|\xi'|)^{m-\beta}\}.\\
\mathcal{A}_h^m=:\textrm{Op}_h^{comp}(S^m)+\textrm{Op}_h(S_{\xi'}^m):=\mathcal{A}^m_{h,i}+\mathcal{A}^m_{h,,\partial}.
\end{gather*}

Consider functions of the form $a=a_i+a_{\partial}$ with $a_i\in C_c^{\infty}(\Omega\times\mathbb{R}^{d})$ which can be viewed as a symbol in $S^0$, and $a_{\partial}\in C_c^{\infty}(Y_+\times\mathbb{R}^{d-1})$ can be viewed as a symbol in $S_{\xi'}^0$. We quantize $a$ as follows: Take $\varphi_i\in C_c^{\infty}(\Omega)$, equal to 1 near the $x$-projection of $\mathrm{supp}(a_i)$ and  $\varphi_{\partial}\in C_c^{\infty}(\mathbb{R}^d)$, equal to 1 near the $x$-projection of $\mathrm{supp}(a_{\partial})$. Define 
\begin{equation*}
\begin{split}
\textrm{Op}^{\varphi_i,\varphi_{\partial}}_h(a)f(y,x')=&
\frac{1}{(2\pi h)^d}\int_{\mathbb{R}^{2d}}
e^{\frac{i(x-z)\xi}{h}}a_i(x,\xi)\varphi_i(z)f(z)dzd\xi\\
&+\frac{1}{(2\pi h)^{d-1}}
\int_{\mathbb{R}^{2(d-1)}}
e^{\frac{i(x'-z')\xi'}{h}}
a_{\partial}(y,x',\xi')\varphi_{\partial}(y,z')f(y,z')dz'd\xi'.
\end{split}
\end{equation*}
According to the symbolic calculus, the operator $\textrm{Op}^{\varphi_i,\varphi_{\partial}}_h(a)$ does not depend on the choice of functions $\varphi_i,\varphi_{\partial}$, modulo operators of norms $O_{L_{\textrm{loc}}^2\rightarrow L_{\textrm{comp}}^2}(h^{\infty})$, and we will use the notion $\mathrm{Op}_h(a)$ in the sequel. Notice that the acting of tangential operator $\mathrm{Op}_h(a_{\partial})$ can be viewed as pseudo-differential operator on the manifold $\partial\Omega$, parametrized by the parameter $y\in[0,\epsilon_0)$.  The bounded family of operators $\mathcal{A}^m_{h,\partial}$ is defined uniquely up to a family of operators with norms uniformly dominated by $Ch$, as $h\rightarrow 0$.  Moreover, for any family $(A_h)$, such that 
$$ \|A_h-\textrm{Op}_h(a_{\partial})\|_{L^2\rightarrow L^2}=O(h),
$$
the principal symbol $\sigma(A)$ is determined uniquely as a function on $T^*\partial\Omega$, smoothly depending on $y$, i.e.
$\sigma(A)\in C^{\infty}([0,\epsilon_0)\times T^*\partial\Omega)$. 

When we deal with vector-valued functions, we could require the symbol $a$ to be matrix-valued. Now for any sequence of vector-valued function $w_k$, uniformly bounded in $L^2(\Omega)$, there exists a subsequence (still use $w_k$ for simplicity), and a nonnegative definite Hermitian matrix-valued Radon measure $\mu_i$ on $T^*\Omega$ such that
$$ \lim_{k\rightarrow 0}(\textrm{Op}_{h_k}(a_i)w_k|w_k)_{L^2}=\langle\mu_i,a_i\rangle
:=
\int_{T^*\Omega}\textrm{tr }(a_id\mu_i).
$$  
For a proof, see for example \cite{Burq2002} or the textbook \cite{bookZworski}, and the micro-local version was appeared in \cite{Gerard1991}.

From now on the symbols and operators will be scalar-valued unless otherwise specified. Suppose $u_k$ be a sequence of solutions to

\begin{equation}\label{Stokes-semisequencechapter1}
\left\{
\begin{aligned}
&-h_k^2\Delta u_k-u_k+h_k\nabla q_k=f_k,\;(u_k,f_k)\in (H^2(\Omega)\cap V)\times H,\\
&h_k\mathrm{div } u_k=0,\;\textrm{in} \ \Omega \\
\end{aligned}
\right.
\end{equation}
under the assumptions below:
\begin{equation}\label{assumption1chapter1}
\begin{split}
&\|u_k\|_{L^2(\Omega)}=O(1),\quad f_k\in H\textrm{ and }\|f_k\|_{L^2(\Omega)}=o(h_k),\\
&\|h\nabla q_k\|_{L^2(\Omega)}=O(1),\quad \int_{\Omega}q_kdx=0.\\
\end{split}
\end{equation}
Let $\mu_i$ be the interior measure associated with $(u_k)_k$. Then the following result shows that $\mu_i$ is supported on $\textrm{Car}(P)$.
\begin{proposition}\label{Carchapter1}
	Let $a_i\in C_c^{\infty}(\Omega\times\mathbb{R}^d)$ be equal to $0$ near $\mathrm{Car}(P)$, then we have
	$$ \lim_{k\rightarrow\infty}(\mathrm{Op}_{h_k}(a_i)u_k|u_k)_{L^2}= 0.
	$$ %%\mathrm{}这个命令可以把数学环境中的斜体变为整体
\end{proposition} 
\begin{proof} Note that the symbol $b(x,\xi)=\frac{a_i(x,\xi)}{|\xi|^2-1}\in S^0$ is well-defined from the assumption on $a_i$. From symbolic calculus, we have
	$$ \textrm{Op}_{h_k}(a_i)=B_{h_k}(-h_k^2\Delta-1)+O_{L^2\rightarrow L^2}(h_k).
	$$
	Therefore
	\begin{equation*}
	\begin{split}
	(B_{h_k}(-h_k^2\Delta-1)u_k|u_k)_{L^2}&=(B_{h_k}f_k|u_k)_{L^2}-(B_{h_k}h_k\nabla q_k|u_k)_{L^2}\\
	&=o(1)+([h_k\nabla,B_{h_k}]q_k|u_k)_{L^2}-(h_k\nabla B_{h_k}q_k|u_k)_{L^2}\\
	&=o(1), \textrm{ as }k\rightarrow\infty,
	\end{split}
	\end{equation*}
	where in the last line we have used the symbolic calculus, integration by part, and Lemma \ref{press.normchapter1}.
\end{proof}

Now we denote by $Z=j(\mathrm{Car}(P))$.
Proposition \ref{Carchapter1} indicates that the interior defect measure $\mu_i$ is supported on 
$Z$. To define the defect measure near the boundary, we have to check that if 
$a_{\partial}\in C_c^{\infty}(U\times\mathbb{R}^{d-1})$ vanishing near $Z$ (i.e. $a_{\partial}$ is supported in the elliptic region for all $y$ small) then
$$ \lim_{k\rightarrow\infty}(\textrm{Op}_{h_k}(a_{\partial})u_k|u_k)_{L^2}=0.
$$
Indeed, this can be ensured by the analysis of boundary value problem in the elliptic region, which will be given later. Now for any family of operator $A_h\in\mathcal{A}_h^0$, let $a=\sigma(A_h)$ be the principal symbol of $A_h$
and we define $\kappa(a)\in C^0(Z)$ via
$\kappa(a)(\rho):=a(j^{-1}(\rho)).$
Note that $Z$ is a locally compact metric space and the set
$$ \{\kappa(a):a=\sigma(A_h),A_h\in\mathcal{A}_h^0\}
$$
is a locally dense subset of $C^0(Z)$.
We then have the following proposition, which guarantees the existence of a Radon measure on $Z$:
\begin{proposition}
	There exists a subsequence of $u_k,h_k$ and a nonnegative definite Hermitian matrix-valued Radon measure $\mu$ on $Z$, such that
	$$ \lim_{k\rightarrow\infty}
	(A_{h_k}u_k|u_k)_{L^2}=\langle\mu,\kappa(a)\rangle, a=\sigma(A_{h}),\forall A_h\in \mathcal{A}_h^0.
	$$
\end{proposition} 
The proof of this result can be found in \cite{Burq2002}, see also \cite{Burq-Lebeau2001} and \cite{Gerard1991} for its micro-local counterpart. Notice that if we write 
$ a=a_i+a_{\partial},
$
then
$$
(A_k u_k| u_k)\rightarrow \int_{T^*\Omega}\textrm{ tr }(a_i(\rho)d\mu_i(\rho))+
\int_Z \textrm{ tr }(a_{\partial}(\rho)d\mu(\rho)).
$$

The following result shows that information about frequencies higher than the scale $h_k^{-1}$ is not lost, and the measure $\mu$ contains the relevant information of the sequence $(u_k)$.
\begin{proposition}\label{frequencyconcentratechapter1}
	The sequence of solution $(u_k)$ is $h_k-$oscillating in the following sense:
	$$ \lim_{R\rightarrow\infty}\limsup_{k\rightarrow\infty}
	\int_{|\xi|\geq Rh_k^{-1}}
	|\widehat{\psi u_k}(\xi)|^2d\xi=0,\forall \psi\in C_c^{\infty}(\Omega),
	$$
	$$ \lim_{R\rightarrow\infty}\limsup_{k\rightarrow\infty}
	\int_0^{\epsilon_0}dy\int_{|\xi'|\geq Rh_k^{-1}}
	|\widehat{\psi u_k}(y,\xi')|^2d\xi'=0,\forall \psi\in C_c^{\infty}(\ov{\Omega}),
	$$
	where in the second formula,  the Fourier transform is only taken for the $x'$ direction. Consequently, the total mass $\|u_k\|_{L^2(\Omega)}^2$ converges to $\langle\mu_i,T^*\partial\Omega\rangle+\langle\mu,\mathbf{1}_{Z}\rangle$.
\end{proposition}
The proof will be given in appendix.

\section{A priori information about the system}
\subsection{Information about the trace}
We consider the semi-classical Stokes system
\begin{equation}\label{Stokes-semichapter1}
\left\{
\begin{aligned}
&-h^2\Delta u-u+h\nabla q=f,\;(u,f)\in (H^2(\Omega)\cap V)\times H\\
&h\ \mathrm{ div } u=0,\;\textrm{in} \ \Omega \\
\end{aligned}
\right.
\end{equation}
Assume that $\|u\|_{L^2(\Omega)}=O(1),\|f\|_{L^2(\Omega)}=o(h).$
Taking inner product with $u$ and doing integration by part, we have
$ \|h\nabla u\|_{L^2(\Omega)}=O(1).
$
Since $q\in L^2(\Omega)/\mathbb{R}$, we may assume that $\int_{\Omega}qdx=0$. From the regularity theory of the steady Stokes system, (see \cite{bookTemam}, page 33) and  Poincar\'{e}'s inequality, we have
$$ \|h^2\nabla^2 u\|_{L^2(\Omega)}=O(1),\|q\|_{L^2(\Omega)}=O(h^{-1}), \|h\nabla q\|_{L^2(\Omega)}=O(1).
$$

The following is a direct consequence of trace theorem for  $q_0=q|_{\partial\Omega}$. 
\begin{lemma}
	$\|q_0\|_{H^{1/2}(\partial\Omega)}=O(h^{-1}).$
\end{lemma}
We also have the hidden regularity property for the normal derivative.
\begin{lemma}\label{hiddenregularitychapter1}
	$h\partial_{\mathbf{\nu}}u|_{\partial\Omega}=(h\partial_{\mathbf{\nu}}u_{\pa},0)$ and
	$\|h\partial_{\mathbf{\nu}}u|_{\partial\Omega}\|_{L^2(\partial\Omega)}=O(1)$.
\end{lemma}
The proof of this lemma will be given in appendix. We will recover some information for low frequencies from the following lemma:
\begin{lemma}\label{press.normchapter1}
	Suppose $u\rightharpoonup 0$ in $L^2(\Omega)$. Then
	after extracting to subsequences, we have $h\nabla q\rightharpoonup 0$ weakly in $L^2(\Omega)$ and $hq\rightarrow 0$ strongly in $H^{1/2}(\Omega)$.
\end{lemma}
\begin{proof}
	We may assume that $h\nabla q\rightharpoonup r$ weakly in $L^2(\Omega)$. Then the Rellich theorem implies that $hq\rightarrow P$ strongly in $L^2(\Omega)$, and thus $\nabla P=r$ with the property $\int_{\Omega} P=0$.  Moreover $\Delta P=0$ in $\Omega$, in the distributional sense. Since the sequence $(h^2\nabla^2 u)$ is bounded in $L^2$, then up to a subsequence, $h^2\nabla^2 u\rightharpoonup W$ weakly in $L^2$. From the Rellich theorem, the sequence $(h^2u)$ converges strongly in $L^2$ and the strong limit must be $0$, due to the fact that $u\rightharpoonup 0$, weakly in $L^2$. Thus $W=0$ and this implies that $\nabla P=0$. Finally, we must have $P=0$ since it has zero mean value. The last assertion follows from the Rellich theorem.
\end{proof}

\subsection{Semi-classical parametrix of the pressure term}
In system \eqref{Stokes-semichapter1}, the family of pressures $q$ satisfy the boundary value problem of Laplace equation
$$ -h^2\Delta q=0,\textrm{ in }\Omega,q|_{\partial\Omega}=q_{0}
$$
with unknown boundary data $q_0$. We denote by $\textrm{PI}(q_0)$ the Poisson integral of the corresponding harmonic function with trace $q_0$. Let $\mathcal{N}$ be the Dirichlet-Neumann operator satisfying 
$$ \mathcal{N}q_0=\partial_{\nu}\mathrm{PI}(q_0)|_{\partial\Omega}.
$$
Next we study the behaviour of the sequence of pressure $q$ in the regime of frequency scale $h^{-1}$. We always fix the notation
$$ \lambda(y,x',\xi')=|\xi'|_{g^{-1}(y,x')}\sim |\xi'|.
$$
Let $Y=(-\epsilon_0,\epsilon_0)_y\times X_x$ and $Y_+=[0,\epsilon_0)_y\times X_x$. We first have the $L^2$ bound of $q$, micro-locally away from $\xi'=0$.
\begin{lemma}\label{microlocaltracechapter1}
	Let $(f_h)_{0<h<1}$ be a $h$-dependent family of distributions such that  $\|f_h\|_{L^2(\mathbb{R}^{n})}=O(h^{-N})$ for some $N\in\mathbb{N}$. Assume that for any $\chi\in C_c^{\infty}(\mathbb{R}^{2n})$, vanishing near $\xi=0$, we have $\|\chi(x,hD_x)f_h\|_{H^{\frac{1}{2}}(\mathbb{R}^n)}=O(h^{-1})$. Then
	$$ \|\chi(x,hD_x)f_h\|_{L^2(\mathbb{R}^n)}=O(h^{-\frac{1}{2}}).
	$$
\end{lemma}
\begin{proof}
	Assume that $\{|\xi|\leq 2\delta_0\}\cap \textrm{supp}(\chi)=\emptyset$. Take $\Phi\in C_c^{\infty}(\mathbb{R}^n)$ such that
	$$ \Phi(\xi)=1,|\xi|\leq \delta_0,\quad  \Phi(\xi)=0,|\xi|>2\delta_0.
	$$ 
	We write
	$$ \chi(x,hD_x)f=\Phi(hD_x)\chi(x,hD_x)f+(1-\Phi(hD_x))\chi(x,hD_x)f.
	$$
	From the support property we have  $\Phi(hD_x)\chi(x,hD_x)f=O_{H^{\infty}}(h^{\infty})$. Thus $(1-\Phi(hD_x))\chi(x,hD_x)f=O_{H^{\frac{1}{2}}}(h^{-1})$. Let $b(\xi)=|\xi|^{1/2}(1-\Phi(\xi))$, and we have $b(hD_x)\chi(x,hD_x)f=O_{L^2}(h^{-\frac{1}{2}})$. Since $b(\xi)\neq 0$ on supp$(\chi)$, we have
	$$ \|\chi(x,hD_x)f\|_{L^2(\mathbb{R}^n)}\leq C\|b(hD_x)\chi(x,hD_x)f\|_{L^2(\mathbb{R}^n)}+Ch\|\chi(x,hD_x)f\|_{L^2(\mathbb{R}^n)}\leq Ch^{-\frac{1}{2}}.
	$$
\end{proof}

\begin{lemma}\label{pressure L^2 boundchapter1}
	Given $\delta_0>0$ and $\varphi,\widetilde{\varphi}\in C_c^{\infty}(Y_+)$. For any $\chi_{\delta_0}\in C_c^{\infty}(Y_+\times\mathbb{R}^{d-1})$ such that $\chi_{\delta_0}|\equiv 0$ if $\lambda(y,x',\xi')\leq 2\delta_0$, we have
	$$ \|\widetilde{\varphi}\mathrm{Op}_h(\chi_{\delta_0})(\varphi q)\|_{L^2(\mathbb{R}_+^{d})}+h^{1/2}\|\left(\widetilde{\varphi}\mathrm{Op}_h(\chi_{\delta_0})(\varphi q)\right)|_{y=0}\|_{L^2(\mathbb{R}^{d-1})}\leq C_{\delta_0,\varphi,\widetilde{\varphi}}.
	$$
\end{lemma}
\begin{proof}
	Write $D_j=\frac{1}{i}\frac{\partial}{\partial x_j'}$, we have $\|hD_j(\varphi q)\|_{L^2(\mathbb{R}_+^{d})}=O(1)$. Since $\frac{\xi_j'}{|\xi'|^2}\chi_{\delta_0}(y,x',\xi')\in S^0_{\partial},$ then for $\chi_j=\frac{\xi'_j}{|\xi'|^2}\chi_{\delta_0}$, we have
	$$ \widetilde{\varphi}\chi_{\delta_0}(y,x',hD_{x'})(\varphi q)=\sum_{j=1}^{d-1}\widetilde{\varphi}\chi_j(y,x',hD_{x'})hD_j(\varphi q)+O_{L^2(\mathbb{R}_+^{d})}(1)=O_{L^2(\mathbb{R}_+^{d})}(1),
	$$
	where the implicit bound in big $O$ depends on $\delta_0,\varphi,\widetilde{\varphi}$. For the boundary term, we observe that $\widetilde{\varphi}\mathrm{Op}_h(\chi_{\delta_0})(\varphi q)|_{y=0}=O_{H^{1/2}(\mathbb{R}^{d-1})}(h^{-1})$ from trace theorem. Thus from Lemma \ref{microlocaltracechapter1},  $\widetilde{\varphi}\mathrm{Op}_h(\chi_{\delta_0})(\varphi q)|_{y=0}=O_{L^{2}(\mathbb{R}^{d-1})}(h^{-1/2})$.
\end{proof}
We express semi-classical Laplace operator $h^2\Delta_g$ in the geodesic coordinates of turbulence neighborhood $Y$ by
$$ P_0=h^2\partial_y^2+g^{ij}\partial_{i}\partial_j+hM_{j}(y,x')h\partial_j+
hH(y,x')h\partial_y
$$
where $\partial_j=\partial_{x'_{j}}$. We make the ansatz
$$ \widetilde{q}(y,x'):=\frac{1}{(2\pi h)^{d-1}}\int a(y,h,x',\xi')e^{\frac{ix'\xi'}{h}}\theta(\xi')d\xi',
$$
then we calculate
\begin{equation*}
\begin{split}
P_0(\widetilde{q})(y,x',\xi')=&\frac{1}{(2\pi h)^{d-1}}\int \left(h^2\partial_y^2 a+g^{jk}(h^2\partial_j\partial_ka-g^{jk}\xi'_j\xi'_k a)\right)
e^{\frac{ix'\xi'}{h}}\theta(\xi')d\xi'\\
+&\frac{1}{(2\pi h)^{d-1}}
\int \left(ihg^{jk}\xi'_k\partial_ja\right)
e^{\frac{ix'\xi'}{h}}\theta(\xi')d\xi'\\
+&\frac{1}{(2\pi h)^{d-1}}\int\left((h^2M_j\partial_j a+ihM_j\xi'_ja)+h^2H\partial_ya\right)e^{\frac{ix'\xi'}{h}}\theta(\xi')d\xi'.
\end{split}
\end{equation*}
We next look for the formal semi-classical expansion
$$ a(y,h,x',\xi')\backsimeq \sum_{j\geq 0}h^j a_j(y,h,x',\xi')
$$
with $a_j\in S_{\partial}^{-j}$ and $h^k\partial_y^k a_j\in S_{\partial}^{-j+k}$.
We obtain
\begin{equation}
\begin{aligned}
P_0\widetilde{q}\backsimeq
\frac{1}{(2\pi h)^{d-1}}\int &((h^2\partial_y^2 a_0-g^{ij}\xi_i\xi_j a_0)\\
+&h(ig^{jk}\xi'_k\partial_ja_0+iM_j\xi'_ja_0+h^2H\partial_ya_0)\\
+&h(h^2\partial_y^2a_1-g^{jk}\xi'_j\xi'_ka_1)\\
+&h^2(g^{jk}\partial_j\partial_ka_0+M_j\partial_ja_0)\\
+&h^2(ig^{jk}\xi'_k\partial_ja_1+iM_j\xi'_ja_1+h^2H\partial_ya_1)\\
+&h^2(h^2\partial_y^2a_2-g^{jk}\xi'_j\xi'_ka_2)\\
+&\cdot\cdot\cdot)
e^{\frac{ix'\xi'}{h}}\theta(\xi')d\xi'.
\nonumber
\end{aligned}
\end{equation}
Pick  $\varphi_{1,0}=\varphi_1|_{\partial\Omega},\varphi_1\in C_c^{\infty}(Y)$. For $\widetilde{q_0}=\varphi_{1,0}q_0$, we put $$\theta(\xi')=\mathcal{F}_h(\widetilde{q_0}(\xi'))=(2\pi h)^{-(d-1)}\int_{\mathbb{R}^{d-1}}\widetilde{q_0}(x')e^{-ix'\xi'/h}dx',$$
$$a_0(0,\cdot)\equiv 1,\quad a_j(0,\cdot)\equiv 0,\forall j\geq 1,$$ 
and we define the functions $a_j$ inductively as follows: firstly we define  $a_0$
$$ a_0(y,x',\xi')=e^{-\frac{y\lambda(y,x',\xi')}{h}}, \lambda(y,x',\xi')=:\sqrt{g^{ij}\xi'_i\xi'_j}\sim |\xi'|, 
$$ and the quantity
$$ (h^2\partial_y^2-\lambda^2)a_0=h\left(\frac{h^2}{\lambda^2}\frac{y^2\lambda^2}{h^2}(\partial_y
\lambda)^2+\frac{2y\lambda}{h}\partial_y\lambda-2\partial_y\lambda\right)e^{-\frac{y\lambda}{h}}$$
can be viewed as of order $h$.
Next we set $a_j,j\geq 1$ implicitly by solving a sequence of linear ODEs:
\begin{equation*}
\begin{split}
h^2\partial_y^2 a_1-\lambda^2a_1=&-h^{-1}(h^2\partial_y^2-\lambda^2)a_0-(ig^{jk}\xi'_k\partial_ja_0+iM_j\xi'_ja_0+h^2H\partial_ya_0).\\
h^2\partial_y^2 a_{n}-g^{ij}\xi_i\xi_j a_{n}=&-(g^{ij}\partial_i\partial_ja_{n-2}+M_j\partial_ja_{n-2})\\ &-
(ig^{jk}\xi'_k\partial_ja_{n-1}+iM_j\xi'_ja_{n-1}+h^2H\partial_ya_{n-1}),n\geq 2.
\end{split}
\end{equation*}

Unfortunately, the functions $a_j$ constructed above are not symbols, since they have singularities when $\xi'=0$. This indicates that we can only obtain information of $q(h)$ from such parametrix away from $\xi'=0$. We modify the construction above by setting
$$ A_0(y,x',\xi')=e^{-\frac{y\lambda}{h}}\psi_{\delta_0}(\lambda)\varphi_2(y,x'), \quad (y,x',\xi')\in \mathbb{R}_+^{d}\times\mathbb{R}^{d-1},
$$
with $\psi_{\delta_0}=\psi(\delta_0^{-1}\cdot),\psi\in C^{\infty}(\mathbb{R}_+)$ satisfying $\psi(s)\equiv 1$ when $s\geq 1$ and $\psi(s)=0$ when $0<s\leq\frac{1}{2}$.  We next modify other $A_j$ in the same manner. Indeed, the ODEs which define $A_j$ are linear ODEs in $y$ variable. Thus for $j\geq 1$, $A_j(y,x'\xi')\equiv 0$ when $\lambda(y,x',\xi')\leq \frac{\delta_0}{2}$.  We define the particular class of symbols in $S_{\partial}^j$.
\begin{definition}
	$$ \mathcal{E}_{\partial}^j:=\left\{a\in S_{\partial}^j:|(h\partial_y)^l\partial^{\alpha}_{x',\xi'}a(y,x',\xi')|\leq C_{l,\alpha}e^{-\frac{C'_{l,\alpha}y}{h}} \right\}.
	$$
\end{definition}
\begin{lemma}\label{exponentialsymbolchapter1}
	The symbols constructed above can be chosen to satisfy $A_j\in \mathcal{E}_{\partial}^{-j}$ for all $j\in\mathbb{N}$.
\end{lemma}
The proof will be given in appendix.

In summary, we have $A\simeq \sum_{j\geq 0}h^jA_j$, and a tangential symbol $B_{\delta_0}(y,x',\xi')$ compactly supported in $\lambda(y,x',\xi')\leq \frac{\delta_0}{2}$, such that
\begin{equation*}
\begin{split}
\varphi P_0A(y,x',hD_{x'})(\varphi_{1,0}q_0)=&\varphi B_{\delta_0}(y,x',hD_{x'})(\varphi_{1,0}q_0)+O_{H^{\infty}}(h^{\infty}),\\
\varphi_0A(0,x',hD_{x'})(\varphi_{1,0}q_0)=&\varphi_0\textrm{Op}_h(\psi_{\delta}(\lambda))(\varphi_{1,0}q_0)+O_{H^{\infty}}(h^{\infty}).
\end{split}
\end{equation*} 
The following proposition states that the parametrix constructed above is an approximation of the pressure $q$ in the relevant semi-classical scale.
\begin{proposition}\label{parametrixpressurechapter1}
	There exists $A\in S_{\partial}^0$ with principal symbol
	$$ A_0(y,x',\xi')=\exp\left(-\frac{y\lambda(y,x',\xi')}{h}\right)\psi_{\delta_0}(\lambda(y,x',\xi'))\varphi_1(y,x'),
	$$	
	which satisfies asymptotic expansion $\displaystyle{A\sim \sum_{j\geq 0}h^jA_j, A_j\in \mathcal{E}_{\partial}^{-j}}$. Moreover, for any $$\varphi,\varphi_1\in C_c^{\infty}(Y_+), \varphi_1|_{\mathrm{supp }(\varphi )}\equiv 1,$$ we have $\varphi\mathrm{Op}_h(\chi_{\delta_0}A_j)(\varphi_{1,0}q_0)=O_{L^2(\mathbb{R}_+^{d})}(1)$ for all $j$, and
	\begin{equation*}
	\begin{split}
	&\varphi\mathrm{Op}_h(\chi_{\delta_0})(\varphi_1 q)=\varphi\mathrm{Op}_h(\chi_{\delta_0}A)(\varphi_{1,0}q_0)+O_{L^2(\mathbb{R}_+^d)}(h^{3/4}),\\
	&\varphi\mathrm{Op}_h(\chi_{\delta_0})h\partial_y(\varphi_1q)=\varphi\mathrm{Op}_h(\chi_{\delta_0}\lambda A)(\varphi_{1,0}q_0)+O_{L^2(\mathbb{R}_+^d)}(h^{3/4}),\\
	&\varphi\mathrm{Op}_h(\chi_{\delta_0})h\partial_y(\varphi_1q)=\varphi\mathrm{Op}_h(\chi_{\delta_0}\lambda A)(\varphi_{1,0}q_0)+O_{H^\frac{2}{3}(\mathbb{R}_+^d)}(h^{1/4}),
	\end{split}
	\end{equation*}
	where $\varphi_0=\varphi|_{\partial\Omega},\varphi_{1,0}=\varphi_1|_{\partial\Omega},\chi_{\delta_0,0}=\chi_{\delta_0}|_{y=0}$.
\end{proposition}
We postpone the proof of this proposition in appendix. A direct consequence is that the singularities of the family of pressures $(q_h)$ must concentrate in a very thin strip near the boundary.
\begin{lemma}\label{concentration1chapter1}
	With the same $\chi_{\delta_0}\in C_c^{\infty}(Y_+\times\mathbb{R}^{d-1})$ and  $\varphi_1,\varphi\in C_c^{\infty}(Y_+)$, for any $0<y_0< \epsilon_0$, we have
	$$ \int_{y_0}^{\epsilon_0}\|\varphi\mathrm{Op}_h(\chi_{\delta_0})(\varphi_1q)\|_{L^2(\mathbb{R}^{d-1})}^2dy\leq C_{\delta_0}(e^{-\frac{cy_0}{h}}+h),
	$$
	where the constant $C_{\delta_0}$ only depends on $\delta_0$ and is independent of $y_0$ and $h$.
\end{lemma}
\begin{proof}
	The second term appearing on the right hand side comes from all the possible remainder terms. It suffices to estimate the term $$\displaystyle{\int_{y_0}^{\epsilon_0}\|\varphi\mathrm{Op}_h(\chi_{\delta_0}A_0)(\varphi_1 q_0)\|_{L^2(\mathbb{R}^{d-1})}^2dy}.$$
	Since $\varphi_{1,0}q_0=O_{L^2(\mathbb{R}^{d-1})}(h^{-1/2})$, micro-locally, we have for each fixed $y>0$ that
	\begin{equation*}
	\begin{split}
	\|\varphi\mathrm{Op}_h(\chi_{\delta_0}A_0)(\varphi_{1,0} q_0)\|_{L^2(\mathbb{R}^{d-1})}\leq &Ch^{-1/2}\sum_{|\beta|\leq Cd}h^{\frac{|\beta|}{2}}\sup_{y>0,(x',\xi')}|\partial_{x',\xi'}^{\beta}(\chi_{\delta_0}A_0)|+O(h^{\infty})\\
	\leq &Ch^{-1/2}e^{-\frac{cy}{h}}\Big(1+\sum_{1\leq m,n\leq Cd}h^{m/2}\big(\frac{y}{h}\big)^{n}\Big)+O(h^{\infty}).
	\end{split}
	\end{equation*}
	Squaring and Integrating the right hand side in $y$ variable yields the desired conclusion.
\end{proof}

\section{Main Steps of the Proof}
The proof of Theorem \ref{Stchapter1} can be divided into several steps according to different geometric situations. We want to show that for any given point $\rho_0\in ^bT^*\ov{\Omega}$, if $\rho_0\notin $ supp $\mu$, then $\gamma(s,\rho_0)\notin $ $\mathrm{supp}(\mu)$ for any $s>0$. 
The first step is to show that if $\rho_0\in T^*\Omega$, $\rho_0\notin $ $\mathrm{supp}(\mu)$,  then $\gamma(s,\rho_0)\notin $  $\mathrm{supp}(\mu)$ for all $s>0$ provided that $\pi(\gamma(\cdot,\rho_0)|_{[0,s]})\cap \partial\Omega=\emptyset$. This can be summarized by the following proposition, in which we have stronger conclusion that the measure is also invariant under the flow.
\begin{proposition}\label{interior propagationchapter1}
	For any real-valued scalar function $a\in C_c^{\infty}(\Omega\times\mathbb{R}^d)$ vanishing near $\xi=0$, we have
	$$ \frac{d}{ds}\langle\mu,a\circ \gamma(s,\cdot)\rangle=0.
	$$
\end{proposition}
%We prove this proposition in the present section since it is rather easy.
\begin{proof}
	Let $A=\textrm{Op}_h(a)$ and $P=-h^2\Delta-1$. Applying the equation and Lemma \ref{press.normchapter1}, we have  
	\begin{equation}
	\begin{split}
	\frac{i}{h}\left([P,A]u|u \right)_{\Omega}&= \frac{i}{h}(Au|Pu)_{\Omega}-\frac{i}{h}(APu|u)_{\Omega}\\
	&=\frac{i}{h}(Au|f-h\nabla q)_{\Omega}-\frac{i}{h}(A(f-h\nabla q)|u)_{\Omega}\\
	&=-\frac{i}{h}(Au|h\nabla q)_{\Omega}+\frac{i}{h}(Ah\nabla q|u)_{\Omega}+o(1)\\
	&=-\frac{i}{h}([A,h\textrm{div }]u|q)_{\Omega}+\frac{i}{h}([A,h\nabla]q|u)_{\Omega}+o(1)\\
	&=i(\textrm{Op}_h(\nabla a)\cdot u|q)_{\Omega}-i(\textrm{Op}_h(\nabla a)q|u)_{\Omega}+o(1)\\
	&=i(u|\textrm{Op}_h(\nabla \ov{a})q)_{\Omega}-i(\textrm{Op}_h(\nabla a)q|u)_{\Omega}+o(1).
	\end{split}
	\end{equation}
	where we have used integration by part freely, thanks to the fact that $A$ has compact support in $x\in\Omega$. 
	Now we claim that for any $\chi=a$ or $\ov{a}$, vanishing near $\xi=0$, we have
	$$ (u|\textrm{Op}_h(\nabla \chi)q)_{\Omega}=o(1).
	$$
Indeed, $q=O_{L^2(\Omega)}(1)$ micro-locally away from $\xi=0$ since $h\nabla q=O_{L^2}(1)$. Moreover, $h^2\Delta(\textrm{Op}_h(\nabla \chi)q)=O_{L^2(\Omega)}(h)$ and this implies that $\textrm{Op}_h(\nabla \chi)q=o_{L^2}(1)$ since the symbol of $h^2\Delta\textrm{Op}_h(\nabla \chi)$ vanishes near $\xi=0$ as well as $x$ near the boundary. 	In view of the definition of $\mu$, this completes the proof.
\end{proof}
For the second step, we need prove that if $\rho_0\in\mathcal{E}$, then $\mu=0$ in a neighborhood of $\rho_0$.

\begin{proposition}\label{ellipticchapter1}
	$\mu\mathbf{1}_{\mathcal{E}}=0$.  Furthermore, for $\nu$, the semi-classical defect measure of the sequence $(h_k\partial_{\mathbf{\nu}}u_k|_{\partial\Omega},h_k)$, we have $\nu\mathbf{1}_{\mathcal{E}}=0$. 
\end{proposition}

The third step consists of proving that after reflection near a hyperbolic point, the measure $\mu$ is still zero. 
\begin{proposition}\label{hyperbolicchapter1}
	Suppose $\rho_0\notin \mathrm{supp}(\mu)$ and there exists $s_0>0$ such that $\gamma(s_0,\rho_0)\in\mathcal{H}$ and $\pi(\gamma(s,\rho_0))\in \Omega$ for all $0\leq s<s_0$. Then there exists $\delta>0$ such that $$\pi(\gamma(\cdot,\rho_0)|_{[s_0,s_0+\delta]})\cap \mathrm{supp }(\mu)=\emptyset.$$
\end{proposition}

Next step is to prove the propagation near a diffractive point.
\begin{proposition}\label{diffractivechapter1}
	Suppose $\rho_0\notin\mathrm{supp }(\mu)$ and there exists $s_0>0$ such that $\gamma(s_0,\rho_0)\in\mathcal{G}^{2,+}$ and $\pi(\gamma(s,\rho_0))\in\Omega$ for all $0\leq s<s_0$. Then $\gamma(s_0,\rho_0)\notin\mathrm{supp }(\mu).$
\end{proposition}

To deal with higher order contact, we will use induction. First let us introduce
\begin{definition}[$k$-propagation property]
	For $k\geq 2$, we say that $k$-propagation property holds, if along generalized ray $\gamma(s,\rho_0)$,  the following statement is true: 
	For some $\sigma_0>0$, if $\gamma(\cdot,\rho_0)|_{[0,\sigma_0)}\cap \mathrm{supp }(\mu)=\emptyset$  (or $\gamma(\cdot,\rho_0)|_{(-\sigma_0,0]}\cap \mathrm{supp }(\mu)=\emptyset$) and $\displaystyle{\gamma(\sigma_0,\rho_0)\in \bigcup_{2\leq j\leq k}\mathcal{G}^j} $ (or $\displaystyle{\gamma(-\sigma_0,\rho_0)\in \bigcup_{2\leq j\leq k}\mathcal{G}^j} $), then $\gamma(\sigma_0,\rho_0)\notin \mathrm{supp }(\mu)$ (or $\gamma(-\sigma_0,\rho_0)\notin \mathrm{supp }(\mu)$ ).
\end{definition}

The last step for the proof of Theorem \ref{Stchapter1} can be reduced to 
\begin{proposition}\label{highorderchapter1}
	$k$-propagation property holds for all $k\geq 2$.
\end{proposition}
%%%%%%%%%%%%%%%%%%%%%%%%%%%%%%%%%%%%%%%%%%%
\section{Near $\mathcal{E}$}
This section is devoted to the proof of Proposition \ref{ellipticchapter1}. We set
$ Q(y,x',\xi'):=\sqrt{\lambda^2-1}$ and define $Q_h=\varphi \mathrm{Op}_h(Q\chi_0)\varphi_1$ with $\chi_0\in C_c^{\infty}(\mathbb{R}_{\xi'}^{d-1})$ with support near $\mathcal{E}$ in which $1+\delta<\lambda<C$. With a bit abuse of notation, we refer $q_0,q$ to be $\varphi\mathrm{Op}_h(\chi_0)\varphi_1q_0,\varphi\mathrm{Op}_h(\chi_0)\varphi_1q$ and $u$ to be $\varphi\mathrm{Op}_h(\chi_0)\varphi_1 u$. In this manner,  we can combine the parametrix in last section to write the system \eqref{WStokesmainchapter1} as
\begin{equation}
\left\{
\begin{aligned}
& (-h^2\partial_y^2+Q_h^2)u_{\pa,j}+g^{jk}h\partial_{x'_k}(\textrm{Op}_h(A_0)q_0)=R_{\pa,j}=O_{L^2(\mathbb{R}_+^d)}(h),\\
& (-h^2\partial_y^2+Q_h^2)u_{\pe}+h\partial_y(\textrm{Op}_h(A_0)q_0)=R_{\pe}=O_{L^2(\mathbb{R}_+^d)}(h).
\end{aligned}
\right.
\end{equation}
Note that the symbol $A_0(y,x',\xi')$ is defined in last section which equals to  $e^{-\frac{y\lambda}{h}}$  since $\lambda>1$. 
Take $\psi\in C^{\infty}(\mathbb{R}_+),$ with $\psi|_{[0,\epsilon_0]}\equiv 1, \psi_{[2\epsilon_0,\infty)}\equiv 0$. Denoting the extended distributions of $u$ by $w=(w_{\pa},w_{\pe})=(u_{\pa},u_{\pe})\psi(y)\mathbf{1}_{y\geq 0}$, we have from standard elliptic parametrix construction (see Appendix A) that modulo $O_{H^{\infty}(\mathbb{R}_+^d)}(h^{\infty})$,
\begin{equation}
\left\{
\begin{aligned}
& w_{\pa,j}=E(y,x',hD_y,hD_{x'})(-\psi(y)g^{jk}h\partial_{x'_k}(\textrm{Op}_h(A_0)q_0)+ hv_j\otimes \delta_{y=0}
+\psi(y)R_{\pa,j}),\\
& w_{\pe}=E(y,x',hD_y,hD_{x'})(-h\psi(y)\partial_y(\textrm{Op}_h(A_0)q_0)+\psi(y)R_{\pe}).
\end{aligned}
\right.
\end{equation}
where $v=h\partial_y u_{\pa}|_{y=0}=O_{L_{x'}^2}(1)$. Recall that the principal symbol of $E$ is given by
$$ E^0:=\frac{\chi_0(\xi')\varphi(y,x')}{\eta^2+\lambda(y,x',\xi')^2-1},
$$
Now we need a lemma which deals with the trace of error terms:
\begin{lemma}\label{errortracechapter1}
	Assume that $R=\varphi\mathrm{Op}_h(\chi_0)\varphi_1R+O_{H^{\infty}}(h^{\infty})$, then if  $\|\psi(y)R\|_{L^2(\mathbb{R}_+^d)}=O(h), $ we have $$\|E(y,x',hD_y,hD_{x'})(\psi(y)R)|_{y=0}\|_{L^2(\mathbb{R}_{x'}^{d-1})}=O(h^{1/3}).$$ 
\end{lemma}
\begin{proof}
	From the parametrix construction above, we know that
	$$ |\partial_{y,x',\eta,\xi'}^{\alpha}E(y,x';\eta,\xi')|\leq \frac{C_{\alpha}}{\eta^2+Q(y,x',\xi')^2}.
	$$	
	Therefore, the symbols $\eta E(y,x';\eta,\xi')$ and $ \lambda(y,x',\xi')E(y,x';\eta,\xi')$ are uniformly bounded in $S^0$. Thus $E(y,x';hD_y,hD_{x'})(\psi R)=O_{L^2(\mathbb{R}_+^d)}(h)=O_{H^1(\mathbb{R}_+^d)}(1)$, and from interpolation, we have $E(y,x';hD_y,hD_{x'})(\psi R)=O_{H^{2/3}(\mathbb{R}_+^d)}(h^{1/3})$. The conclusion then follows from trace theorem that $H^s(\mathbb{R}_+^d)\rightarrow H^{s-1/2}(\mathbb{R}^{d-1})$ is bounded for $s>1/2$. 
\end{proof}
\begin{proof}[Proof of Proposition \ref{ellipticchapter1}]
	Denote by $\mathcal{F}_h(q_0)=\theta$ the semi-classical Fourier transform of $q_0$, we calculate
	\begin{equation}\label{elliptic1chapter1}
	\begin{split}
	&E(y,x',hD_y,hD_{x'})(\psi(y)h\partial_y\mathrm{Op}_h(A_0)q_0)\\
	&=\frac{1}{(2\pi h)^d}\iint e^{\frac{i(y-z)\eta}{h}}dzd\eta\int
	\frac{e^{\frac{ix'\xi'}{h}}\theta(\xi')\varphi(y,x')\chi_0(\xi')}{\eta^2+Q^2(y,x',\xi')}\psi(z)h\partial_z(
	e^{-\frac{z\lambda(z,x',\xi')}{h}})d\xi'+R_1 \\
	&=-\frac{h}{(2\pi h)^d}\iint \frac{\langle\xi'\rangle\theta(\xi')e^{\frac{i(y\eta+x'\xi')}{h}}B_1(\eta,x',\xi')\varphi(y,x')\chi_0(\xi')}{\eta^2+Q^2(y,x',\xi')}d\eta d\xi' \\
	&-\frac{h^2}{(2\pi h)^{d}}\int e^{\frac{ix'\xi'}{h}}\theta(\xi')d\xi'\int\frac{e^{\frac{iy\eta}{h}}B_0(\eta,x',\xi')\varphi(y,x')\chi_0(\xi')}{\eta^2+Q(y,x',\xi')^2}d\eta
	+R_1,
	\end{split}
	\end{equation}
	with reminder term $R_1=O_{L^2(\mathbb{R}^d)}(h)$, where $\lambda_0=\lambda|_{y=0}$,
	$$ B_1(\eta,x',\xi')=\int_0^{\infty}\psi(z)e^{-\left(\frac{i\eta+\lambda(z,x',\xi')}{h}\right)z}\frac{\lambda(z,x',\xi')}{\langle\xi'\rangle}
	\frac{1}{h}dz,
	$$
	and
	$$ B_0(\eta,x',\xi')=\int_0^{\infty}\psi(hz)z(\partial_z\lambda)(hz,x',\xi')e^{-(i\eta+\lambda(hz,x',\xi'))z}dz.
	$$
	We notice that
	$$ K_0(y,x',\xi'):=\int\frac{e^{\frac{iy\eta}{h}}B_0(\eta,x',\xi')\varphi(y,x')\chi_0(\xi')}{\eta^2+Q(y,x',\xi')^2}d\eta
	$$
	is a bounded symbol in $S_{\xi'}^{0}$.  Thus the second term on the right hand side of \eqref{elliptic1chapter1} is equal to $R_2=O_{C^0(\mathbb{R}_y;L^2(\mathbb{R}^{d-1}_{x'}))}(h)$ and we may concentrate on the first term.
		Write
	$$ B_1(\eta,x',\xi')=\int_0^{\infty}\psi(hz)e^{-(i\eta+\lambda(hz,x',\xi'))z}\frac{\lambda(hz,x',\xi')}{\langle\xi'\rangle}dz. 
	$$
	Taylor expansion gives
	\begin{equation*}
	\begin{split}
	e^{-\lambda(hz,x',\xi')z}\lambda(hz,x',\xi')\psi(hz)=&e^{-\lambda_0(x',\xi')z}\lambda_0(x',\xi')+\int_0^1\frac{d}{dt}\left(e^{-\lambda(htz,x',\xi')z}\lambda(htz,x',\xi')\psi(htz)\right)dt\\
	=&e^{-\lambda_0(x',\xi')z}\lambda_0(x',\xi')+h\int_0^1P_t(z,x',\xi')e^{-\lambda(htz,x',\xi')z}dt
	\end{split}
	\end{equation*}
	with
	$$ P_t(z,x',\xi')=-z^2(\lambda\partial_y\lambda)(htz,x',\xi')\psi(htz)+z(\partial_y\lambda)(htz,x',\xi')\psi(htz)+z\lambda(htz,x',\xi')\psi'(htz).
	$$
	Thus we have $\displaystyle{B_1(\eta,x',\xi')=\frac{\lambda_0(x',\xi')}{(i\eta+\lambda_0(x',\xi'))\langle\xi'\rangle}+h\widetilde{B}_1(\eta,x',\xi')}$, where $$\widetilde{B}_1(\eta,x',\xi')=\int_0^1\int_0^{\infty}e^{-(i\eta+\lambda(htz,x',\xi'))z}\frac{1}{\langle\xi'\rangle}P_t(z,x',\xi')dzdt.$$ 
	Note that near a point in $\mathcal{E}$, $|\partial^{\alpha}_{x'}\partial^{\beta}_{\xi'}P_t(z,x',\xi')|\leq C_{\alpha\beta}z^2$, independent of $t,h$, hence the symbol
	$$ \widetilde{K}_1(y,x',\xi')=\int\frac{e^{\frac{iy\eta}{h}}\varphi(y,x')\chi_0(\xi')}{\eta^2+Q(y,x',\xi')^2}\widetilde{B}_1(\eta,x',\xi')d\eta \in S_{\xi'}^0.
	$$
	Therefore, the symbol in the principal term of $E(y,x',hD_y,hD_{x'})(\psi(y)h\partial_y\mathrm{Op}_h(A_0)q_0)$ equals to 
	\begin{equation*}
	\begin{split}
	K_1(y,x',\xi')=&\lambda_0(x',\xi')\int\frac{e^{\frac{iy\eta}{h}}\varphi(y,x')\chi_0(\xi')}{(\eta^2+Q(y,x',\xi')^2)(i\eta+\lambda_0(x',\xi'))}d\eta+h\widetilde{K}_1(y,x',\xi')\\
	=&2\pi\lambda_0\varphi(y,x')\chi_0(\xi')\left(\frac{e^{-\frac{yQ}{h}}}{2(\lambda_0-Q)Q}-\frac{e^{-\frac{y\lambda_0}{h}}}{\lambda_0^2-Q^2}\right)+h\widetilde{K}_1(y,x',\xi')\\
	=&2\pi\lambda_0\varphi(y,x')\chi_0(\xi')\left(\frac{e^{-\frac{yQ}{h}}-e^{-\frac{y\lambda_0}{h}}}{2(\lambda_0-Q)Q}+\frac{e^{-\frac{y\lambda_0}{h}}}{2Q(\lambda_0+Q)}\right)+h\widetilde{K}_1(y,x',\xi').
	\end{split}
	\end{equation*}
	Note that
	$$ E_1(y,x',\xi')=2\pi\lambda_0\varphi(y,x')\chi_0(\xi')\left(\frac{e^{-\frac{yQ}{h}}-e^{-\frac{y\lambda_0}{h}}}{2(\lambda_0-Q)Q}+\frac{e^{-\frac{y\lambda_0}{h}}}{2Q(\lambda_0+Q)}\right)>0
	$$
	near $\mathcal{E}$, we have 
	$$ E(y,x',hD_y,hD_{x'})(\psi(y)h\partial_y\mathrm{Op}_h(A_0)q_0)=E_1(y,x',hD_{x'})q_0+R_1+R_2,
	$$
	with $R_2=O_{C(\mathbb{R}_y;L^2(\mathbb{R}_{x'}^{d-1}))}(h).$
	We claim that $R_1=O_{H^1(\mathbb{R}^d)}(1)$. Once it is justified, by interpolation, we have $\|R_1\|_{H^{2/3}(\mathbb{R}^d)}=O(h^{1/3})$. To verify the claim, we note that the symbol of the reminder term $R_1$ is of the  form $hS^{-1}$ (in both $\eta$ and $\xi'$ variables), hence the symbolic calculus yields $\partial_yR_1=O_{L^2(\mathbb{R}^d)}(1), $ and $ \partial_{x'}R_1=O_{L^2(\mathbb{R}^d)}(1)$.
	
	We next calculate the parallel component
	\begin{equation}
	\begin{split}
	&\frac{1}{(2\pi h)^d}\iint\psi(z)e^{\frac{i(y-z)\eta}{h}}dzd\eta\int \frac{e^{\frac{ix'\xi'}{h}-\frac{z\lambda(z,x',\xi')}{h}}\theta(\xi')g^{jk}(z,x')\xi'_k\varphi(y,x')\chi_0(\xi')}{\eta^2+Q(y,x',\xi')^2}d\xi'\\
	&=\frac{1}{(2\pi h)^d}\iint \frac{e^{\frac{i(x'\xi'+y\eta)}{h}}\xi'_k\theta(\xi')\varphi(y,x')\chi_0(\xi')}{\eta^2+Q(y,x',\xi')^2}d\eta d\xi'\int_0^{\infty}\psi(z)e^{-\frac{i\eta+\lambda(z,x',\xi')}{h}z}g^{jk}(z,x')dz\\
	&=\frac{h}{(2\pi h)^d}\iint \frac{e^{\frac{i(x'\xi'+y\eta)}{h}}\xi'_k\theta(\xi')\varphi(y,x')\chi_0(\xi')}{(\eta^2+Q(y,x',\xi')^2)}B_{2,jk}(\eta,x',\xi')d\eta d\xi'\\
	&=:E_2(y,x',hD_{x'})q_0.
	\nonumber
	\end{split}
	\end{equation}
	where
	$$ B_{2,jk}(\eta,x',\xi')=\int_0^{\infty}\psi(z)e^{-\frac{i\eta+\lambda(z,x',\xi')}{h}z}g^{jk}(z,x')
	\frac{1}{h}dz.
	$$
	Define
	$$ K_{2,jk}(y,x',\xi')=\int\frac{e^{\frac{iy\eta}{h}}B_{2,jk}(\eta,x',\xi')\varphi(y,x')\chi_0(\xi')}{\eta^2+Q(y,x',\xi')^2}d\eta,
	$$
	and from similar argument we have
	$$ K_{2,k}(y,x',\xi')=g^{jk}(0,x')\int\frac{e^{\frac{iy\eta}{h}}\varphi(y,x')\chi_0(\xi')}{(\eta^2+Q(y,x',\xi')^2)(i\eta+\lambda_0(x',\xi'))}+h\widetilde{K}_2(y,x',\xi')
	$$ 
	and the principal symbol of $E_2(y,x',hD_{x'})$ is elliptic if $\lambda_0(\xi')>1 $ and $y$ small enough. Finally,
	\begin{equation}
	\begin{split}
	&E(y,x',hD_y,hD_{x'})(hv\otimes\delta_{y=0})\\=&
	\frac{h}{(2\pi h)^d}\iint\frac{\mathcal{F}_h(v)(\xi')e^{\frac{i(y\eta+x'\xi')}{h}}\varphi(y,x')\chi_0(\xi')}{\eta^2+Q(y,x',\xi')^2}d\xi'd\eta+O_{L^2(\mathbb{R}^d)}(h)\\
	=&\frac{h}{(2\pi h)^d}\int\mathcal{F}_h(v)(\xi')e^{\frac{ix'\xi'}{h}}\frac{\pi e^{-\frac{yQ(y,x',\xi')}{h}}\varphi(y,x')\chi_0(\xi')}{Q(y,x',\xi')}d\xi'+O_{L^2(\mathbb{R}^d)}(h)\\
	=&:E_3(y,x',hD_{x'})v+O_{L^2(\mathbb{R}^d)}(h),
	\end{split}
	\end{equation}
	and again, $E_3(y,x',hD_{x'})$ is elliptic near $\lambda_0(\xi')>1$. Moreover, we deduce from the same argument as for $R_1$ that the reminder terms are indeed of   
	$O_{H^{2/3}(\mathbb{R}^d)}(h^{1/3})$. Now the boundary condition $(w_{\pe},w_{\pa})|_{y=0}=0$ and trace theorem yields 
	$$ E_1(0,x',hD_{x'})q_0=O_{L^2(\mathbb{R}_{x'}^{d-1})}(h^{1/3}),$$
	$$ E_2(0,x',hD_{x'})q_0+E_3(0,x',hD_{x'})v=O_{L^2(\mathbb{R}_{x'}^{d-1})}(h^{1/3}).
	$$
	Therefore, from the ellipticity of $E_1,E_2,E_3$, the measure of pressure at the elliptic region vanishes, so does the measure of $v$, namely $\nu|_{\mathcal{E}}=0$.
	The proof of Proposition \ref{ellipticchapter1} is complete.
	
\end{proof}

%%%%%%%%%%%%%%%%%%%%%%%%%%%%%%%%%%%%%%%%%%%%%
\section{Near $\mathcal{H}$}
We take $\varphi_1,\varphi\in C_c^{\infty}(Y_+)$ such that $\varphi_1|_{\textrm{supp}(\varphi)}\equiv 1$. For any tangential symbol $b\in C_c^{\infty}(Y_+\times\mathbb{R}^{d-1})$, we define the pseudo-differential operator $B_h=\varphi\mathrm{Op}_h(b)\varphi_1$, with compact support in $Y_+$. We will change the notation of tangential variables $(x',\xi')$ to $(x,\xi)$. We always work in local coordinate $(y,x)$ and sometimes abuse the notation $u=\varphi_1u,q=\varphi_1 q$ as compactly supported functions in $Y_+$. Note that $q_0$, the trace of $q$ is not bounded in $L^2$ in priori. Fortunately, it turns out that $q_0=O_{L^2}(1)$, micro-locally near a point in $ \mathcal{H}$. 
\subsection{$L^2$ bound of boundary datums}
Take $b(y,x,\xi),b_1(y,x,\xi)\in C_c^{\infty}([0,\epsilon_0)\times \mathcal{H})$, such that $b_1|_{[0,\epsilon_0/2)\times supp(b)}\equiv 1$. Let $Q(y,x,\xi)=\sqrt{1-\lambda(y,x,\xi)^2}b_1(y,x,\xi)$. We will first factorize the operator $(-h^2\Delta-1)$ near a hyperbolic point.
\begin{lemma}\label{hyper-factorzichapter1}
	For $0\leq y<\epsilon_0$, we have
	$$ B_h(-h^2\Delta-1)=-(hD_y-Q_h^+)(hD_y-Q_h^-)+R'=-(hD_y-Q_h^-)(hD_y-Q_h^+)+R'',
	$$
	where $R',R''\in C^{\infty}([0,\epsilon_0],h^{\infty}\Psi^{-\infty}(\partial\Omega))$, and
	$Q_h^{\pm}$ have principal symbol $\pm Q(y,x,\xi)$.
\end{lemma}
\begin{proof}
	The proof is quite standard, and we follow the construction in \cite{Burq-Lebeau2001} by translating word by word to the semi-classical setting. In local coordinate, we have
	$$ B_h(-h^2\Delta-1)=h^2D_y^2+R(y,x,hD_{x})+hM_1(y,x')hD_y+hM_0(y,x)hD_x
	$$
	with $\sigma(R)=Q^2$. Set $q_1^{+}=\sqrt{Q}(y,x',\xi')$, $Q_1^+=\textrm{Op}_h(q_1^+)$ and $Q_1^-=-Q_1^+-hM_1$.
	Direct calculation gives
	\begin{equation*}
	\begin{split}
	(hD_y-Q_1^{+})(hD_y-Q_1^-)&=h^2D_y^2-(Q_1^+)^2-hQ_1^+M_1-(Q_1^++Q_1^-)hD_y-\frac{h}{i}\partial_y(Q_1^-)\\
	&=h^2D_y^2-(Q_1^+)^2+hM_1hD_y-h(Q_1^+M_1-i\partial_y(Q_1^-)).
	\end{split}
	\end{equation*}
	Thus $B_h(-h^2\Delta-1)-(hD_y-Q_1^+)(hD_y-Q_1^-)=hT_1,
	$
	with some operator $T_1$, bounded in $L^2$.
	Now for $j\geq 1$, suppose that we have
	$$ B_h(-h^2\Delta-1)-(hD_y-Q_j^+)(hD_y-Q_j^-)=h^jT_j,
	$$
	by setting $Q_{j+1}^{\pm}:=Q_j^{\pm}+h^jS_{j+1}^{\pm}$ with 
	$S_{j+1}^++S_{j+1}^-=0$ and $\displaystyle{ \sigma(S_{j+1}^+)=\frac{\sigma(T_j)}{2\sigma(Q_j^+)}},
	$ we obtain that
	\begin{equation*}
	\begin{split}
	&B_h(-h^2\Delta-1)-(hD_y-Q_{j+1}^+)(hD_y-Q_{j+1}^-)\\=&
	h^jT_j+h^j(S_{j+1}^+Q_j^-+Q_j^+S_{j+1}^-)-h^j(S_{j+1}^++S_{j+1}^-)hD_y
	-\frac{h^{j+1}}{i}\partial_y(S_{j+1}^-)+h^{2j}S_{j+1}^+S_{j+1}^-\\
	=:&h^{j+1}T_{j+1},
	\end{split}
	\end{equation*}
	for some operator $T_{j+1}$ bounded in $L^2$. Thus the proof can be completed by induction. 
\end{proof}
Define $w=\varphi_1u-h\nabla (\varphi_1q)$, $w^{\pm}=B_h(hD_y-Q_h^{\pm})w$ and its boundary values $w_{0}^{\pm}:=w^{\pm}|_{y=0}$. Note that $\varphi P_hw=\varphi f$. 

\begin{proposition}\label{boundchapter1}
	$\|B_hh\partial_y w_{\pe}\|_{L^2(\mathbb{R}_+^d)}=O(1),$ and consequently,
	$\|w_{\pe}^{\pm}\|_{L^2(\mathbb{R}_+^d)}=O(1).$
\end{proposition}
\begin{proof}
	From $h\textrm{div }u=0$, we have $\varphi h\textrm{div }w=0$, hence
	$$ \varphi (h\partial_y w_{\pe}+h\textrm{div}_{\pa}w_{\pa})=O_{L^2(\mathbb{R}_+^d)}(h),
	$$
	where in local coordinates,
	$$ \textrm{div}_{\pa}w_{\pa}=\frac{1}{\sqrt{\det(g)}}\sum_{j=1}^{d-1}\partial_{x_j}
	(\sqrt{\det(g)}w_{\pa,j}).
	$$
	Therefore,
	$$\|B_hh\partial_y w_{\pe}\|_{L^2(\mathbb{R}_+^d)}\leq O(h)+\|B_hh\textrm{div}_{\pa} w_{\pa}\|_{L^2(\mathbb{R}_+^d)}=O(1).
	$$
\end{proof}
Now we recall the following hyperbolic energy estimate.
\begin{lemma}\label{hyper-estimateschapter1}
	Suppose $A_{h}=\mathrm{Op}_{h}(a)$ is ellptic (with real-valued symbol $a$ smoothly depending on $t$) of order $0$ on a compact manifold $M$ and $w$ are solutions of the $h$-dependent equations
	$$ (hD_t\pm A_{h})w=g,\quad (t,x)\in \mathbb{R}\times M.$$
	Assume that for any compact time interval $I$ and small $h$,$$\|w\|_{L^2(I\times M )}\leq C(I),\quad \|g\|_{L^2(I\times M )}\leq C(I)h,$$ then we have for all small $h$,
	$$ \sup_{t\in I'}\|w(t)\|_{L^2(M)}\leq C(I'),\;\forall I'\subset I\textrm{compact}.
	$$
\end{lemma}
\begin{proof}
	By symmetry, it suffices to treat the case $hD_t-A_h$. Take $\chi(t)\in C_c^{\infty}(I')$, and we may assume that $0\in I'$ with $\chi(0)=1$. Multiplying by $\chi(t)$ to the equation, we have
	$$ (hD_t- A_h)(\chi w)=\chi g+[\chi,hD_t-A_h]w=:r=O_{L^2(\mathbb{R}\times M)}(h).
	$$
	We now calculate
	\begin{equation}
	\begin{split}
	h\frac{d}{dt}(\chi w|\chi w)(t)_{L^2(M)}&=(ihD_t\chi w|\chi w)_{L^2(M)}+(\chi w|ihD_t\chi w)_{L^2(M)}\\
	&=i(A_h(\chi w)+r|\chi w)_{L^2(M)}-i(\chi w|A_h(\chi w)+r)_{L^2(M)}\\
	&=i((A_h-A_h^*)\chi w|\chi w)_{L^2(M)}+i( r|\chi w)_{L^2(M)}-i(\chi w| r)_{L^2(M)}	\nonumber
	\end{split}
	\end{equation}
	Integrating the formula above from $0$ to $\sup I'$, we finally have
	$ \|w(0)\|_{L^2(M)}^2=O(1).
	$
\end{proof}

\begin{lemma}
	$\|w_{0}^{\pm}\|_{L^2(\mathbb{R}^{d-1})}=O(1)$.
\end{lemma}
\begin{proof}
	From Proposition \ref{boundchapter1}, we have $(hD_y-Q_h^{\mp})w_{\pe}^{\pm}=O_{L^2(\mathbb{R}_+^d)}(h)$.  Applying the previous lemma to $w_{\pe}^{\pm}$, we have$\|w_{0,\pe}^{\pm}\|_{L^2(\mathbb{R}_x^{d-1})}=O(1)$. Combining the boundary condition, we have
	$$ B_h(Q_h^+-Q_h^-)(h\partial_yq)|_{y=0}=-B_h(Q_h^+-Q_h^-)h\mathcal{N}q_0=w_{0,\pe}^+-w_{0,\pe}^-=O_{L^2(\mathbb{R}^{d-1})}(1).
	$$
	Remark that in priori, $\mathcal{N}$ is a classical first order pseudo-differential operator, and we only have
	$$\|B_hh\mathcal{N}q_0\|_{L^2(\mathbb{R}^{d-1})}\leq \|B_h\|_{H^{-1}\rightarrow L^2}h\|\mathcal{N}q_0\|_{H^{-1}(\mathbb{R}^{d-1})}=O(h^{-1}).	$$
	From the exact pricipal symbol of $Q_h^{\pm}$, we have $\|B_hh\mathcal{N}q_0\|_{L^2(\mathbb{R}^{d-1})}=O(1)$, and the constant in big $O$ depends on the micro-local cut-off $b(y,x',\xi')$. As a consequence, $\|w_{0,\pe}^{\pm}\|_{L^2(\mathbb{R}^{d-1})}=O(1)$.
	
	It remains to study $w_{\pa}^{\pm}$. Notice that their boundary values are
	$$ w_{0,\pa}^{\pm}=B_h(v-(hD_yh\nabla_{\pa}q)|_{y=0})-B_hQ_h^{\pm}h\nabla_{\pa}q_0,
	$$
	where $v=(h\partial_y u)|_{y=0}=O_{L^2(\mathbb{R}^{d-1})}(1)$. All terms are obviously bounded in $L^2(\mathbb{R}^{d-1})$ except the trace of  $B_hh\nabla_{\pa}hD_yq$. To bound it, we use the support property of $b$ and Proposition \ref{parametrixpressurechapter1}, hence $B_hh\nabla_{\pa}hD_yq|_{y=0}=-B_hh\nabla_{\pa}h\mathcal{N}q_0=O_{L^2(\mathbb{R}_x^{d-1})}(1)$.
\end{proof}
Again by hyperbolic estimates, we have the following result:
\begin{proposition}\label{boundL2chapter1}
	$\|w^{\pm}\|_{L^2(\mathbb{R}_+^{d})}=O(1).$
	In particular,
	$$\|B_hhD_yw\|_{L^2(\mathbb{R}_+^{d})}+\|B_hh\mathcal{N}q_0\|_{L^2(\mathbb{R}_x^{d-1})}+\|B_h h^2\Delta_{0}q_0\|_{L^2(\mathbb{R}_x^{d-1})}=O(1),
	$$
	where $\Delta_0=\Delta_{\partial\Omega}$, the Laplace-Beltrami operator on $\partial\Omega$. 
\end{proposition}
\begin{proof}
	It remains to prove $ \|B_hh^2\Delta_{0}q_0\|_{L^2(\mathbb{R}_x^{d-1})}=O(1).
	$
	Indeed, 
	\begin{equation*}
	\begin{split}
	B_hh\partial_y w_{\pe}=&B_hh\partial_y u_{\pe}-h^2B_h\partial_y^2q\\=
	&	h^2B_h\partial_y u_{\pe}+h^2B_h\Big(\frac{1}{\sqrt{\det(g)}}\sum_{1\leq j,k\leq d-1}\partial_j(g^{jk}\partial_kq_0)+\frac{\partial_y\sqrt{\det(g)}}{\sqrt{\det(g)}}\partial_y q\Big).
	\end{split}
	\end{equation*}
	Thus 
	\begin{equation*}
	\begin{split}
	B_hh\partial_yw_{\pe}|_{y=0}=B_hh^2\Delta_0 q_0+O_{L^2(\mathbb{R}^{d-1})}(1),
	\end{split}
	\end{equation*}
	thanks to $h\partial_y u_{\pe}=0$ and $B_hh\mathcal{N}q_0=O_{L^2(\mathbb{R}^{d-1})}(h^{-1})$.
	From $w_{0,\pe}^{\pm}=B_hhD_yw_{\pe}|_{y=0}+B_hQ_h^{\pm}h\mathcal{N}q_0=O_{L^2(\mathbb{R}^{d-1})}(1)$, we deduce that  
	$B_hhD_yw_{\pe}|_{y=0}=O_{L^2(\mathbb{R}^{d-1})}(1)$,  and these yield 
	$\|B_hh^2\Delta_{0}q_0\|_{L^2(\mathbb{R}^{d-1})}=O(1)$.
\end{proof}
\begin{corollary}\label{vanishingmeasureqchapter1}
	$\|B_hh\nabla q\|_{L^2(\mathbb{R}_+^d)}=o(1)$.
\end{corollary}
\begin{proof}
	We will go back to the global notation in this calculation. It suffices to show that $B_hh\nabla q=\varphi\mathrm{Op}_h(b)\varphi_1h\nabla q=o_{L^2(\Omega)}(1)$ since there are only change of bounded weight in the integral with respect to the measure $\sqrt{\det(g)}dydx$ and $dydx$ in local coordinate, and the former admit us to apply integration by part and the structure of the equation in a simple way. We calculate
	\begin{equation*}
	\begin{split}
	(B_hh\nabla q|B_hh\nabla q)_{L^2(\Omega)}=&([B_h,h\nabla ]q|B_hh\nabla q)_{L^2(\Omega)}+(h\nabla B_hq|B_hh\nabla q)_{L^2(\Omega)}\\
	=&o(1)-(A_hB_hq|h\mathrm{ div }B_hh\nabla q)_{L^2(\Omega)}\\
	+&(B_hhq_0|B_h(h\partial_{\nu}q)|_{\partial\Omega})_{L^2(\partial\Omega)}\\
	=&o(1)+(B_hhq_0|B_hh\mathcal{N}q_0)_{L^2(\partial\Omega)},
	\end{split}
	\end{equation*}
	where we have used the fact that $hq=o_{L^2(\Omega)}(1)$ and $\Delta q=0$ in the calculation. Now from Lemma \ref{press.normchapter1}, we know that $hq\rightharpoonup 0$ weakly in $H^1(\Omega)$ and $hq_0\rightarrow 0$ strongly in $L^2(\partial\Omega)$. The last term is $o(1)$ since $B_hh\mathcal{N}q_0=O_{L^2(\partial\Omega)}(1)$. 
\end{proof}

%%%%%%%%%%%%%%%%%%%%%%%%%%%%%%%%%%%%%%%%%%%%%%%%%%%%%%%%%%%%%%%%%%%%

%%%%%%%%%%%%%%%%%%%%%%%%%%%%%%%%%%%%%%%%%%%%%%%%%%%%%%%%%%%%%%%%%%

\subsection{propagation estimate}

In this subsection, we prove Proposition \ref{hyperbolicchapter1}.
We factorize $-h^2\Delta-1$ as
$(hD_y-Q_h^{\pm})(hD_y-Q_h^{\mp})+R^{\pm}$ near $z_0\in\mathcal{H}$ and choose $Q_h^{\pm}$ with principal
symbols $\pm Q(y,x,\xi)=\sqrt{1-\lambda^2}b_1(y,x,\xi)$, as in the previous subsection. Take $\psi\in C_c^{\infty}([0,\epsilon_0))$ with $\psi\equiv 1$ in a neighborhood of $y=0$. By an abuse of notation, we introduce
$$ w^{\pm}=B_h^{\pm}(hD_y-Q_h^{\mp})w,
$$
where $B_h^{\pm}$ have principal symbols $\psi(y)b^{\pm}(y,x,\xi)$. Here, $b^{\pm}$ are solutions of
\begin{equation}\label{ODEchapter1}
\frac{\partial b^{\pm}}{\partial y}\mp H_{Q(y,x,\xi)}b^{\pm}=0,\quad b^{\pm}|_{y=0}=b_0,
\end{equation}
where $b_0$ is another micro-localization near $z_0$ with $b_1|_{\textrm{supp}(b_0)}=1$,
and $H_{Q}b=\{Q,b\}$. Note that the compact support of $\psi(y)b^{\pm}$
can be chosen arbitrarily close to the semi-bicharacteristic curves $\gamma^{\pm}$ corresponding
to the principal symbol $p$.
Moreover, $b^{\pm}$ are invariant along $\gamma^{\pm}$. Under these notations, Proposition \ref{hyperbolicchapter1} can be rephrased as follows
\begin{proposition}\label{hyper propagationchapter1}
	Let $\mu$ be the defect measure of $u$. If $$b^+\mu\mathbf{1}_{0<y\leq\epsilon_0}=0\;(b^-\mu\mathbf{1}_{0<y\leq\epsilon_0}=0),$$
	then we have $$b^-\mu\mathbf{1}_{0<y\leq\epsilon_0}=0\;(b^+\mu\mathbf{1}_{0<y\leq\epsilon_0}=0).$$
	Moreover, we have in fact $b^+\mu=b^-\mu=0$ in this case.
\end{proposition}
%%%%%%修改到此处
The proof will be divided into several lemmas. First we calculate
$$(hD_y-Q_h^{\pm})w^{\pm}=[hD_y-Q_h^{\pm},B_h^{\pm}](hD_y-Q_h^{\mp})w+B_h^{\pm}(hD_y-Q_h^{\pm})(hD_y-Q_h^{\mp})w,
$$
%$$(hD_y-Q_h^{\pm})w^{\pm}=[hD_y-Q_h^{\pm},B_h^{\pm}](hD_y-Q_h^{\mp})w+B_h^{\pm}(hD_y-Q_h^{\pm})(hD_y-Q_h^{\mp})w，$$ 这一行是原有的，运行莫名其妙提示问题
and
$$[hD_y-Q_h^{\pm},B_h^{\pm}]=\frac{h}{i}\textrm{Op}_h(\partial_y b^{\pm}\mp H_{Q}b^{\pm})\psi(y)
+\frac{h}{i}\psi'(y)B_h^{\pm}+R''.
$$
The first operator vanishes thanks to the definition of $b^{\pm}$, and the remainder term $R''=O_{L^{2}(\mathbb{R}_+^d)}(h^2)$. Therefore we have
$$ \|R''(hD_y-Q_h^{\mp})w\|_{L^2(\mathbb{R}_+^d)}=O(h^2),
$$
and consequently
$$ (hD_y-Q_h^{\pm})w^{\pm}=\frac{h}{i}\psi'(y)w^{\pm}+g^{\pm},
$$
with $g^{\pm}=o_{L^2(\mathbb{R}_+^d)}(h)$. 

\begin{lemma}\label{out-incoming wavechapter1}
	Let $\mu^{\pm}$ be the semi-classical defect measure of $w^{\pm}$ and $b$ is defined as above. Suppose that
	$b^{\pm}\mu^{\pm}\mathbf{1}_{0<y\leq \epsilon_0}=0$, then we must have
	$b^{\pm}\mu^{\pm}\equiv 0$ and $\mu_0^{\pm}=0$, where $\mu_0^{\pm}$ is the defect measure
	of $w_0^{\pm}=w^{\pm}|_{y=0}$.
\end{lemma}
\begin{proof}
	For $y_0=\epsilon_0/2$, we first claim that $\|w^{\pm}(y_0)\|_{L_{x}^2}=o(1)$.
	Indeed, from the assumption and compactness, the measure $\mu^{\pm}$ vanishes in a small neighborhood of semi-bicharacteristic curve $\gamma^{\pm}$. Thus $\|w^{\pm}\|_{L^2([y_0,\epsilon_0]\times\mathbb{R}^{d-1})}=o(1)$, provided that we choose supp$(b_0)$ small enough in the definition of $w^{\pm}$. Finally, repeating the argument in the proof of Lemma \ref{hyper-estimateschapter1}, we have
	$$ -h\|w^{\pm}(y_0)\|_{L^2(\mathbb{R}^{d-1})}^2=i\int_{y_0}^{\epsilon_0}((Q_h^{\pm}-(Q_h^{\pm})^*)\chi w^{\pm}|\chi
	w^{\pm})_{L^2(\mathbb{R}^{d-1})}(y)dy+o(h).
	$$
	The claim then
	follows. 
	
	Integrating the identity
	\begin{equation}
	\begin{split}
	h\frac{d}{dy}(w^{\pm}|w^{\pm})_{L^2(\mathbb{R}^{d-1})}&=(i(Q_h^{\pm}-(Q_h^{\pm})^*)w^{\pm}|w^{\pm})_{L^2(\mathbb{R}^{d-1})}
	+2h(\psi'(y)w^{\pm}|w^{\pm})_{L^2(\mathbb{R}^{d-1})}\\
	&+2\textrm{Im}(w^{\pm}|g^{\pm})_{L^2(\mathbb{R}^{d-1})}
	\nonumber
	\end{split}
	\end{equation}
	from $y=z<y_0$ to $y=y_0$, we have
	$$\|w^{\pm}(z)\|_{L^2(\mathbb{R}^{d-1})}^2\leq C\int_{z}^{y_0}\|w^{\pm}(y)\|_{L^2(\mathbb{R}^{d-1})}^2dy+o(1).
	$$
	Using
	$ \int_0^{y_0}\|w^{\pm}(y)\|_{L^2(\mathbb{R}^{d-1})}^2dy=o(1),
	$
	we obtain that $\|w_0^{\pm}\|_{L^2(\mathbb{R}^{d-1})}=o(1)$. This completes the proof of Lemma \ref{out-incoming wavechapter1}.
\end{proof}
\begin{remark}\label{remark1chapter1}
	Away from the boundary, the defect measure of $u$ equals to the defect measure of
	$w$, and it propagates along the bicharacteristic curves $\gamma^{\pm}$. Since we can
	decompose $w$ into $w^+$ and $w^-$ near a hyperbolic point, we call $w^+$($w^-$)the incoming wave and
	the out-coming wave. Thus the above proposition asserts that if we have no singularity of $w^+$($w^-$) along
	incoming wave(out-coming wave) near the boundary but strictly away from the boundary, then there is no singularity of the boundary data of incoming wave(out-coming wave).
\end{remark}
Changing the role of $y=y_0$ and $y=0$ in the proof of Lemma \ref{out-incoming wavechapter1}, we conclude that if $\mu_0^{\pm}=0$, then $b^{\pm}\mu^{\pm}=0$.
To finish the proof of Proposition \ref{hyper propagationchapter1}, we need understand how the singularity transfers form boundary data of in-coming wave to the boundary data of out-coming wave.

\begin{lemma}\label{finalstepchapter1}
	$\mu_{0}^{\pm}=0$ implies that $\mu_0^{\mp}\mathbf{1}_{\xi\neq 0}=0.$ Consequently, $\mu^{\mp}\mathbf{1}_{\xi\neq 0}=0.$	
\end{lemma}
\begin{proof}
	By symmetry, we only need to deduce  $\mu_0^-\mathbf{1}_{\xi\neq 0}=0$ from $\mu_0^+=0$. For $\delta>0$, we define
	$$ b_{0,\delta}(x,\xi)=b_0(x,\xi)\Big(1-\widetilde{\psi}\big(\frac{\lambda(0,x,\xi)}{\delta}\big)\Big), 
	$$
	with some $\widetilde{\psi}\in C_c^{\infty}(\mathbb{R}), \widetilde{\psi}|_{[-2,2]}\equiv 1$. We define $b_{\delta}^{\pm}(y,x,\xi)$ by solving ODE \eqref{ODEchapter1} with initial data $b_{0,\delta}$. Let $B_{\delta,h}^{\pm}$ be the associated semi-classical PdO of $b_{\delta}^{\pm}$. From compactness and continuous dependence of the initial data, we have that  $\delta<\lambda(y,x,\xi)<c_0<1$ on supp$(b_{\delta}(y))$ for $0\leq y\leq \epsilon_0$, since on $Y_+\times\mathbb{R}^{d-1}$, $\frac{\lambda(y,x,\xi)}{|\xi|}\sim 1$. Note that the solutions of the transport equation \eqref{ODEchapter1} are given by  $$b^{\pm}(y,x,\xi)=b_0\circ\gamma^{\pm}(y)^{-1}(x,\xi),\quad b_{\delta}^{\pm}(y,x,\xi)=b_{0,\delta}\circ\gamma^{\pm}(y)^{-1}(x,\xi),$$
	we have that $\frac{b_{\delta}^{\pm}}{b^{\pm}}$ is a smooth function with compact support in $Y_+\times\mathbb{R}^{d-1}$. Denote by $w_{\delta}^{\pm}=B_{\delta,h}^{\pm}(hD_y-Q_h^{\mp})w$, and $\mu_{\delta}^{\pm}$ its semi-classical defect measure, we have $\mu_{\delta}^{\pm}=\mu^{\pm}\left(\frac{b_{\delta}^{\pm}}{b^{\pm}}\right)^2$. In particular, $\mu_{\delta,0}^{\pm}=\mu_0^{\pm}\left(1-\widetilde{\psi}\left(\frac{\lambda_0}{\delta}\right)\right)^2$ and  supp$(\mu_{\delta}^{\pm})\subset$ supp$(\mu^{\pm})$. On the boundary, $B_{\delta,h}^+$ and $B_{\delta,h}^-$ coincide and will be denoted by $B_{\delta,h}^0$. Taking the trace of $w_{\delta}^{\pm}$, we have
	\begin{equation*}
	\left\{
	\begin{aligned}
	&w_{\delta,0,\pa}^+=-iB_{\delta,h}^0v+ih^2B_{\delta,h}^0\partial_y(\nabla q)_{\pa}|_{y=0}+B_{\delta,h}^0Q_h^+h\nabla_{\pa}q_0,\\
	&w_{\delta,0,\pe}^+=iB_{\delta,h}^0h^2\partial_y^2q|_{y=0}+B_{\delta,h}^0Q_h^+h\partial_yq|_{y=0},
	\end{aligned}
	\right.
	\end{equation*}
	where $v=h\partial_yu|_{y=0}=O_{L_x^2}(1)$.
	Similarly, we have
	\begin{equation*}
	\left\{
	\begin{aligned}
	&w_{\delta,0,\pa}^-=-iB_{\delta,h}^0v+ih^2B_{\delta,h}^0\partial_y(\nabla q)_{\pa}|_{y=0}+B_{\delta,h}^0Q_h^-h\nabla_{\pa}q_0,\\
	&w_{\delta,0,\pe}^-=iB_{\delta,h}^0h^2\partial_y^2q|_{y=0}+B_{\delta,h}^0Q_h^-h\partial_yq|_{y=0}.
	\end{aligned}
	\right.
	\end{equation*}
	Notice that $\sigma(Q_h^+)=-\sigma(Q_h^-)$, we write
	$\alpha=-B_{\delta,h}^0h^2\Delta_{0}q_0,\beta=B_{\delta,h}^0Q_h^+h\mathcal{N}q_0,$ hence
	$$ w_{\delta,0,\pe}^{\pm}=i\alpha\mp \beta+O_{L_x^2}(h).
	$$
	From the assumption $\|w_{\delta,0,\pe}^+\|_{L^2}=o(1)$, we have that
	$\|i\alpha-\beta\|_{L^2}^2=o(1),
	$
	and this implies that
	$\|\alpha\|^2+\|\beta\|^2-2\textrm{Im}(\alpha|\beta)=o(1).
	$
	We claim that $\textrm{Im}(\alpha|\beta)=o(1).$
	
	Indeed, from Proposition \ref{boundL2chapter1} and the ellipticity of the Dirichlet-Neumann operator $\mathcal{N}$, we have that $q_0=O_{L^2(\mathbb{R}_x^{d-1})}(1)$, micro-locally away from $\xi=0$.  
	Now from the trace theorem and Proposition \ref{parametrixpressurechapter1}, we have
	$$\beta=A_{\delta,h}q_0+O_{L^2(\mathbb{R}_x^{d-1})}(h^{1/3})
	$$
	for some PdO with real-valued principal symbol $a_{\delta},$ compactly supported and vanishing when $\lambda(y,x,\xi)\leq\delta/4$. Similarly,
	$$ \alpha=A_{\delta,h}'q_0+o_{L^2(\mathbb{R}_x^{d-1})}(1)
	$$
	for some PdO with real-valued principal symbol $a_{\delta}'.$
	Thus 
	$\textrm{Im}(\alpha|\beta)_{L^2}=o(1),$  since all the principal symbols involved in the inner product are real-valued.
	Now from $\|\alpha\|_{L^2}=o(1),\|\beta\|_{L^2}=o(1),$ one deduce that the terms on the righthand side of $w_{\delta,0,\pa}^{\pm}$ involving pressure are also $o_{L_x^2}(1)$, and $v=o_{L_x^2}(1)$ follows since $w_{\delta,0,\pa}^-=o_{L^2(\mathbb{R}^{d-1})}(1)$. Therefore $\mu_{\delta,0}^{-}=0$ and consequently $\mu_{\delta}^-=0$ from Lemma \ref{out-incoming wavechapter1}. This implies that $\mu_{0}^-\mathbf{1}_{\lambda>\delta}=\mu^-\mathbf{1}_{\lambda>\delta}=0$.  Since $\delta>0$ is arbitrary, we have that $b^-\mu^-\mathbf{1}_{\xi\neq 0}=0$. Moreover, Corollary \ref{vanishingmeasureqchapter1} implies that $\mu\mathbf{1}_{\xi\neq 0}=0$. This completes the proof of Lemma \ref{finalstepchapter1}.
\end{proof}
Now we finish the proof of Proposition \ref{hyper propagationchapter1} by showing the following lemma. 	
\begin{lemma}
	$\mu^+=0$ implies that $\mu^-=0$.
\end{lemma}
\begin{proof}
	We only need deal with $\xi=0$. 
	Take $\widetilde{\psi}$ to be a cut-off function which equals $1$ in a near the origin. Pick any $\epsilon>0,$ we define the operator
	$$ B_{h}^{\epsilon,\pm}=\textrm{Op}_h(\widetilde{\psi}(\lambda(y,x,\xi)/\epsilon))B^{\pm}_h.$$
	Applying divergence equation for $w^{\pm}$
	$$ B_{h}^{\epsilon,\pm}h\textrm{div}_{\pa}w^{\pm}_{\pa}+B_{h}^{\epsilon,\pm}h\partial_yw^{\pm}_{\pe}=O_{L^2(\mathbb{R}_+^d)}(h),
	$$
	we have 
	$$ \|B_{h}^{\epsilon,\pm}h\partial_yw^{\pm}_{\pe}\|_{L^2(\mathbb{R}_+^d)}\leq\|B_{h}^{\epsilon,\pm}h\textrm{div}_{\pa}w^{\pm}_{\pa}\|_{L^2(\mathbb{R}_+^d)}
	+R_{\epsilon}(h)$$
	with $R_{\epsilon}(h)\rightarrow 0,$ as $h\rightarrow 0$ for each fixed $\epsilon>0$.
	By estimating the operator norm from its symbol, we have
	$$\|B_{h}^{\epsilon,\pm}h\partial_yw_{\pe}^{\pm}\|_{L^2(\mathbb{R}_+^d)}\leq C\epsilon+R_{\epsilon}(h),
	$$
	and
	$$ \limsup_{h\rightarrow 0^+}\|B^{\epsilon,\pm}_{h}h\partial_yw_{\pe}^{\pm}\|_{L^2(\mathbb{R}_+^d)}\leq C\epsilon.
	$$
	Using the equation
	$hD_yw_{\pe}^{\pm}=Q_h^{\pm}w_{\pe}^{\pm}+O_{L^2(\mathbb{R}_+^d)}(h)
	$,
	we have
	$$\limsup_{h\rightarrow 0^+}\|B_{h}^{\epsilon,\pm}Q_h^{\pm}w_{\pe}^{\pm}\|_{L^2(\mathbb{R}_+^d)}\leq C\epsilon.
	$$
	Finally let $\epsilon\rightarrow 0$, we have
	$\mu_{\pe}^{\pm}\mathbf{1}_{\xi=0}=0.
	$
	Therefore $\mu_{\pe}^-=0$. As a consequence of the proof of Lemma \ref{out-incoming wavechapter1},
	$\mu_{0,\pe}^{\pm}\mathbf{1}_{\xi=0}=0.
	$
	Now let $\mu_{\alpha},\mu_{\beta}$ be the defect measures of $\alpha=-B_h^0h^2\Delta_0q_0,\beta=B_h^0Q_hh\mathcal{N}q_0$, and let $\mu_{i\alpha\pm\beta}$ be the defect measure of $i\alpha\pm\beta$. Denote also by 
	$\mu_{\alpha\beta}$ the limit corresponding to the quadratic form $(A_h\alpha|\beta)$. Similarly for $\mu_{\beta\alpha}$. Note that $\mu_{\alpha\beta}=\overline{\mu_{\beta\alpha}}$.
	From
	\begin{equation*}
	\begin{split}
	&\langle\mu_{i\alpha+\beta},\mathbf{1}_{\xi=0}\rangle=\langle\mu_{\alpha},\mathbf{1}_{\xi=0}\rangle+\langle\mu_{\beta},\mathbf{1}_{\xi=0}\rangle-\langle 2\Im\mu_{\alpha\beta},\mathbf{1}_{\xi=0}\rangle=0\\
	&\langle\mu_{i\alpha+\beta},\mathbf{1}_{\xi=0}\rangle=\langle\mu_{\alpha},\mathbf{1}_{\xi=0}\rangle+\langle\mu_{\beta},\mathbf{1}_{\xi=0}\rangle+\langle 2\Im\mu_{\alpha\beta},\mathbf{1}_{\xi=0}\rangle=0,
	\end{split}
	\end{equation*}
	we have that $\mu_{\alpha}\mathbf{1}_{\xi=0}=\mu_{\beta}\mathbf{1}_{\xi=0}=0$.

	Next we consider parallel components. The key claim is that the measure corresponding to $B_h^0Q_h^{\pm}h\nabla_{\pa}q_0$ vanishes on the set $\{\xi=0\}$.
	Indeed, from Lemma \ref{press.normchapter1} and the trace theorem, $hq_0\rightarrow 0$ strongly in $L^2(\partial\Omega)$. From the ellipticity of $\mathcal{N}$, there exists a classical pseudo-differential operator $E$ of order $-1$ such that $E\mathcal{N}=I+R$, where $R$ is a classical smoothing operator. Our goal is to show that
	\begin{equation*}
	\lim_{\epsilon\rightarrow 0}\limsup_{h\rightarrow 0}\|B_{h}^0B_{h}^{\epsilon,0}Q_h^{\pm}h\nabla_{\pa}q_0\|_{L^2(\mathbb{R}^{d-1})}=0.
	\end{equation*}
	From symbolic calculus and the strong convergence of $hq_0$ in $L^2(\mathbb{R}_x^{d-1})$, it suffices to prove
	\begin{equation}\label{zerolimitchapter1}
	\lim_{\epsilon\rightarrow 0}\limsup_{h\rightarrow 0}\|h\nabla_{\pa}B_{h}^0B_{h}^{\epsilon,0}Q_h^{\pm}q_0\|_{L^2(\mathbb{R}^{d-1})}=0.
	\end{equation}

	We write
	\begin{equation}\label{oddcommutatorchapter1}
	\begin{split}
	h\nabla_{\pa}B_h^0B_{h}^{\epsilon,0}Q_h^{\pm}q_0=&
	h\nabla_{\pa}B_h^0B_{h}^{\epsilon,0}Q_h^{\pm}E\mathcal{N}q_0-h\nabla_{\pa}B_h^0B_{h}^{\epsilon,0}Q_h^{\pm}Rq_0\\
	=&\nabla_{\pa}EB_h^0B_{h}^{\epsilon,0}Q_h^{\pm}h\mathcal{N}q_0+h\nabla_{\pa}[B_h^0B_{h}^{\epsilon,0}Q_h^{\pm},E]\mathcal{N}q_0\\
	-&h\nabla_{\pa}B_h^0B_{h}^{\epsilon,0}Q_h^{\pm}Rq_0.
	\end{split}
	\end{equation}
	Here we are taking the commutator between a semi-classical PdO and a classical PdO, hence the semi-classical symbolic calculus is not applicable. Yet, it is not difficult to check that for any $a\in C_c^{\infty}(T^*\partial\Omega)$, $E\in S_{x,\xi}^{-1}$, 
	$$[a(x,hD_x),E(x,D_x)]=h\textrm{Op}(S^{-1})+\textrm{Op}(S^{-2}),
	$$
	where the implicit constants only depend on the semi-norms of the symbols $a(x,\xi)$ and $E(x,\xi)$. Notice that $h\nabla_{\pa}B_h^0,B_{h}^{\epsilon,0},Q_h^{\pm}$ are uniformly bounded operators in $L_x^2$ with respect to $h$, thus $\nabla_{\pa}B_h^0B_{h}^{\epsilon,0}Q_h^{\pm}R$, $\nabla_{\pa}B_h^0\textrm{Op}(S^{-2})\mathcal{N}$, $h\nabla_{\pa}B_h^0\textrm{Op}(S^{-1})\mathcal{N}$ are uniformly bounded operators in $L_x^2$ with respect to $h$. Thus from the strong convergence of $hq_0$, the last two terms on the right hand side of \eqref{oddcommutatorchapter1} are killed when we let $h\rightarrow 0$ first. Thus \eqref{zerolimitchapter1} follows from the vanishing of the measure of $\pm\beta=B_h^0Q_h^{\pm}h\mathcal{N}q_0$ on the set $\{\xi=0\}$.
	Combining the assumption that $\mu_{0,\pa}^+\mathbf{1}_{\xi=0}=0$, we deduce that $\mu_{0,\pa}^-\mathbf{1}_{\xi=0}=0$. The proof of Proposition \ref{hyper propagationchapter1} is now complete.
\end{proof}

%%%%%%%%%%%%%%%%%%%%%%%%%%%%%%%%%%%%%%%%%%%%%%%%%%%%%%%%%%%%%%%%%%%%%%%%%%%%
\section{Near $\mathcal{G}^{2,+}$}

In this section, we follow the strategy of V.~Ivrii (see \cite{bookIvrii} or \cite{bookHormander(III)}) to prove Proposition \ref{diffractivechapter1}. Denote by $G=\det(\ov{g})$ and $P_h=h^2\Delta_H-1$, we have
\begin{lemma}\label{HodgeLaplacechapter1}
	In local coordinate $Y_+$, we have
	$$P_h=-h^2\frac{\ov{g}}{\sqrt{G}}\partial_y\big(\sqrt{G}\ov{g}^{-1}\partial_y\big)+R_h=h^2D_y^2+\mathrm{Op}_h(r)+O_{L^2\rightarrow L^2}(h),
	$$	
	where $R_h$ is a matrix-valued second order differential operator in $x$ with scalar principal symbol $r(y,x,\xi)=1-\lambda(y,x,\xi)^2$, which is self-ajoint with respect to the $(\cdot|\cdot)_{L^2(Y_+)}$.  
\end{lemma}
The proof will be given in the appendix.

To simplify the notations, in a fix local coordinate in $Y_+$, we will identify $u=\varphi_1u,q=\varphi_1 q$ and all the operators $B$ by $\varphi B\varphi_1$.
\begin{proposition}\label{priorichapter1}
	For any tangential operator $B$ with scalar principal symbol $b(y,x,\xi)$ vanishing near $\xi=0$,  we have
	$$ \limsup_{h\rightarrow 0}\|BhD_yu\|_{L^2(Y_+)}\leq
	\sup_{\rho\in\textrm{supp }(b)}|r(\rho)|^{1/2}|b(\rho)|.
	$$
\end{proposition}
\begin{proof}
	We calculate
	\begin{equation*}
	\begin{split}
	(BhD_yu|BhD_yu)_{Y_+}=&([B,hD_y]u|BhD_yu)_{Y_+}+(hD_yBu|BhD_y u)_{Y_+}\\
	=&O(h)+(Bu|Bh^2D_y^2u)_{L^2(Y_+)}\\
	=&O(h)-(Bu|BRu)_{L^2(Y_+)}+(Bu|BP_hu)_{Y_+}\\
	=&O(h)-(Bu|BRu)_{Y_+}-(Bu|Bh\mathrm{d} q)_{Y_+}.
	\end{split}
	\end{equation*}
	Integrating by part and using symbolic calculus, we have
	\begin{equation*}
	\begin{split}
	(Bu|Bh\mathrm{d} q)_{Y_+}=&(Bu|h\mathrm{d} Bq)_{Y_+}+(Bu|[B,h\mathrm{d}]q)_{Y_+}\\
	=&-([h\mathrm{d}^*,B]u|Bq)_{Y_+}+(Bu|[B,h\mathrm{d}]q)_{Y_+}\\
	=&O(h),
	\end{split}
	\end{equation*}
	thanks to the fact that $B$ has scalar-valued principal symbol.
\end{proof}

The proof of Proposition \ref{diffractivechapter1} is based on the following integration by part result.

\begin{proposition}\label{integrating by partchapter1}
	Given real scalar-valued tangential symbols $a_0,a_1$, there exist tangential operators $A_0,A_1$(constructed in the local coordinate) with real, scalar-valued principal symbol $a_0,a_1$ and $A=A_1hD_y+A_0$, such that for any 1-form $w$ with compact support in $Y_+$, we have
	$$\frac{2}{h}\Im(P_hw|Aw)_{Y_+}=(A_1hD_yw|hD_yw)_{\underline{\partial}Y_+}+
	\Re\sum_{j=0}^2(C_j(hD_y)^jw|w)_{Y_+}+O(h),
	$$
	where the tangential operators $C_j$ have scalar-valued principal symbol $c_j(y,x,\xi)$ and
	$$ \sum_{j=0}^2c_j(y,x,\xi)\eta^j=\{p,a\}.
	$$
\end{proposition}
\begin{proof}
	We first calculate
	\begin{equation*}
	\begin{split}
	I=&\frac{1}{ih}\big(-\frac{h^2}{\sqrt{G}}\ov{g}h\partial_y(\sqrt{G}\ov{g}^{-1}h\partial_yw)\big| Aw \big)_{Y_+}-\frac{1}{ih}\big(Aw\big| -\frac{h^2}{\sqrt{G}}\ov{g}h\partial_y(\sqrt{G}\ov{g}^{-1}h\partial_yw) \big)_{Y_+}\\
	=&(hD_yw|A_1hD_yw)_{\underline{\partial}Y_+}+(A_1hD_yw|hD_yw)_{\underline{\partial}Y_+}\\
	+&\frac{1}{ih}(hD_yw|hD_yAw)_{Y_+}-\frac{1}{ih}(hD_yAw|hD_yw)_{Y_+}\\
	=&(hD_yw|A_1hD_yw)_{\underline{\partial}Y_+}+(A_1hD_yw|hD_yw)_{\underline{\partial}Y_+}
	\\
	+&\frac{1}{ih}(hD_yw|[hD_y,A]w)_{Y_+}-\frac{1}{ih}([hD_y,A]w|hD_yw)_{Y_+}\\
	+&\frac{1}{ih}(hD_yw|AhD_yw)_{Y_+}-\frac{1}{ih}(AhD_yw|hD_yw)_{Y_+},
	\end{split}
	\end{equation*}
	and the last two terms on the right hand side equal to 
	\begin{equation*}
	\begin{split}
	\frac{1}{ih}(A^*hD_yw|hD_yw)-(A_1^*hD_yw|hD_yw)_{\underline{\partial}Y_+}-\frac{1}{ih}(AhD_yw|hD_yw)_{Y_+}.
	\end{split}
	\end{equation*}
	We want to construct operators $A_0,A_1$ such that $A_1^*=A_1+O(h^{2})$ and $A^*=A+O(h^{2})$. Assume that
	$$ \widetilde{a_1}\backsimeq a_1^{(0)}+\frac{h}{i}a_1^{(1)}
	$$
	with real-valued $a_1^{(j)}$ (not necessarily scalar-valued). 
	From
	$$ \int_{Y_+}\langle A_1u|v\rangle_{\mathbb{R}^{d-1}}\sqrt{G}dydx=\int_{Y_+}\langle \ov{g}^{-1}A_1u,v\rangle_{\mathbb{R}^{d-1}}\sqrt{G}dydx,
	$$
	the symbol of $A_1^*$ is equal to the symbol of
	$\frac{\ov{g}}{\sqrt{G}}\mathrm{Op}_h(\widetilde{a_1}^*)\sqrt{G}\ov{g}^{-1}$, which can be expressed by
	$$ b_1(y,x,\xi)\backsimeq\sum_{k\geq 0}\left(\frac{h}{i}\right)^kb_1^{(k)}(y,x,\xi),
	$$
	with 
	$$ b_1^{(k)}(y,x,\xi)=\sum_{j=0}^1\sum_{|\alpha|+|\beta|+j=k}(-1)^j\partial_{\xi}^{\beta}\big(\frac{\ov{g}}{\sqrt{G}}\partial_{\xi}^{\alpha}\partial_x^{\alpha}a_1^{(j)}\big)\cdot\partial_x^{\beta}(\sqrt{G}\ov{g}^{-1}),\quad k\geq 1.
	$$
	We have that
	$$ b_1^{(0)}=a_1^{(0)},\quad b_1^{(1)}=-a_1^{(1)}+\sum_{|\alpha|+|\beta|=1}\partial_{\xi}^{\beta}\big(\frac{\ov{g}}{\sqrt{G}}\partial_{\xi}^{\alpha}\partial_x^{\alpha}a_1^{(0)}\big)\cdot\partial_x^{\beta}(\sqrt{G}\ov{g}^{-1})
	$$
	We set $$a_1^{(0)}=a_1,\quad  a_1^{(1)}=\frac{1}{2}\sum_{|\alpha|+|\beta|=1}\partial_{\xi}^{\beta}\big(\frac{\ov{g}}{\sqrt{G}}\partial_{\xi}^{\alpha}\partial_x^{\alpha}a_1^{(0)}\big)\cdot\partial_x^{\beta}(\sqrt{G}\ov{g}^{-1}),$$
	thus $A_1^*=A_1+O(h^2)$. Note that $a_1^{(1)}$ is matrix-valued while $a_1^{(0)}$ is real and scalar-valued.
	
	The construction of $A_0$ is similar. We observe that $(hD_y)^*=hD_y+h\frac{g}{\sqrt{G}}D_y(\sqrt{G}\ov{g}^{-1})$ and set 
	$$ \widetilde{a_0}=a_0^{(0)}+\frac{h}{i}a_0^{(1)}.
	$$  
	$A_0^*$ has symbol which can be expanded as
	$$ b_0\backsimeq \sum_{k\geq 0}\big(\frac{h}{i}\big)^kb_0^{(k)}
	$$
	with $b_0^{(0)}=a_0^{(0)}$ and
	\begin{equation}\label{b0kchapter1}
	b_0^{(k)}(y,x,\xi)= \sum_{j=0}^1\sum_{|\alpha|+|\beta|+j=k}(-1)^j\partial_{\xi}^{\beta}\big(\frac{\ov{g}}{\sqrt{G}}\partial_{\xi}^{\alpha}\partial_x^{\alpha}a_0^{(j)}\big)\cdot\partial_x^{\beta}(\sqrt{G}\ov{g}^{-1}),\quad k\geq 1.
	\end{equation}
	Note that
	\begin{equation*}
	\begin{split}
	(hD_y)^*A_1^*-A_1^*hD_y=&[(hD_y)^*,A_1^*]+A_1^*(hD_y)^*-A_1^*hD_y\\=&\frac{h}{i}(\partial_yA_1^*)+\frac{h}{i}\big[\frac{\ov{g}}{\sqrt{G}}\partial_y(\sqrt{G}\ov{g}^{-1}),A_1^*\big]+\frac{h}{i}A_1^*\frac{\ov{g}}{\sqrt{G}}\partial_y(\sqrt{G}\ov{g}^{-1}),
	\end{split}
	\end{equation*} 
	and its symbol can be expanded as 
	$$ \sum_{k\geq 0}\big(\frac{h}{i}\big)^k\kappa_k(y,x,\xi)
	$$
	with $\kappa_0=0$ and 
	\begin{equation*}
	\begin{split}
	\kappa_1=&\partial_yb_1+b_1\frac{\ov{g}}{\sqrt{G}}\partial_y(\sqrt{G}\ov{g}^{-1}),\\
	\kappa_k=&\sum_{|\alpha|=k-1}\frac{1}{i^{|\alpha|+1}}\{\partial_{\xi}^{\alpha},\partial_x^{\alpha}\}\big(\frac{\ov{g}}{\sqrt{G}}\partial_y(\sqrt{G}\ov{g}^{-1}),b_1\big)\\
	&+\frac{h}{i}\big(\partial_yb_1+b_1\frac{\ov{g}}{\sqrt{G}}\partial_y(\sqrt{G}\ov{g}^{-1})\big)+\frac{h}{i}\sum_{|\alpha|\geq 1}\frac{h^{|\alpha|}}{i^{|\alpha|}}\{\partial_{\xi}^{\alpha},\partial_x^{\alpha}\}\big(\frac{\ov{g}}{\sqrt{G}}\partial_y(\sqrt{G}\ov{g}^{-1}),b_1\big), \quad k\geq 2,
	\end{split}
	\end{equation*}
	where
	$$\{\partial_{\xi}^{\alpha},\partial_x^{\alpha}\}(f_1,f_2)=\partial_{\xi}^{\alpha}f_1\partial_{x}^{\alpha}f_2-\partial_{\xi}^{\alpha}f_2\partial_{x}^{\alpha}f_1.
	$$
	We set $b_0^{(0)}=a_0$ and $a_0^{(1)}$ such that $a_0^{(1)}=b_0^{(1)}+\kappa_1$(it has a solution thanks to \eqref{b0kchapter1}).
	Finally, we construct $A_j$ by $\varphi_1\mathrm{Op}_h(\widetilde{a_j})\varphi_1$ in local coordinate and it can be easily verified that $$A_1^*=A_1+O_{L^2\rightarrow L^2}(h^2), \quad A^*=A+O_{L^2\rightarrow L^2}(h^2). 
	$$  
	Therefore 
	$$I=(hD_yw|A_1hD_yw)_{\underline{\partial}Y_+}+\frac{1}{ih}(hD_yw|[hD_y,A]w)_{Y_+}-\frac{1}{ih}([hD_y,A]w|hD_yw)_{Y_+}+O(h).
	$$
	We next calculate 
	\begin{equation*}
	\begin{split}
	\frac{1}{ih}(R_hw|Aw)_{Y_+}-\frac{1}{ih}(Aw|R_hw)_{Y_+}
	=&\frac{1}{ih}((A^*R _h-R_h^*A)w|w)_{Y_+}\\
	=&\frac{1}{ih}([A,R_h]w|w)_{Y_+}+O(h),
	\end{split}
	\end{equation*}
	since $R_h$ is self-ajoint and $A^*-A=O_{L^2\rightarrow L^2}(h^2)$. Moreover, the principal symbol of $\frac{1}{ih}[A,R_h]$ is $\{r,a\}$. This completes the proof of Proposition \ref{integrating by partchapter1}.
\end{proof}

Now assume that we are working near a diffractive point $\rho\in\mathcal{G}^{2,+}$ in $Y_+$ where
$$ \partial_yr\geq c_0>0
$$

The following lemma is a semi-classical version of Lemma 24.4.5 in \cite{bookHormander(III)}. The proof is slightly more complicated, due to the different equation that we are considering.
\begin{lemma}\label{translationchapter1}
	Let $B_j=\varphi B_j\varphi_1,$ with real, scalar-valued tangential principal symbols $b_j, j=0,1,2$, compactly supported and
	$$ \sum_{j=0}^2b_j(y,x,\xi)\eta^j=-\psi(y,x,\eta,\xi)^2 \textrm{when }\eta^2=r(y,x,\xi),
	$$
	with some smooth function $\psi\in C_c^{\infty}(\mathbb{R}^d\times (\mathbb{R}^d\setminus \{0\}))$. Assume that
	$$ dr\neq 0, \partial_yr>0, \textrm{on }\{y=r=0\} \cap \bigcup_{j=1}^2\mathrm{supp }(b_j).
	$$
	Then one can chose compactly supported,  tangential operators $\Psi_j,j=0,1$ with real, scalar-valued principal symbols $\psi_j,j=0,1$, satisfying
	$$ \psi_0(y,x,\xi)=\psi(y,x,0,\xi), \psi_1(y,x,\xi)=\partial_{\eta}(y,x,0,\xi) \textrm{ when }\eta=r(y,x,\xi)=0,
	$$
	so that for any solution $u$ of $P_hu=f-h\nabla q,h\mathrm{ div }u=0$ with $u|_{y=0}=0$, we have
	\begin{equation}\label{positivitychapter1}
	\begin{split}
	&\Re\sum_{j=0}^2(B_j(hD_y)^jv|v)_{Y_+}+\|\Psi_0v+\Psi_1hD_yv\|_{L^2(Y_+)}^2
	+(\Theta P_hv|v)_{Y_+}\\
	= &o(1)
	\end{split}
	\end{equation}
	as $h\rightarrow 0$, where $v=\varphi\mathrm{Op}_h(\chi)\varphi_1u$ and $\chi\in C_c^{\infty}(Y_+\times\mathbb{R}^{d-1})$ has support near $\rho\in \mathcal{G}^{2,+}$. $\Theta$ is a tangential operator, depending on $\psi_j,b_j$ whose principal symbols are scalar-valued.
\end{lemma}

The proof is based on the following elementary lemma, for which the proof can be found as Lemma 24.4.3 in \cite{bookHormander(III)}, 
\begin{lemma}\label{test}
	Let $X$ be an open subset of $\ov{\mathbb{R}_+^{n}}=\{x\in\mathbb{R}^n:x_1\geq 0\}$, and let $r\in C^{\infty}(X)$. Assume that $r$ is real-valued, that $dr\neq 0$ when $r=0$ and that $\frac{\partial r}{\partial x_1}>0$ when $r=x_1=\frac{\partial r}{\partial x_j}=0$ for $j\neq 1$. Let
	$$ F(t,x)=\sum_{j=0}^2f_j(x)t^j
	$$
	be a quadratic polynomial in $t$ with coefficients in $C^{\infty}(X)$ such that
	$$  F(t,x)=-\psi(t,x)^2\textrm{ when }t^2=r(x),
	$$
	where $\psi\in C^{\infty}(\mathbb{R}\times X).$ Then one can find $\psi_0,\psi_1,\theta\in C^{\infty}(X)$ such that $\psi_0(x)=\psi(0,x),\psi_1(x)=\frac{\partial\psi}{\partial t}(0,x)$ when $r(x)=0$, and
	$$ F(t,x)+(\psi_0(x)+t\psi_1(x))^2\leq \theta(x)(t^2-r(x)),\quad \forall t\in\mathbb{R}, x\in X.
	$$
\end{lemma}

\begin{proof}[Proof of Lemma \ref{translationchapter1}]
	Choose $C^{\infty}$ functions $\psi_0(y,x,\xi)$ and $\psi_1(y,x,\xi)$ as in Lemma \ref{test}, such that  $\psi_j(y,x,\xi)=\partial_{\eta}^{j}\psi|_{y=0},j=0,1$ when $\eta=r(y,x,\xi)=0$ and $$ \sum_{j=0}^2b_j\eta^j+(\psi_0+\eta\psi_1)^2\leq \theta(y,x,\xi)(\eta^2-r).
	$$
	Since $\psi_0,\psi_1$ and each $b_j$ are compactly supported in variables $(y,x,\xi)$, we may assume that	$\theta$ is smooth and with compact support. Define $\Theta=\varphi\textrm{Op}_h(\theta)\varphi_1$, $\Psi_j=\varphi\textrm{Op}_h(\psi_j)\varphi_1, j=0,1$ and consider the quantity
	\begin{equation*}
	\begin{split}
	\Re\sum_{j=0}^2(B_j(hD_y)^jv|v)_{Y_+}+((\Psi_0+\Psi_1hD_y)^2v|v)_{Y_+}
	-(\Theta hD_yv|hD_yv)_{Y_+}+(\Theta R_hv|v)_{Y_+}.
	\end{split}
	\end{equation*}
	The expression above can be written under the form below
	$$ \sum_{j=0}^2\left(C_j(hD_y)^jv|v\right)_{Y_+},
	$$
	where the tangential operators $C_j$ have real, scalar-valued principle symbol. Moreover,
	$$ 	 \sum_{j=0}^2c_j(y,x,\xi)\eta^j\leq 0.
	$$
	However, since the symbol is not bounded in $\eta$ and we can not apply sharp G\aa rding inequality directly. To resolve this issue, we extend each $c_j$ to $\widetilde{c_j}\in C_c^{m}(\mathbb{R}\times\mathbb{R}^{2d-2})$ who agrees with $c_j$ on $y\geq 0$ up to order $m$, any given order, of derivatives. This is possible since any order of $y$ derivative of all the symbols has continuous limit as $y\rightarrow 0$. We still use the notation $c_j$ in what follows. Let $\underline{v}=v\mathbf{1}_{y\geq 0}$ and we use the boundary condition $v|_{y=0}=0$ and calculate 
	\begin{equation*}
	\begin{split}
	\big(\sum_{j=0}^2C_j(hD_y)^jv\big|v\big)_{Y_+}
	=&\big(\sum_{j=0}^2C_j(hD_y)^j\underline{v}\big|\underline{v}\big)_{Y_+}
	\\=
	&\big(\psi\big(\frac{hD_y}{A}\big)\sum_{j=0}^2C_j(hD_y)^j\underline{v}\big|\underline{v}\big)_{Y_+}\\
	+&\big(\big(1-\psi\big(\frac{hD_y}{A}\big)\big)\sum_{j=0}^2C_j(hD_y)^j\underline{v}\big|\underline{v}\big)_{Y_+}\\
	=:&\mathrm{I}+\mathrm{II},
	\end{split}
	\end{equation*}
	for any big number $A>0$. Now we apply sharp G\aa ding inequality to the first term to get
	$$ I\leq C_Ah,
	$$
	with some constant $C_A$ depending on $A$.
	For the second term, the principle symbol is supported in the elliptic region and we define 
	$$ \Xi(y,x,\eta,\xi):=\big(1-\psi\big(\frac{\eta}{A}\big)\big)\displaystyle{\sum_{j=0}^2}\frac{c_j(y,x,\xi)\eta^j}{\eta^2-r(y,x,\xi)}\in S^0(\mathbb{R}^{2d}),
	$$
	hence we can bound
	\begin{equation*}
	\begin{split}
	|\mathrm{II}|\leq & O(h)+
	C\big(\Xi(y,x,hD_y,hD_x)\chi(y,x,hD_x)P_h\underline{u}\big|\big(1-\psi\big(\frac{2hD_y}{A}\big)\underline{v}\big)\big)_{Y_+}\\
	=&O(h)+C\big(\Xi(y,x,hD_y,hD_x)\chi(y,x,hD_x)(hw\otimes\delta_{y=0})\big|\big(1-\psi\big(\frac{2hD_y}{A}\big)\underline{v}\big)\big)_{Y_+}\\
	+&C\big(\Xi(y,x,hD_y,hD_x)\chi(y,x,hD_x)(\mathbf{1}_{y\geq 0}h\nabla q)\big|\big(1-\psi\big(\frac{2hD_y}{A}\big)\underline{v}\big)\big)_{Y_+},
	\end{split}
	\end{equation*}
	with $w=hD_yu|_{y=0}$. 
	Note that to obtain the expression above, one can not use symbolic calculus to deal with commutator between semi-classical tangential symbol and the classical symbol. However, since $P_h$ is a differential operator, we can compute its commutator with $\chi(y,x,hD_x)$ explicitly.
	
	Now from Proposition \ref{etaellipticchapter1}, the limsup of the third term on the right hand side when $h\rightarrow 0$ can be bounded by $\epsilon(A)$ with
	$ \displaystyle{\lim_{A\rightarrow\infty}\epsilon(A)=0.}
	$ Here we can use the flat metric to estimate the $L^2$ norm. 
	The second term on the right hand side can be bounded by
	\begin{equation*}
	\begin{split}
	&Ch\|(1-h^2\Delta_{y,x})^{-\frac{s}{2}}(w\otimes\delta_{y=0})\|_{L^2(\mathbb{R}^d)}\|(1-h^2\Delta_{y,x})^{\frac{s}{2}}\underline{v}\|_{L^2(\mathbb{R}^d)}
	\leq Ch^{1-s},
	\end{split}
	\end{equation*} 
	for any $s\in\left(
	\frac{1}{2},1\right)$. Here we have used the fact that $\delta_{y=0}\in H^{-s}(\mathbb{R}_y)$ for any $s>\frac{1}{2}$ and $h^s\underline{v}$ is bounded in $H^{s}(\mathbb{R}^d)$ since $\underline{v}|_{y=0}=0$ and $h\nabla_{y,x'}v$ is bounded in $L^2(\mathbb{R}^d)$.
	Therefore, for any $A>0$, we have showed that
	$$ \limsup_{h\rightarrow 0}|\mathrm{II}|\leq \epsilon(A),
	$$
	and	this completes the proof of Lemma \ref{translationchapter1}.
\end{proof}

Adapting to the notations in this section, Proposition \ref{diffractivechapter1} can be rephrased as follows
\begin{proposition}\label{propagation of supportchapter1}
	Suppose that $\rho\in\mathcal{G}^{2,+}$, and $\rho_0\in T^*\Omega$ approaching to $\rho$ such that $\partial_yr(\rho_0)\geq \frac{1}{2}\partial_yr(\rho)\geq c_0$. Let $\gamma_-=[\rho_0,\rho]$ be a segment of the generalized ray issued from $\rho_0$ to $\rho$ (the trajectory under the canonical projection is tangent to the boundary at $\rho$). Then if $\rho_0\notin\mathrm{supp }(\mu)$, we have $\rho\notin\mathrm{supp }(\mu)$.
\end{proposition}
\begin{proof}
	Take a small neighborhood $\Gamma_0$ of $\rho_0$ such that $\Gamma_0\cap\mathrm{supp}(\mu)=\varnothing$. Take a small neighborhood $W_0\subset \overline{\Omega}\times \mathbb{R}^{d-1}$ such that$ \frac{\partial r}{\partial y}(y,x,\xi)\geq c_0/4>0.$ Shrinking $W_0$ if necessary, we assume that each point $(y,x,\xi)\in W_0$ with $r(y,x,\xi)\geq 0$ can be connected by a (possibly broken) ray issued from $\Gamma_0$ with at most one reflection or tangency at $\partial\Omega$. It suffices to prove the following statement:
	
	For any $\chi\in C_c^{\infty}(\overline{\Omega}\times\mathbb{R}^{d-1})$ with  $\mathrm{supp}(\chi)\subset W_0$, small enough, we have
	$$  \varphi\mathrm{Op}_h(\chi)\varphi_1u=o_{L^2}(1),h\rightarrow 0.
	$$
	
	As in \cite{bookHormander(III)}, we construct test functions which satisfy the following properties:
	\begin{lemma}\label{test1chapter1}
		There exists
		$$ a(y,x,\eta,\xi)=a_0(y,x,\xi)+a_1(y,x,\xi)\eta,\quad a_j\in C_c^{\infty}(W_0)
		$$
		with the following properties:
		\begin{enumerate}
			\item $a_1(0,x,\xi)=-t(x,\xi)^2$, for some $t\in C_c^{\infty}(T^*\underline{\partial}Y_+)$,
			\item For some large $M\geq 0$, when $p=\eta^2-r(y,x,\xi)=0$, we have
			$$ \{p,a\}+aM=-\psi(y,x,\eta,\xi)^2+\omega(y,x,\xi)(\eta-r^{1/2}(y,x,\xi)),\;a=s^2,
			$$
			where $s\in C^{\infty}(Y_+\times(\mathbb{R}^d\setminus\{0\}))$,$\psi\in C_c^{\infty}(Y_+\times\mathbb{R}^d\setminus\{0\})$ and $\omega\in C_c^{\infty}(W_0)$. Moreover, $r|_{\mathrm{supp}(\omega)}>0$.
		\end{enumerate}
	\end{lemma}
	The construction is exactly the same as in \cite{bookHormander(III)} and will be given in the appendix D for the sake of completeness.
	
	Now we take $\chi\in C_c^{\infty}(W_0)$ with $\chi\equiv1,$ in a neighborhood of  $\mathrm{supp}(a_1)\cup\mathrm{supp}(a_2)$. Let $v=\varphi\mathrm{Op}_h(\chi)\varphi_1u,$ and we calculate
	\begin{equation}
	\begin{split}
	(P_hv|Av)_{Y_+}=&(\varphi\textrm{Op}_h(\chi)\varphi_1P_hu|Av)_{Y_+}+
	([P_h,\varphi\textrm{Op}_h(\chi)\varphi_1]u|Av)_{Y_+}\\
	=&(\varphi\textrm{Op}_h(\chi)\varphi_1f|Av)_{Y_+}-(\varphi\textrm{Op}_h(\chi)\varphi_1h\mathrm{d} q|Av)_{Y_+}\\+&([P_h,\varphi\textrm{Op}_h(\chi)\varphi_1]u|Av)_{Y_+}.
	\nonumber
	\end{split}
	\end{equation}
	Here we have used the differential form to calculate the inner product.
	Notice that $\{p,\chi\}=0$ on supp$(a_j)$ and $f=o_{L^2}(h)$, $hD_yu_{\pe}|_{y=0}=0,$ thus
	\begin{equation}\label{3chapter1}
	\begin{split}
	\frac{2}{h}\Im(P_hv|Av)_{\Omega}&=o(1)-\frac{2}{h}\Im([\varphi\textrm{Op}_h(\chi)\varphi_1,h\mathrm{d}]q|A\varphi\textrm{Op}_h(\chi)\varphi_1u)_{Y_+}\\
	+&\frac{2}{h}\Im(\textrm{Op}_h(\chi)q|h\mathrm{d}^*(A\varphi\textrm{Op}_h(\chi)\varphi_1u))_{Y_+}.
	\end{split}
	\end{equation}
	From Proposition \ref{integrating by partchapter1},
	\begin{equation}\label{4chapter1}
	\begin{split}
	\sum_{j=0}^2(C_j(hD_y)^jv|v)_{Y_+}&=-(A_1hD_yv|hD_yv)_{\underline{\partial}Y_+}-\frac{2}{h}\Im([\varphi\mathrm{Op}_h(\chi)\varphi_1,h\mathrm{d}]q|Av)_{Y_+}\\
	&+\frac{2}{h}\Im(\varphi\mathrm{Op}_h(\chi)\varphi_1q|h\mathrm{d}^*(A\varphi\mathrm{Op}_h(\chi)\varphi_1u))_{Y_+}+o(1).
	\end{split}
	\end{equation}
	Since the principal symbol of $A$ is scalar-valued, by using $\mathrm{d}^*u=0$, we have 
	$$ \frac{2}{h}\Im(\varphi\mathrm{Op}_h(\chi)\varphi_1q|hd^*(A\varphi\mathrm{Op}_h(\chi)\varphi_1u))_{Y_+}
	=(\varphi\mathrm{Op}_h(\chi)\varphi q|\Upsilon u)_{Y_+}+O(h)$$
    and
	$$-\frac{2}{h}\Im([\varphi\mathrm{Op}_h(\chi)\varphi_1,hd]q|Av)_{Y_+}=(\Upsilon_2q|Av)_{Y_+}+O(h),
	$$
		where $\Upsilon=\Upsilon_0+\Upsilon_1 hD_y$,
	and $\Upsilon_j$ are matrix-valued tangential pseudo-differential operators with principal symbols supported in $\mathrm{supp}(\chi)$.
	Applying Lemma \ref{translationchapter1} to the function
	$$ \sum_{j=0}^2c_j\eta^j+aM-\omega(\eta-r^{1/2})=-\psi^2,
	$$
	we have
	\begin{equation}\label{5chapter1}
	\begin{split}
	&\Re\big(\sum_{j=0}^2C_j(hD_y)^ju|u\big)_{Y_+}-\Re(\Phi(hD_y-\widetilde{Q}_+)v|
	\varphi(y,x,hD_x)v)_{Y_+}\\
	&+\Re(Mv|Av)_{Y_+}+(\Theta P_hv|v)_{Y_+}+\|\Psi_0v+\Psi_1hD_yv\|_{L^2(Y_+)}^2\\
	\leq& o(1)+ Ch\|v\|_{L^2(Y_+)}^2,
	\end{split}
	\end{equation}
	where the compact supported tangential operator $\Phi$ has scalar-valued principal symbol $\phi\in C_c^{\infty}(W_0)$ and $r|_{\textrm{supp} (\phi)}>0$, $\phi=1$ in a neighborhood of supp $\omega$.  $\widetilde{Q}_{+}$ is the operator constructed in the hyperbolic region with principal symbol $r^{1/2}$. This is possible since in the proof of Lemma \ref{test1chapter1}, we indeed have $r\geq \delta^2|\xi|^2$ on the support of $\omega$.
	Note that the principal symbol of $A$ is positive on $\eta^2-r=0$, we can apply Lemma \ref{translationchapter1} again to the term $(Mv|Av)_{\Omega}$ and bound it from below by
	$$ o(1)-|(\Theta_1P_hv|v)_{Y_+}|.
	$$
	Thus
	we have
	\begin{equation}\label{6chapter1}
	\begin{split}
	&-(A_1hD_yv|hD_yv)_{\underline{\partial}Y_+}+\|\Psi_0v+\Psi_1hD_yv\|_{L^2(Y_+)}^2\\
	&\leq o(1)+Ch\|v\|_{L^2(Y_+)}^2+C|(\Theta P_v|v)_{Y_+}|+C|(\Theta_1P_hv|v)_{Y_+}|\\&+
	\left|\Re(\varphi\mathrm{Op}_h(\phi)\varphi_1(hD_y-\widetilde{Q}_+)v|
	\varphi\mathrm{Op}_h(\omega)\varphi_1v)_{Y_+}\right|
	\\&+|(\Upsilon_2q|Av)_{Y_+}|+|(\varphi\mathrm{Op}_h(\chi)\varphi_1q|\Upsilon u)_{Y_+}|.
	\end{split}
	\end{equation}
	The terms on the left hind side are essentially positive from the sharp G\aa rding (semi-classical, see \cite{bookZworski}) inequality, hence we only need to control the terms on the right hind side. The term $|(\Theta P_hv|v)_{\Omega}|+|(\Theta_1 P_hv|v)_{\Omega}|=o(1)$ follows from the equation and symbolic calculus since the principal symbols of $\Theta$ and $\Theta_1$ are scalar-valued. Next we claim that
	\begin{equation}\label{claimchapter1}
	\left|\Re(\varphi\mathrm{Op}_h(\phi)\varphi_1(hD_y-\widetilde{Q}_+)v|
	\varphi\mathrm{Op}_h(\omega)\varphi_1v)_{Y_+}\right|=o(1),\quad h\rightarrow 0.
	\end{equation}
	Indeed, micro-locally on supp$(\phi)$, $r\gtrsim \delta^2>0$, hence in the region where $\lambda^2(y,x,\xi)<1$, we could construct $\widetilde{Q}_+,\widetilde{Q}_-$ micro-locally such that $$
	P_h=(hD_y-\widetilde{Q}_-)(hD_y-\widetilde{Q}_+)+O(h^{\infty})$$
	as we have done in the hyperbolic case. 
	From symbolic calculus and Corollary \ref{vanishingmeasureqchapter1}, we have 
	$$ (hD_y-\widetilde{Q}_-)(hD_y-\widetilde{Q}_+)u=O_{L_{y,x}^2}(h)+h\nabla q=o_{L_{y,x}^2}(1),\textrm{ micro-locally on supp}(\phi).
	$$
	Therefore the measure $\mu$ concentrates on $\{\eta=-\sqrt{r}\}\cup\{\eta=\sqrt{r}\}$. For any point $\rho_1\in$ supp$(\phi)\cap$supp$(\mu)$, with $\eta(\rho_1)=-\sqrt{r(\rho_1)}<0$, the backward generalized ray issued from $\rho_1$ must enter $\Gamma_0$ without meeting any point in $\mathcal{G}^{2,+}$, since along the backward flow, $\eta$ is decreasing. Consequently, away from the boundary,  $u=o_{L^2(Y_+)}(1)$ and hence $(hD_y-\widetilde{Q}_+)u=o_{L^2(Y_+)}(1)$, micro-localized near $\eta=-\sqrt{r}$, due to the fact that $hD_y-\widetilde{Q}_-$ is micro-locally elliptic near $\eta=\sqrt{r}$. Near the boundary and some point $\rho_1\in\mathcal{H}\cap \mathrm{supp}(\phi)$, any point can be connected backwardly to $\Gamma_0$ by at most transversal reflection. Thus \eqref{claimchapter1} holds true.
	
	It remains to control the last two terms involving pressure. We just treat one of them, and the other can be treated in the same way. %%%%%%%%%%%%%%%%%%%%%%%%%%%%%%%%订正到此处
	Pick $\varphi_0\in C_c^{\infty}((-2,2))$ which is equal to $1$ on $(-1,1)$. Define
	$$ \chi_{\epsilon}(y,x,\xi)=\chi(y,x,\xi)\varphi_0\left(r(y,x,\xi)\epsilon^{-1}\right).
	$$
	We fix any $\epsilon>0$, small enough, and write
	\begin{equation}
	\begin{split}
	(\Upsilon_2q|Av)_{Y_+}=&(\Upsilon_2q|A\varphi\textrm{Op}_h(\chi_{\epsilon})\varphi_1u)_{Y_+}
	\\+&(\Upsilon_2q|A\varphi\textrm{Op}_h(\chi-\chi_{\epsilon})\varphi_1u)_{Y_+}\\
	=:&\mathrm{I}_{h,\epsilon}+\mathrm{II}_{h,\epsilon}.
	\end{split}
	\end{equation}
		We first deal with $\mathrm{I}_{h,\epsilon}$.
	Notice that from Proposition \ref{priorichapter1}, we have
	$$ \limsup_{h\rightarrow 0}\|hD_y\varphi\mathrm{Op}_h(\chi_{\epsilon})\varphi_1u\|_{L^2(Y_+)}\leq C\epsilon^{1/2}.$$
	Applying Cauchy Schwartz, we have
	\begin{equation}\label{non-concentration}
	\begin{split}
	\int_0^{y_0}\|\varphi\mathrm{Op}_h(\chi_{\epsilon})\varphi_1u\|_{L^2(\underline{\partial}Y_+,\sqrt{G}dx)}^2dy&\leq
	Ch^{-2}\int_0^{y_0}\int \Big|\int_0^{y}hD_y\varphi\textrm{Op}_h(\chi_{\epsilon})\varphi_1u(s,x)ds\Big|^2dxdy\\
	&\leq \frac{Cy_0^2}{h^2}\|hD_y\varphi\textrm{Op}_h(\chi_{\epsilon}\varphi_1u)\|_{L_{x,y}^2}^2.
	\nonumber
	\end{split}
	\end{equation}
	
	By choosing $\theta\in (0,1/2)$ and $y_0=h\epsilon^{-\theta}$, we estimate
	\begin{equation}
	\begin{split}
	|\mathrm{I}_{h,\epsilon}|&\leq \Big(\int_0^{y_0}+\int_{y_0}^{\epsilon_0}\Big)|(\Upsilon_2q|A\varphi\mathrm{Op}_h(\chi_{\epsilon})\varphi_1u)_{L^2(\underline{\partial}Y_+,\sqrt{G}dx)}|dy\\
	&\leq C\frac{1}{\epsilon^{2\theta}}(\|hD_y\textrm{Op}_h(\chi_{\epsilon})u\|_{L_{x,y}^2}^2+O(h))+ Ce^{-\frac{c}{\epsilon^{\theta}}},
	\nonumber
	\end{split}
	\end{equation}
	where we have used Lemma \ref{concentration1chapter1}. Note that Lemma \ref{concentration1chapter1} is applicable even when the micro-local cut-off $\chi_{\delta_0}$ is matrix-valued.
		In summary we have
	$$\limsup_{h\rightarrow 0}|\mathrm{I}_{h,\epsilon}|\leq C(\epsilon^{1-2\theta}+e^{-\frac{c}{\epsilon^{\theta}}}).
	$$

	We now turn to the estimates of
	$\mathrm{II}_{h,\epsilon}$. This can be done from geometric argument. Let $$S_{\epsilon}:=\{(y,x,\xi):r(y,x,\xi)\geq\epsilon,y\leq 4\epsilon/c_0\}\cap W_0.$$
	We claim that for any ray $\gamma$ with $\gamma(0)\in\Gamma_0$ and $\Gamma(s_0)\in S_{\epsilon}$, $\gamma|_{[0,s_0]}\cap\mathcal{G}^{2,+}=\emptyset.$
	
	Indeed, by contradiction, assume that for some $\gamma$ and $s_1\in[0,s_0]$, we have $\rho_1=\gamma(s_1)\in\mathcal{G}^{2,+}$.   
	After time $s_1$, along $\gamma$ we have 
	$$ \dot{y}=2\eta,\dot{\eta}=\partial_yr\geq c_0/4,
	$$
	with $y(s_1)=\eta(s_1)=0,\eta(s_0)\geq \sqrt{\epsilon}$. This implies that $s_0-s_1\geq 4\sqrt{\epsilon}/c_0$ and $y(s_0)\geq c_0T^2/4\geq \frac{4\epsilon\delta_0}{c_0}$. The claim then follows.
	
	Now we write
	\begin{equation*}
	\begin{split}
	\mathrm{II}_{\epsilon}=&\left(\varphi_0\left(c_0y/\epsilon\right)\Upsilon_2q|A\varphi\textrm{Op}_h(\chi-\chi_{\epsilon})\varphi_1u\right)_{Y_+}
	\\+&\left(\left(1-\varphi_0\left(c_0y/\epsilon\right)\right)\Upsilon_2q|A\varphi\textrm{Op}_h(\chi-\chi_{\epsilon})\varphi_1u\right)_{Y_+}.
	\end{split}
	\end{equation*}
	
	From the discussion above, the first term on the right hand side above tends to $0$ as $h\rightarrow 0$ for any fixed $\epsilon>0$, while the second term is controlled from above by
	$$ \int_{\frac{\epsilon\delta_0}{4C}}^{\epsilon_0}\int\left|C(y,x,hD_x)q\right|^2dxdy
	$$
	for some zero order semi-classical tangential operator with principal symbol $c(y,x,\xi)$ such that supp $c\cap\{\xi=0\}=\emptyset$.  Applying Lemma \ref{concentration1chapter1}, we have
	$ \displaystyle{\limsup_{h\rightarrow 0}|\mathrm{II}_{\epsilon}|=0}
	$
	is true for any $\epsilon>0$. 
	Notice that the left hand side of \eqref{5chapter1} is independent of $\epsilon,$ we have
	$$ \limsup_{h\rightarrow 0}((-A_1hD_yv|hD_yv)_{\underline{\partial}Y_+}+\|\Psi_0v+\Psi_1hD_y v\|_{L^2(Y_+)}^2)=0.
	$$
	From the construction of $a_0,a_1$ and the corresponding expression of $\psi_0,\psi_1$, we can choose another different $\widetilde{a_0},\widetilde{a_1}$, such that the function $\widetilde{\psi}_0+\widetilde{\psi}_1\eta$ is independent of $\psi_0+\psi_1\eta$ on supp$(\chi)$(see appendix D). It follows then
	$$ \|v\|_{L^2(Y_+)}+\|hD_y v\|_{L^2(Y_+)}=o(1),\quad h\rightarrow 0,
	$$
	and this completes the proof of Proposition \ref{propagation of supportchapter1}.
\end{proof}

\section{Near $\mathcal{G}^{2,-}$ and $\mathcal{G}^k$ for $k\geq 3$}

This section is devoted to the proof of Proposition \ref{highorderchapter1}. Before proving it, we need some preparation. In what follows, we take tangential operators
$$A=\varphi\mathrm{Op}_h(a)\varphi_1,\quad  A^*=A+O_{L^2(\underline{\partial}Y_+)}(h^2).
$$
\begin{proposition}\label{integrating by part2chapter1}
	$$\frac{1}{h}(([P,A]u|u)_{Y_+}=\frac{1}{h}(Au|Pu)_{Y_+}-\frac{1}{h}(APu|u)_{Y_+}+O(h).
	$$
\end{proposition}
\begin{proof}
	The proof goes in  exactly the same way and much simpler than the diffractive case, and we omit it here.
\end{proof}
Recall that $r_0=r|_{y=0}$ and $r_1=\partial_yr|_{y=0}$.
Direct calculation gives 
$$ H_pa=2\eta\frac{\partial a}{\partial y}+\frac{\partial r}{\partial y}\frac{\partial a}{\partial \eta}+H_{-r}a.
$$
Pick $\rho_0\in\mathcal{G}^{2,-}\subset T^*\partial\Omega\diagdown\{0\}$ and a small neighborhood $U\subset T^*\partial\Omega\diagdown\{0\}$ of $\rho_0$.  
Let $L\subset U$ be a co-dimension $1$ hypersurface containing $\rho_0$ in $T^*\partial\Omega$ and transversal to the vector field $H_{-r_0}$. For small positive numbers $\delta,\tau>0$, define
$$ L^{\pm}(\delta,\tau;\rho_0):=\{\exp(tH_{-r_0})(\rho)\in U:\rho\in L,\textrm{dist }(\rho,\rho_0)\leq \delta^2,0\leq \pm t\leq \tau\}.
$$
When there is no risk of confusion, we write it simply as $L^{\pm}(\delta,\tau)$. Define also
\begin{equation*}
\begin{split}
&F^{\pm}(\delta,\tau):=\{(y,x,\xi):0\leq y\leq \delta^2,(x,\xi)\in L^{\pm}(\delta,\tau)\},\\
&F(\delta,\tau)=F^{+}(\delta,\tau)\cup F^-(\delta,\tau).
\end{split}
\end{equation*} 
Let $C_1>0$ sufficiently large and $\delta_0>0,\tau_0>0$ sufficiently small so that $\delta<\delta_0,\tau<\tau_0$
\begin{equation}\label{smallconditionchapter1}
|r(y,x,\xi)|\leq \frac{1}{2}C_1^2\delta^2
\end{equation}
in $F(\delta,\tau)$ for all $0<\delta\leq\delta_0,0<\tau\leq\tau_0$. With the same constant $C_1$, we further define the sets
\begin{equation*}
\begin{split}
V^{\pm}(\delta,\tau):=&\{(y,x,\eta,\xi):0\leq y\leq \delta^2/2,(x,\xi)\in L^{\pm}(\delta,\tau)\}\\
\cup &\{(y,x,\eta,\xi):\delta^2/2\leq y\leq \delta^2,(x,\xi)\in L^{\pm}(\delta,\tau),|\eta|\leq C_1\delta\},\\
W^{\pm}(\delta,\tau):=&\{(y,x,\eta,\xi):0\leq y\leq \delta^2/2,(x,\xi)\in L^{\pm}(\delta,\tau)\}\\
\cup &\{(y,x,\eta,\xi):\delta^2/2\leq y\leq \delta^2,(x,\xi)\in L^{\pm}(\delta,\tau),|\eta|\leq 2C_1\delta\},
\end{split}
\end{equation*}
$$V(\delta,\tau):=V^{+}(\delta,\tau)\cup V^-(\delta,\tau),\quad W(\delta,\tau)=W^{+}(\delta,\tau)\cup W^-(\delta,\tau).
$$
We need test functions constructed in \cite{Melrose-SjostrandI}:
\begin{lemma}[\cite{Melrose-SjostrandI}]\label{MS1chapter1}
	Let $I=[0,\epsilon_0)_y$. There exist $\sigma>0,\delta_0>0,\tau_0>0$, small enough with the hierarchy $\delta_0\ll\sigma\ll 1$, and families of smooth functions $a_{\delta}\in C_c^{\infty}(I\times U),g_{\delta},h_{\delta}\in C^{\infty}(Y_+\times\mathbb{R}_{\eta}\times\mathbb{R}^{d-1}_{\xi}\setminus\{0\})$,w here $0<\delta\leq \delta_0$, satisfying the following properties:
	\begin{enumerate}
		\item $a_{\delta}\geq 0,$ $\mathrm{supp}(a_{\delta})\subset F^{+}(\delta,\sigma\delta)\cup F^-(\delta,\delta^2)$ .
		\item $a_{\delta}(0,\exp(tH_{-r_0}(\rho_0)))\neq 0,\;\forall 0\leq t<\delta\sigma$.
		\item $a_{\delta}>0$ on $\mathrm{supp}( a_{\delta'})$ if $0<\delta'<\delta\leq\delta_0$ and $a_{\delta}$ is independent of $y$ for $0\leq y<\delta^2/2$.
		\item $g_{\delta}+h_{\delta}=-H_pa_{\delta}$.
		\item in $W(\delta,\tau)$, $g_{\delta}\geq 0$ and $g_{\delta}>0$ when $a_{\delta}\neq 0$.
		\item For any $m>1$ and any multiple index $\alpha\in\mathbb{N}^d$, $|g_{\delta
		}^{-\frac{1}{m}}\partial^{\alpha}g_{\delta}|=O_{\delta}(1)$, locally uniformly on $W(\delta,\tau_0)$, where the implicit constant inside $O_{\delta}(1)$ depends on $\alpha, m$ and $\delta$.
		\item $\mathrm{supp}(h_{\delta})\subset I\times L^-(\delta,\delta^2)\times \mathbb{R}_{\eta}$, and $\mathrm{supp}(g_{\delta})\cup \mathrm{supp}(h_{\delta})\subset \mathrm{supp}(a_{\delta})$, $g_{\delta},h_{\delta}$ are independent of $\eta$ for $0\leq y\leq \delta^2/2$.
	\end{enumerate}
\end{lemma}
For the convenience of the reader, we will recall the proof in the appendix D.

According to the lemma, we have $\partial (g_{\delta}^{1/2})=2g_{\delta}^{-1/2}\partial g_{\delta}=O(1)$, this implies that $g_{\delta}^{1/2}\in C^{\infty}(W(\delta,\tau))$.
Set $b_{\delta}:=g_{\delta}^{1/2}\in C^{\infty}(W(\delta,\tau))$. Note $b_{\delta}$ may not be smooth with compact support. We need split it into two parts as follows:  
Let $\phi_1\in C^{\infty}(\mathbb{R})$ such that $\phi_1\equiv 1$ if $0\leq y\leq \frac{\delta^2}{4}$ and $\phi_1\equiv 0$ if $y>\frac{3\delta^2}{8}$.
Let $\phi_2\in C^{\infty}(\Omega\times\mathbb{R}^d\setminus\{0\})$ with compact support in $x,\xi,\eta$ variables, such that $\phi_2\geq 0$ and $\phi_2\equiv 0$ whenever $y\leq \frac{\delta^2}{4}$ or $|\eta|>2C_1\delta$. Indeed, we can choose 
$\kappa\in C_c^{\infty}(\mathbb{R})$, non-negative, smooth and with compact support, such that $\kappa(z)\equiv 0$ when $|z|>2C_1\delta $ and $\kappa(z)\equiv 1$ when $|z|\leq \frac{3}{2}C_1\delta$. Now let $\phi_2(y,x,\eta,\xi)^2=(1-\phi_1(y)^2)\kappa(\delta^{-1}\eta)\chi_{\delta}(y,x,\xi)$ with $\chi_{\delta}|_{\mathrm{supp}(a_{\delta})}\equiv 1$, $\mathrm{supp}(\chi_{\delta})\subset F^{+}(\delta,\sigma\delta)\cup F^{-}(\delta,\sigma\delta)$.
We observe that 
$$ W^{\pm}(\delta,\tau)\cap\textrm{supp }(1-\phi_1^2-\phi_2^2)\subset\big\{(y,x,\eta,\xi):\frac{\delta^2}{4}\leq y\leq \delta^2,|\eta|>\frac{3}{2}C_1\delta,\big\}.
$$
We finally put $b_{\delta,j}:=\phi_jb_{\delta},j=1,2$. Note that $b_{\delta,1}\in C_c^{\infty}(F(\delta,\tau))$ is a tangential symbol (since for $y\geq\frac{\delta^2}{2},$ $\mathrm{supp}(g_{\delta})\subset \mathrm{supp}(a_{\delta})$ is compact) while $b_{\delta,2}\in C_c^{\infty}(W(\delta,\tau))$ is a usual interior symbol with compact support in $T^*\Omega$. %%%%%%%%%%%%%%%订正到此处
\subsection{Gliding case}
The propagation of support of $\mu$ near a gliding point in $\mathcal{G}^{2,-}$ can be stated as follows:

\begin{proposition}\label{glidingchapter1}
	Suppose $\rho_0\in\mathcal{G}^{2,-}$ and $L^{+}(\delta_0,\tau_0)\cup L^-(\delta_0,\tau_0)\subset\mathcal{G}^{2,-}$ for some $\delta_0,\tau_0>0$.  Then for any $\sigma>0$ with $\sigma\delta_0<\tau_0$, such that if
	$$ \{(y,x,\eta,\xi):0\leq y\leq \delta^2,(x,\xi)\in L^-(\delta,\delta^2;\rho_0)\}\cap \mathrm{supp }(\mu)=\varnothing$$
	for some $0<\delta\leq \delta_0,$ then $\exp(tH_{-r_0})(\rho_0)\notin $ $\mathrm{supp }(\mu)$ for any $t\in [0,\sigma \delta)$.
\end{proposition}
We need several lemmas.
\begin{lemma}\label{compositionlawchapter1}
	Suppose $a\in C_c^{\infty}(\mathbb{R}^{2d}), b\in C^{\infty}([0,1],C_c^{\infty}(\mathbb{R}^{2(d-1)}))$ with the following support property:
	$$a(y,x,\eta,\xi)\equiv 0 \textrm{ if } y\leq c_0<1 \textrm{ or } |\eta|>C_0|\xi|.$$
	Then the usual symbolic calculus for $a(y,x,hD_y,hD_x)b(y,x,hD_x)$ still valid. In particular,
	$$ a(y,x,hD_y,hD_x)b(y,x,hD_x)=c(y,x,hD_y,hD_x)+O_{L^2\rightarrow L^2}(h),
	$$
	with 
	$$ c(y,x,\eta,\xi)=a(y,x,\eta,\xi)b(y,x,\xi)
	$$
\end{lemma}
We postpone the proof in the appendix F.
\begin{lemma}\label{geometric1chapter1}
	Given any $\rho_1\in\mathcal{G}$, there exist $\delta_1>0,\tau_1>0,\sigma_1>0$ with $\delta_1\ll\sigma_1$ and  $\sigma_1\delta_1<\tau_1$ such that if $\rho\in T^*\partial\Omega$ and $\mathrm{dist }(\rho,\rho_1)\leq \delta^2$ for some $0<\delta\leq\delta_1$,  then $\mathrm{dist }(\gamma(s,\rho),\gamma(s,\rho_1))\leq C\delta^2$ for $|s|\leq \sigma_1\delta$. In particular, $\gamma(s,\rho)\in W(\delta,\tau_1)$ for all $|s|\leq \sigma_1\delta$.
\end{lemma}

\begin{proof}
	Write $\gamma(s,\rho)$ and $\mathrm{exp}(sH_{-r_0})(\rho)$ in coordinate as  $$\gamma_1(s)=(y(s),\eta(s),x(s),\xi(s)) \textrm{ and } \gamma_2(s)=(\tilde{y}(s),\tilde{\eta}(s),\tilde{x}(s),\tilde{\xi}(s)).$$ From $\dot{y}=2\eta,\dot{\eta}=O(1)$, we have $y(s)\leq Cs^2$ and the same estimate holds for $\tilde{y}(s)$.
	Let $$d(s)=|x(s)-\tilde{x}(s)|^2+|\xi(s)-\tilde{\xi}(s)|^2,$$ and then $\dot{d}(s)\leq C\sqrt{d(s)}$. This implies
	$d(s)\leq C_1\delta^2$ for all $|s|\leq \sigma_1\delta$. By the same argument, we have $\mathrm{dist }(\mathrm{exp}(sH_{-r_0})(\rho),\mathrm{exp}(sH_{-r_0})(\rho_1))\leq C\delta^2$. The conclusion then follows from the triangle inequality.
\end{proof}
We will see the crucial role of $\rho_0\in\mathcal{G}^{2,-}$ in the following lemma: 

\begin{lemma}\label{geometric2chapter1}
	Assume that $\delta_1, \tau_1$ are parameters given in the previous lemma. Suppose that $-C_0\leq \partial_yr(\rho)\leq -c_0<0$ for all $\rho\in W(\delta_1,\tau_1)$. Define $S_{\epsilon}=W(\delta_1,\tau_1)\cap\{r\geq\epsilon,y\leq\epsilon\}$ for  sufficiently small $\epsilon>0$. Then along any ray $\gamma(s,\rho_1)$ in $W(\delta_1,\tau_1)$ with $\rho_1\in S_{\epsilon}$, if $y(\gamma(-t,\rho_1))=0$ for some $0\leq t\leq\tau_1$, we have $r(y(\gamma(-t,\rho_1))\geq c_1\epsilon$, where $c_1$ depends only on $W(\delta_1,\tau_1)$.
\end{lemma}

\begin{proof}
	Assume $\rho_1=(y_1,x_1,\eta_1,\xi_1)\in S_{\epsilon}$ and $\gamma(s,\rho_1)=(y(s),x(s);\eta(s),\xi(s))$. Let $s_3=\inf\{0\leq s\leq \tau_1:y(-s)=0\}$. For $s\in[-s_3,0]$, $\dot{y}=2\eta,-C_0\leq\dot{\eta}=\partial_yr\leq -c_0$. 
	There are two possibilities. If $\eta_1\geq\sqrt{\epsilon}$, then $\eta(-s_3^+)\geq\eta_1>0$ since $\dot{\eta}<0$. Otherwise,  $\eta_1\leq-\sqrt{\epsilon}$, and we denote by $s_2=\inf\{s\in[0,s_1]:\eta(-s)=0\}$.  
	From
	$\displaystyle{\eta_1=\int_{-s_2}^{0}\dot{\eta}ds\geq-C_0s_2,}$
	we have $s_2\geq\frac{|\eta_1|}{C_0}$. Moreover,
	$$ y_1-y(-s_2)=2\eta_1s_2-\int_{-s_2}^{0}ds\int_s^0\ddot{y}ds'\leq 2\eta_1s_2+C_0s_2^2\leq -\frac{|\eta_1|^2}{C_0}.		
	$$
	Now from
	$$ y(-s_2)=y(-s_2)-y(-s_3)=-\int_{-s_3}^{-s_2}ds\int_{s}^{-s_2}\ddot{y}ds'\leq C_0|s_3-s_2|^2,
	$$
	we have $\displaystyle{|s_3-s_2|^2\geq \frac{y(-s_2)}{C_0}\geq\frac{(y(-s_2)-y_1)}{C_0}\geq \frac{|\eta_1|^2}{C_0^2}}$ and finally 
	$$ \eta(-s_3^+)=-\int_{-s_3}^{-s_2}\dot{\eta}ds\geq c_0|s_3-s_2|\geq \frac{c_0\sqrt{\epsilon}}{C_0}.
	$$
	The proof of Lemma \ref{geometric2chapter1} is then complete by applying the argument  above between any two adjacent zeros of $s\mapsto y(\gamma(-s,\rho_1))$.
\end{proof}

\begin{proof}[Proof of Proposition \ref{glidingchapter1}]
	For any $\delta'>0$, we define the operator
	$$N_{\delta'}=\frac{1}{ih}[P,A_{\delta'}]
	$$
	with principal symbol
	$ n_{\delta'}=-H_pa_{\delta'}=g_{\delta'}+h_{\delta'}.
	$
	Define operators
	$$ B_{\delta',j}:=\textrm{Op}_h(b_{\delta',j}),j=1,2,\quad  N_{\delta,3}=\textrm{Op}_h((1-\phi_1^2-\phi_2^2)n_{\delta'}).
	$$
	Write 
	$h_{\delta',j}=\phi_j^2h_{\delta'}, H_{\delta',j}=\textrm{Op}_h(h_{\delta',j}), j=1,2.
	$
	The proposition will follow if we can show that for any $\delta'<\delta,$ 
	\begin{equation}\label{positiveconclusionchapter1}
	\lim_{h\rightarrow 0}\sum_{j=1}^2\|B_{\delta',j}u\|_{L^2(Y_+)}^2=0
	\end{equation}

	We remark that $h_{\delta',1},b_{\delta',1}$ are both tangential symbols while $h_{\delta',2},b_{\delta',2}$ are interior symbols vanishing near the boundary. Observe also that $N_{\delta',3}$ is interior pseudo-differential operator with symbol vanishing near the boundary as well as on $p^{-1}(0)$, thanks to the fact that in $W(\delta',\tau)$, $|r(y,x,\xi)|\leq \frac{1}{2
	} C_1^2\delta'^2$. Thus $N_{\delta',3}u=o_{L^2(Y_+)}(1)$ as $h\rightarrow 0$ for $\delta'>0$ small enough.
	Moreover, from the assumption on the support of $\mu$ near the original point $\rho_0$ we have $H_{\delta',j}u=o_{L_{y,x}^2}(1)$.	Now set
	$$ M_{\delta',j}=\phi_j^2N_{\delta',j}-B_{\delta',j}^*B_{\delta',j}-H_{\delta',j},j=1,2.
	$$
	From symbolic calculus, we deduce that $M_{\delta',1}=O_{L^2\rightarrow L^2}(h)$ and it has the principal symbol depending only on $y,x,\xi$. Note that by definition of $M_{\delta',2}$, we will encounter the composition of tangential operator with interior operator $\textrm{Op}_h(\phi_2^2)$. Since $\phi_2$ has support far away form $y=0$ and $\eta=0$, the symbolic calculus still valid thanks to Lemma \ref{compositionlawchapter1}. Therefore $M_{\delta',2}=O_{L^2\rightarrow L^2}(h)$ is an interior operator. Finally, we obtain that
	$$ N_{\delta'}=N_{\delta',3}+\sum_{j=1}^2(B_{\delta',j}^*B_{\delta',j}+H_{\delta',j})+O_{L^2(Y_+)\rightarrow L^2(Y_+)}(h).
	$$
	Combining all the analysis above and applying Proposition \ref{integrating by part2chapter1}, we have
	\begin{equation}\label{desired controlchapter1}
	\begin{split}
	\sum_{j=1}^2\|B_{\delta',j}u\|_{L^2(Y_+)}^2&\leq o(1)+
	\frac{2}{h}\left|\Im([A_{\delta},h\mathrm{d}]q|u)_{Y_+}\right|
	+\frac{1}{h}\left|\Im(q|h\mathrm{d}^*(A_{\delta}u))_{Y_+}\right|\\
	=&o(1)+|( q|\Upsilon_1 u)_{Y_+}|+|(\Upsilon_2q|u)_{Y_+}|
	\end{split}
	\end{equation}
	where $\Upsilon_1,\Upsilon_2$ are compactly supported matrix-valued tangential operators with principal symbols vanishing outside $\mathrm{supp}(a_{\delta})$.
	
	To finish the proof, we need show that the right hand side of \eqref{desired controlchapter1} is $o(1)$ as $h\rightarrow 0$. Pick $\chi\in C_c^{\infty}(\mathbb{R})$ such that $\chi(s)\equiv 1$ if $0\leq s\leq \frac{1}{2}$ and $\chi(s)\equiv 0$ is $s\geq 1$. Let $\chi_{\epsilon}(y,x,\xi)=\chi(\epsilon^{-1}r(y,x,\xi))$.
	Denote by 
	$$\mathrm{I}_{h,\epsilon}=\left|(\Upsilon_1\varphi\mathrm{Op}_h(\chi_{\epsilon})\varphi_1 u| q)_{Y_+}\right|, 
	\mathrm{II}_{h,\epsilon}=\left|(\Upsilon_1(1-\varphi\textrm{Op}_h(\chi_{\epsilon})\varphi_1)u|q)_{Y_+}\right|.
	$$
	The treatment of $\mathrm{I}_{h,\epsilon}$ is exactly the same as in the diffractive case, so we have
	$$ \lim_{\epsilon\rightarrow 0}\limsup_{h\rightarrow 0}\mathrm{I}_{h,\epsilon}=0.
	$$
	For $\mathrm{II}_{h,\epsilon}$, we may assume that the interval of the integration over $y$ variable is $[0,\epsilon]$, since for $y\geq\epsilon$ we can use the rapid decreasing of $q$ as in the treatment of $\mathrm{I}_{h,\epsilon}$. According to Lemma \ref{geometric1chapter1} and Lemma \ref{geometric2chapter1}, the measure of $\Upsilon_1(1-\varphi\textrm{Op}_h(\chi_{\epsilon})\varphi_1)u$ vanishes, since all the backward generalized rays starting from each point in $S_{\epsilon}$ will enter the small neighborhood of $\rho_0\in \mathcal{G}^{2,-}$ by at most reflection at boundary. From the propagation theorem in the hyperbolic case(Proposition \ref{hyperbolicchapter1}),  the proof of Proposition \ref{glidingchapter1} is complete.
\end{proof}
\begin{remark}
	We remark that as a consequence of Proposition \ref{glidingchapter1}, the measure of $q$(or $h\nabla q$) also vanishes along $\mathrm{exp}(tH_{-r_0})$ for $t\in[0,\sigma\delta)$.
\end{remark}

%%%%%%%%%%%%%%%%%%%%%%%%%%%%%%%%%%%%%%%%%%%%%%%%%%%%%%%%%%%%%%%
\subsection{high order contact}

In this subsection we will use a new coordinate system in a neighborhood $\widetilde{W}_k$ of  $\rho_k\in\mathcal{G}^k$ in $[0,\epsilon_0]\times T^*\partial\Omega$:
$$ (y,\eta,x,\xi)\mapsto (y,\eta,z,\zeta),\quad z=(z_1,z'),\zeta=(\zeta_1,\zeta')
$$
with $ p=\eta^2-r,r=\zeta_1+yr_1(z,\zeta)+O(y^2),$
$ \zeta_1=r_0,$
where $ r_0=r|_{y=0},r_1=\partial_yr|_{y=0}.$
This is possible since
$ d_{x,\xi}r_0\neq 0,\textrm{ if }\xi\neq 0.
$
Along the generalized bicharacteristic curve $\gamma(s)$, $(z,\zeta)$ satisfies
$$ \dot{z}=-\partial_{\zeta}r(y(s),z(s),\zeta(s)),\quad \dot{\zeta}=\partial_zr(y(s),z(s),\zeta(s)).
$$
This implies that in $\widetilde{W}_k$, 
$\displaystyle{ -\dot{z_1}\sim 1>0, \textrm{ as } y\rightarrow 0,}$
and thus
$
s\mapsto z_1(s) $ is strictly decreasing. Moreover,
$ \dot{\zeta_1}\sim y\partial_{z_1}r_1$, as $y\rightarrow 0.$

Suppose now $k\geq 3$,  we have locally that
$$ \mathcal{G}^k:=\{(z,\zeta):\zeta_1=0,\partial_{z_1}^{l}r_1(z,\zeta)=0,\forall l\leq k-3,\partial_{z_1}^{k-2}r_1(z,\zeta)\neq 0 \}.
$$
Define
$ \Sigma_k:=\{(z,\zeta):\partial_{z_1}^{k-3}r_1(z,\zeta)=0,\partial_{z_1}^{k-2}r_1(z,\zeta)\neq 0\}.
$
From implicit function theorem, $\Sigma_k$ is locally a hypersurface and we can write it as
$$\Sigma_k=\{(z,\zeta):z_1=\Theta_k(z',\zeta)\}.
$$
$\mathcal{G}^k$ can be viewed locally as a closed subset of $\Sigma_k$.
Since the map $s\mapsto z_1(s)$ is bijective, we may assume that along each ray, $z_1(0)=\Theta_k(z'(0),\zeta(0))$, and 
$$ z_1(s)<\Theta_k(z'(s),\zeta(s)),s>0;\quad 
z_1(s)>\Theta_k(z'(s),\zeta(s)),s<0.
$$
We see that all the generalized rays are transversal to the codimension 2 manifold(locally) $\Sigma_k$.
Moreover, a ray passes $\Sigma_{k}$ if and only if $y(0)=0$ and $\zeta_1(0)=0$.
Now we define the set near $\rho_k$:
$$ \Sigma^{\pm}_k:=\{(y,\eta,z,\zeta)\in \mathrm{Car}(P)\cap \widetilde{W_k}:z_1\mp \Theta_k(z',\zeta)>0\}.
$$
Note that the gliding rays $\mathrm{exp}(sH_{-r_0})$ intersect transversally to $\Sigma_k$ and  $H_{-r_0}=-\partial_{z_1}$ inside $T^*\partial\Omega$. Thus we can re-parametrize the gliding flows by $z_1$. Moreover, $\Sigma^{\pm}_k\cap \mathcal{G}^j=\emptyset,\forall j\geq k,$ provided that we choose $\widetilde{W}_k$ small enough.  In other word, $z_1$ gives a foliation of $T^*\partial\Omega$ near $\Sigma_k$ for small $|z_1-\Theta_k(z',\zeta)|$.

The following proposition is a long time extension of Proposition \ref{glidingchapter1}, adapted to the notations introduced above.
\begin{proposition}\label{longtimeglidingchapter1}
	Suppose $\rho_0\in\mathcal{G}^{2,-}$ near $\rho_k\in \Sigma_k$ with coordinate $(z,\zeta)$, $z_1>\Theta_k(z',\zeta)$. Then there exists $\delta_0>0$, sufficiently small such that if
	$$ \{(y,x,\eta,\xi):0\leq y\leq \delta^2,(x,\xi)\in L^-(\delta,\delta^2;\rho_0)\}\cap \mathrm{supp }(\mu)=\varnothing$$
	for $0<\delta\leq \delta_0,$ then $\exp(sH_{-r_0})(\rho_0)\notin \mathrm{supp } (\mu)$ for any $s<z_1-\Theta_k(z',\zeta)$.
\end{proposition}

In other words, each generalized ray, issued from gliding set outside supp$(\mu)$ does not carry any singularity until it reaches some point in $\mathcal{G}^k$ for $k\geq 3$.

\begin{proof}
	The proof is purely topological. For each $\rho_0=(z,\zeta)\notin$ supp$(\mu)$ and $z_1>0$, let
	$ s_1:=\sup\{s:s\leq z_1-\Theta_k(z',\zeta), \exp(s'H_{-r_0})\notin \textrm{ supp }(\mu),\forall s'\in[0,s)\}.
	$
	The existence of $s_1$ is guaranteed by Proposition \ref{glidingchapter1}. It remains to show that $s_1=z_1-\Theta_k(z',\zeta)$. By contradiction, suppose $s_1<z_1-\Theta(z',\zeta)$, then the point $\rho_1=(z_1-s_1,z',\zeta)$ is in $\mathcal{G}^{2,-}$. We can apply Proposition \ref{glidingchapter1} again to obtain that for some small $\sigma_1>0$, 
	$\exp(\theta\sigma H_{-r_0})(\rho_1)\notin $supp $(\mu)$ for any $\theta\in[0,1]$. This is a contradiction of the choice of $s_1$. 
\end{proof}
As a consequence, we have
\begin{corollary}\label{longtimegliding1chapter1}
	Suppose $\rho_0\in \mathcal{G}^{2,-}$ and $\rho_0\notin  \mathrm{supp }(\mu)$. Let $\gamma(s)$ be the generalized ray passing $\rho_0$ with $\gamma(0)=\rho_0$. Then $\gamma(s)\notin \mathrm{supp }(\mu)$ for any $s\in[-s_0,s_0]$, provided that $\gamma|_{[-s_0,s_0]}\subset \mathcal{G}^{2,-}$. 
\end{corollary}
Combining the analysis near a diffractive point and a gliding point, we have already established the $k$-propagation property for $k=2$. We will argue by induction to prove $k$-propagation property for all $k\geq 3$. To this end, we need an intermediate step. Let us first introduce a notation 
$$
\Gamma(\rho_0;\delta):=\{(y,x;z,\zeta):0\leq y\leq\delta^2,(z,\zeta)\in L^{-}(\delta,\delta^2;\rho_0)\}
$$
and a definition
\begin{definition}[$k$-pre-propagation property]
	For $k\geq 2$, we say that k-pre-propagation property holds, if the following statement is true:
	
	For any $\rho_k\in\mathcal{G}^k$, there exists a neighborhood $V_k$ of $\rho_k$ in $T^*\partial\Omega$, and $0<\delta_k\ll \sigma_k\ll 1$, depending on $V_k$, such that for any $\displaystyle{\rho_0\in\big(\mathcal{G}^{2,-}\cup\bigcup_{3\leq j\leq k}\mathcal{G}^j\big)\cap V_k}$, if 
	$ \Gamma(\rho_0;\delta)\cap\mathrm{supp }(\mu)=\emptyset
	$
	for some $0<\delta<\delta_k$, then $\exp(sH_{-r_0})(\rho_0)\notin \mathrm{supp }(\mu)$ for all $s\in[0,\sigma_k\delta)$.
\end{definition}
The key step is the following inductive proposition.
\begin{proposition}\label{prepropagationchapter1}
	Suppose $k\geq 3$ and $(k-1)$-propagation property holds true, then $k$-pre-propagation property also holds true.	
\end{proposition}

We do some preparation before proving this proposition. Select a neighborhood $W_k$ of $\rho_k\in\mathcal{G}^k$ in $T^*\partial\Omega$ (and contained in $\widetilde{W}_k$) with compact closure such that in $W_k$ such that 
$ |\partial_{z_1}^{k-2}r_1(\rho)|\geq c_0>0
$ for all $\rho\in\ov{W_k}$.
By abusing the notation, we will refer $\mathcal{G}^k$ to be $\mathcal{G}^k\cap W_k$ in the sequel. According to the asymptotic behaviour of the flow $\mathrm{exp}(sH_{-r_0})$ as $s\rightarrow 0$,  we have for any given $(z_1=\Theta_k(z'_0,\zeta_0),z'_0,\zeta_0)\in\mathcal{G}^k$,
$$ r_1\circ\mathrm{exp}(sH_{-r_0})(z'_0,\zeta_0)=b_k(z'_0,\zeta_0)s^{k-2}+O(s^{k-1}),
$$  	
where $b_k\neq 0$ can be viewed as a function of points in $\mathcal{G}^k$.	
From compactness, we can choose $\sigma>0,\theta>0$ depending only on $\overline{W}_k$ such that for all $\rho\in\mathcal{G}^k$, 
$$ |b_k(\rho)|\geq\theta>0, |r_1\circ\exp(s^{H_{-r_0}})(\rho)|\geq \frac{1}{2}|b_ks^{k-2}|,\quad \forall s\in [-\sigma,0)\cup(0,\sigma].
$$
Now we define a smaller neighborhood $V_k$ of $\rho_k$ such that for any $\rho_0\in V_k$, and $\delta_k>0,\sigma_k>0$, $\mathrm{exp}(sH_{-r_0})(L^{\pm}(\delta_k,\delta_k^2;\rho_0))\subset W_k$ for all $|s|\leq \sigma_k\delta_k$. Moreover, $|r_1\circ\exp(sH_{-r_0})(\rho_0)|\leq \delta_k$. We also put $\widetilde{W}_k=[0,\delta_k^2]\times W_k$, $\widetilde{V}_k=[0,\delta_k^2]\times V_k$.

Choosing a cut-off $\tilde{a}_{\delta}\in C_c^{\infty}$ with $\tilde{a}_{\delta}\equiv 1$ near $\rho_k$, we define
$$ S_{\delta,\epsilon}:=\textrm{ supp }(\tilde{a}_{\delta})\cap \{y\leq\epsilon,r\geq \epsilon\}
$$
for any $0<\epsilon\ll \delta$ .
Note that near $S_{\delta,\epsilon}$ (thus near $\rho_k\in\mathcal{G}^{k},k\geq 3$) we have $|r_1|\leq \delta_k$, and this implies that $\zeta_1\gtrsim \epsilon$, near $S_{\delta,\epsilon}$. We divide the proof of Proposition \ref{prepropagationchapter1} into several lemmas. 
\begin{lemma}\label{geometric3chapter1}
	Given any generalized ray $\gamma(s)=(y(s),\eta(s),z(s),\zeta(s))$ with $\gamma(s_0)\in \Gamma(\rho_0;\delta)\cap\mathcal{G}^{2,-}$ and $\gamma(s_1)\in S_{\delta,\epsilon}$. Assume that $\gamma|_{[s_0,s_1]}\subset \mathrm{Car}(P)\cap\widetilde{W}_k$, then $\gamma(s)\notin \mathcal{G}^k$ for all $s\in[s_0,s_1]$.
\end{lemma}
\begin{proof}
	Take $\Gamma^+(\rho_0;\delta):=\Gamma(\rho_0;\delta)\cap \Sigma_k^+$ and identify points in $\Sigma_k^{\pm}$ as their projection to $(y,x,\xi)$. 
	Let $F_k$(may be empty) be the union of generalized rays issued from $\Gamma^+(\rho_0;\delta)$ which meet $\mathcal{G}^k$. Note that along both real trajectories $\gamma(s)$ and $\exp(sH_{-r_0})$, $s\mapsto z_1$ is strictly decreasing, it suffices to show that  $F_k\cap S_{\delta,\epsilon}\subset \Sigma_k^+$ since generalized rays intersect with $\Sigma_k$ transversally,.
	
	We argue by contradiction. Assume that some ray in $F_k$ satisfies 	$\gamma(s_0)\in \Gamma^+(\rho_0;\delta),$ $\gamma(0)\in\mathcal{G}^k,$ and $\gamma(s_1)\in S_{\delta,\epsilon}$ for $s_0<0<s_1$. Write $\mathrm{exp}(sH_{-r_0})(\gamma(0))=(\tilde{z}(s),\tilde{\zeta}(s))$, and
	$$ r_1\circ \mathrm{exp}(sH_{-r_0})(z'(0),\zeta(0))=r_1(\tilde{z}(s),\tilde{\zeta}(s))=b_k s^{k-2}+O(s^{k-1}), s\rightarrow 0,
	$$
	More precisely, we have
	$$ |b_k(z'(0),\zeta(0))|\geq \theta>0, |r_1(\tilde{z}(s),\tilde{\zeta}(s))|\geq \frac{1}{2}|b_ks^{k-2}|,\forall s\in [-\sigma,0)\cup(0,\sigma].
	$$ 
	After shrinking support of $a_{\delta}$ if necessary, we may assume that $s_1<\sigma$.	According to the parity of $k$ and the sign of $b_k$, there are several situations.
	
	If $b_k<0, $ then for any $k$ we have $\gamma(s)\in\mathcal{G}^{2,-}$ for all $s\in(0,\sigma)$. This is impossible since $r\circ\gamma(s_1)\geq \epsilon$. Otherwise $b_k>0$, in this case we have 
	$r_1(\tilde{z}(s),\tilde{\zeta}(s))\geq b_ks^{k-2}/2,$ for all $s\in (0,\sigma),
	$
	and 
	\begin{equation}\label{negativechapter1}
	\begin{split}
	&(\partial_{z_1}r_1)(\tilde{z}(s),\tilde{\zeta}(s))\\=&(\partial_{z_1}r_1)\circ \mathrm{exp }(sH_{-r_0}(z'(0),\zeta(0))\\
	=&-\partial_s\left(r_1\circ \mathrm{exp }(sH_{-r_0}(z'(0),\zeta(0))\right)\\
	=&-(k-2)b_ks^{k-3}+O(s^{k-2})\leq 0,\forall s\in[0,\sigma),
	\end{split}
	\end{equation}
	thanks to $\partial_{\zeta'}r|_{y=0}=\partial_zr|_{y=0}=0$. 
	Taking the difference with real trajectory $\gamma(s)=(y(s),\eta(s);z(s),\zeta(s))$, we have
	\begin{equation*}
	\begin{split}
	&\partial_{z_1}r_1(y(s),z(s),\zeta(s))-
	\partial_{z_1}r_1(\tilde{z}(s),\tilde{\zeta}(s))\\=&
	\left(\partial_{z_1}r_1(0,z(s),\zeta(s))-\partial_{z_1}r_1(\tilde{z}(s),\zeta(s))\right)
	+(\partial_{z_1}r_1(\tilde{z}(s),\zeta(s))-\partial_{z_1}r_1(\tilde{z}(s),\tilde{\zeta}(s)))\\
	+&(\partial_{z_1}r_1(y(s),z(s),\zeta(s))-\partial_{z_1}r_1(0,z(s),\zeta(s))).
	\end{split}
	\end{equation*}
	Using the fact that $(z(0),\zeta(0))=(\tilde{z}(0),\tilde{\zeta}(0))$ and $y(s)=O(s^2)$, we have
	$$\partial_{z_1}r_1(y(s),z(s),\zeta(s))-
	\partial_{z_1}r_1(\tilde{z}(s),\tilde{\zeta}(s))=O(s).
	$$
	This together with \eqref{negativechapter1} yield
	\begin{equation}\label{dynamical inequalitychapter1}
	\begin{split}
	&\dot{\zeta_1}\leq y\partial_{z_1}r_1(y(s),z(s),\zeta(s))+C_0y^2\leq C_0(y^2+ys), \quad \dot{y}=2\eta,\\
	& \eta^2=\zeta_1+yr_1(z,\zeta)+O(y^2),\quad (\zeta_1(0),y(0))=(0,0),
	\end{split}
	\end{equation}
	where the constant $C_0$ and the implicit constant inside the big $O$ only depends on supp$(a_{\delta})$. 
	
	Applying the formula
	$\displaystyle{H_p^ky(0)=2(H_{-r_0})^{k-2}r_1=2b_k(k-2)!>0}$ and Taylor expansion, we have
	\begin{equation}\label{verticalchapter1}
	\begin{split}
	&y(s)=\frac{2b_k}{k(k-1)}s^k+O(s^{k+1})\geq \frac{b_k}{k(k-1)}s^k,s\in (0,\sigma),\\
	&\dot{y}(s)=\frac{2b_k}{k-1}s^{k-1}+O(s^k)>\frac{b_k}{k-1}s^{k-1}>0,s\in(0,\sigma).
	\end{split}
	\end{equation} 
	Injecting in   \eqref{dynamical inequalitychapter1}, we have 
	$ \dot{\zeta_1}(s)\leq C_0(\epsilon^2+\epsilon s)
	$ for all $s>0$ small as long as $y(s)\leq\epsilon$ and $\gamma(s)\notin S_{\delta,\epsilon}$. 
	For these $s$,
	$$ \zeta_1(s)\leq C_0(\epsilon^2s+\epsilon s^2/2). $$ 
	Setting
	$ s_2=\inf\{0\leq s\leq s_1:\gamma(s)\in S_{\delta,\epsilon}\},
	$
	we know that along the flow,
	$ 2\sqrt{\epsilon}= 2\eta(s_2)=\dot{y}(s_2), $ and this implies that $s_2\sim \epsilon^{\frac{1}{2(k-1)}}
	$ since $y(s)>\epsilon$ if $s>\epsilon^{\frac{1}{2(k-1)}}$.
	
	In summary, we have
	$$ \epsilon\leq r\circ \phi_{s_2}\leq 2C_0 \epsilon^{1+\frac{1}{2(k-1)}}+\delta_k \epsilon+C_1\epsilon^2.
	$$
	However, this contradicts to $ r=\zeta_1+yr_1+O(y^2),$ provided that $\delta_k<1$, $\epsilon\ll \delta_k <1$. 
\end{proof}

\begin{lemma}\label{8.13}
	The conclusion of Proposition \ref{prepropagationchapter1} holds if $\rho_0\in\mathcal{G}^{2,-}$
\end{lemma}
\begin{proof}
	Adapting the notations and argument in the proof of Proposition \ref{glidingchapter1}, we have 
	
	\begin{equation}\label{desired control1chapter1}
	\begin{split}
	\sum_{j=1}^2\|B_{\delta,j}u\|_{L^2(Y_+)}^2&\leq o(1)+
	\frac{2}{h}\left|\Im([A_{\delta},h\mathrm{d}]q|u)_{Y_+}\right|
	+\frac{1}{h}\left|\Im(q|h\mathrm{d}^*(A_{\delta}u))_{Y_+}\right|.
	\end{split}
	\end{equation}	
	The goal is to show that the last two terms on the right hand side tend to $0$ as $h\rightarrow 0$. 
	
	We denote by $\tilde{\gamma}(s)$ the gliding ray $\mathrm{exp}(sH_{-r_0})$ such that $\tilde{\gamma}(s_0)=\rho_0$ for some $s_0<0$. Suppose $\tilde{\gamma}(0)=\rho\in\mathcal{G}^k$ for some $k\geq 3$ and $\tilde{\gamma}(s)\in\mathcal{G}^{2,-}$ for $s\in(s_0,0)$. In view of Corollary \ref{longtimegliding1chapter1}, we may assume that $\rho_0$ is close enough to $\rho$, and $|s_0|$ is small. 
	Pick $\chi\in C_c^{\infty}(\mathbb{R})$ such that $\chi(s)\equiv 1$ if $0\leq s\leq \frac{1}{2}$ and $\chi(s)\equiv 0$ if $s\geq 1$. For any $\epsilon>0$, let $\chi_{\epsilon}(y,x,\xi)=\chi(\epsilon^{-1}r(y,x,\xi))$.
	Let
	$$\mathrm{I}_{h,\epsilon}=\frac{2}{h}\left|(h\mathrm{d}^*(\varphi\textrm{Op}_h(\chi_{\epsilon})\varphi_1u)|q)_{Y_+}\right|, 
	\mathrm{II}_{h,\epsilon}=\frac{2}{h}\left|(h\mathrm{d}^*(1-\varphi\textrm{Op}_h(\chi_{\epsilon})\varphi_1)u|q)_{Y_+}\right|.
	$$
	The treatment of $\mathrm{I}_{h,\epsilon}$ is exactly the same as in the diffractive case, we have
	$$ \lim_{\epsilon\rightarrow 0}\limsup_{h\rightarrow 0}\mathrm{I}_{h,\epsilon}=0.
	$$
	For $\mathrm{II}_{h,\epsilon}$, again, we only concern about the integration over $[0,\epsilon]$ in $y$ variable.
	From Lemma \ref{geometric3chapter1}, any ray entering $S_{\delta,\epsilon}$ can at most pass $\mathcal{G}^j$ for $j<k$. Applying $(k-1)$-propagation property, we deduce that for any cut-off $\varphi_{\epsilon}$ with supp$(\varphi_{\epsilon})\subset S_{\delta,\epsilon}$,  $\mathrm{supp}(\varphi_{\epsilon})\cap $ supp$(\mu)=\emptyset$. Therefore
	$$ \lim_{h\rightarrow 0}\mathrm{II}_{h,\epsilon}=0
	$$
	for any $\epsilon>0$. This completes the proof of Lemma \ref{8.13}.
\end{proof}

%%%%%%%%%%%%%%%%%%%%%%%
\begin{lemma}\label{8.14}
	The conclusion of Proposition \ref{prepropagationchapter1} holds if $\rho_0\in\mathcal{G}^{j}$ for some $3\leq j\leq k$.
\end{lemma}
\begin{proof}	
	Taking a micro-local cut-off $\psi(y,x,\xi)$ with support near $\rho_0$,  we have $$\|\varphi\textrm{Op}_h(\psi)\varphi_1u\|_{L^2(Y_+)}=o(1)$$ from the assumption that $\rho_0\notin $ supp$(\mu)$. Note that along the flow of $H_{-r_0}$ and on  supp$(1-\psi)\cap V_k$ we have
	$ |r_1(0,x,\xi)|\geq c(\psi,\delta)>0.
	$
	Hence from Corollary \ref{longtimegliding1chapter1}, if $\mathrm{exp}(tH_{-r_0})(\rho_0)\in \mathcal{G}^{2,-}$ for all $t\in(0,\sigma\delta)$, and then  $\mathrm{exp}(tH_{-r_0})(\rho_0)\notin$ supp$(\mu)$.  Otherwise $\mathrm{exp}(tH_{-r_0})(\rho_0)\in\mathcal{G}^{2,+}$ for all $t\in(0,\sigma\delta)$, we claim that we still have $\mathrm{exp}(tH_{-r_0})(\rho_0)\notin $supp$(\mu)$ from geometric consideration.
	
	Indeed, by considering the backward generalized ray, we conclude that for any $s_0\in (0,\sigma_k\delta)$, there exists $\rho\in\widetilde{W_k}$, so that $\gamma(s_0,\rho)=\mathrm{exp}(s_0H_{-r_0})(\rho_0)$ where $\gamma(s,\rho)$ is the generalized ray issued from $\rho$. From this fact we must have $\gamma(s,\rho)\notin \mathcal{G}^k$ for $s\in(0,s_0)$, since any ray intersecting with $\mathcal{G}^k$ will enter $T^*\Omega$ or $\mathcal{G}^{2,-}$ immediately, provided that the neighborhood $W_k$ is chosen to be small enough. By $(k-1)$-propagation property, if suffices to show that $\rho\notin\mathrm{supp}(\mu)$.
	
	Therefore, by definition of $\Gamma(\rho_0;\delta)$,  we only need to show that
	$$ \rho\in\{(y,z,\zeta):0\leq y\leq \delta^2,|(z,\zeta)-\rho_0|\leq\delta^2\}.
	$$
	We will prove this by comparing two flows $\mathrm{exp}(sH_{-r_0})(\rho_0)=(\tilde{z}(s),\tilde{\zeta}(s))$ and $\gamma(s,\rho)=(y(s),\eta(s),z(s),\zeta(s))$.
	Taking the difference of the two, we have
	\begin{equation*}
	\begin{split}
	\frac{d}{ds}\left(z_1(s)-\tilde{z}_1(s)\right)=&-\partial_{\zeta_1}r(y(s),z(s),\zeta(s))+\partial_{\zeta_1}r(0,\tilde{z}(s),\tilde{\zeta}(s))
	=O(y(s)),
	\end{split}
	\end{equation*} 
	\begin{equation*}
	\begin{split}
	\frac{d}{ds}(z'(s)-\tilde{z'}(s))=O(y(s)),\;
	\frac{d}{ds}(\zeta(s)-\tilde{\zeta}(s))=O((y(s)),\;
	\frac{dy}{ds}=2\eta(s).
	\end{split}
	\end{equation*} 
	Note that $|\eta|^2=|r|=O(1)$ and $y(s_0)=0,\tilde{z}(s_0)=z(s_0),\tilde{\zeta}(s_0)=\zeta(s_0)$, we have 
	$$ y(s)\leq C(s-s_0)^2\textrm{ for all } s\in [0,s_0].
	$$  
	Hence $y(0)\leq C\sigma_k^2\delta^2<\delta^2$, provided that $\sigma_k^2<1/C$. Moreover,
	$$ |(z(0),\zeta(0))-\rho_0|\leq Cs_0^3\leq C\sigma_k^3\delta^3\leq \delta^2.
	$$
This completes the proof of Lemma \ref{8.14} as well as Proposition \ref{prepropagationchapter1}.
\end{proof}

\begin{proposition}\label{8.15}
	\label{induction stepchapter1}
	Suppose that (k-1)-propagation property holds. Then k-pre-propagation property implies k-propagation property.
\end{proposition}
\begin{proof}
	Up to re-parameter the flow, we may assume that $\rho_0\in\mathcal{G}^k$ and $\gamma(s)$ is the generalized ray such that $\gamma(0)=\rho_0$. We also denote $\gamma(s)$ by $\gamma(s,\rho_0)$ in view of flow map. Suppose $\gamma(s_0)\notin $supp$(\mu)$ for some $s_0<0$ and $\gamma|_{[s_0,0)}\cap$supp$(\mu)=\emptyset$. Our goal is to show that $\rho_0\notin $supp$(\mu)$.  Let $\sigma_{k-1}>0$ be the required length in the definition of $(k-1)$-propagation property. 
	
	Let $\delta_k>0,\sigma_k>0$ and $V_k$, neighborhood of $\rho_0\in\mathcal{G}^k$ in $T^*\partial\Omega$ and $\widetilde{V}_k$, neighborhood of $\rho_0$ in $[0,\epsilon_0]\times T^*\partial\Omega$, as in the definition of $k$-pre-propagation property which satisfy the conditions in the paragraph in front of Lemma \ref{geometric3chapter1}. Note in particular that we have $V_k\cap\mathcal{G}_j=\emptyset$ for all $j>k$. Without loss of generality, we may assume that $|s_0|<\min\{\sigma_{k-1},\sigma_k\}$ and $\gamma(s_0)\in \widetilde{V}_k$, since otherwise we can choose $s'_0<0, |s'_0|$ small enough and replace $\gamma(s_0)$ by $\gamma(s'_0)$.
	
	Let $\Gamma_0\subset\widetilde{V}_k$ be a neighborhood of $\gamma(s_0)$ so that $\Gamma_0\cap$ supp$(\mu)=\emptyset$. For $\delta_1>0$ small with $\delta_1\ll\sigma_k$, we set $\rho_1=\mathrm{exp}\left(-\frac{\sigma_k\delta_1}{2}H_{-r_0}\right)(\rho_0)$ and define
	$$ U_{\delta_1}:=\{\rho=(y,\eta,z,\zeta)\in \mathrm{Car}(P):0\leq y\leq \delta_1^2,|(z,\zeta)-\rho_1|\leq\delta_1^2\}.
	$$
	From continuous dependence of the generalized bicharacteristic flow, we have 
	$$ U_{\delta_1}\subset \gamma(s_0,\Gamma_0),\quad \textrm{provided that }\delta_1 \textrm{ small enough }.
	$$
	Now we claim that for possibly smaller $\delta_1>0$, we have 
	$$ \gamma(s_1,U_{\delta_1})\cap \bigcup_{j\geq k}\mathcal{G}^j=\emptyset,\quad \forall s_1\in(s_0,0).
	$$ 
	Indeed, it suffices to prove that $\gamma(s_1,U_{\delta_1})\cap\mathcal{G}^k=\emptyset$ since there are no point of $\mathcal{G}^j$ in $\widetilde{V}_k$ for $j>k$. Firstly, from the transversality of the flow $\mathrm{exp}(sH_{-r_0})$ and $\Sigma_k$, we deduce that at $\rho_1$, $z_1>\Theta_k(z',\zeta)$. By choosing $\delta_1$ smaller, there exists $\epsilon_1>0$, such that for all $\rho\in U_{\delta_1}$, 
	$ z_1> \Theta_k(z',\zeta)+\epsilon_1
	$
	holds. In particular, $U_{\delta_1}\subset\Sigma_k^+$.
	We calculate
	\begin{equation*}
	\begin{split}
	\frac{d}{ds}\Theta_k(z'(s),\zeta(s))=&\frac{\partial\Theta_k}{\partial z'}\frac{dz'}{ds}+\frac{\partial\Theta_k}{\partial\zeta}\frac{d\zeta}{ds}\\
	=&-\frac{\partial\Theta_k}{\partial z'}\frac{\partial r}{\partial \zeta'}+\frac{\partial\Theta_k}{\partial\zeta}\frac{\partial r}{\partial z}.
	\end{split}
	\end{equation*}
	Note that in $\widetilde{V_k}$, we can write $r=\zeta_1+yr_1(z,\zeta)+O(y^2)$, hence 
	$$ \frac{d}{ds}\Theta_k(z'(s),\zeta(s))=O(y(s)).
	$$
	
	Next we argue by contradiction, assume that for some $s_1\in(s_0,0)$ and $\rho\in U_{\delta_1}$ we have $\gamma(s_1,\rho)\in\mathcal{G}^k$ and $\gamma(s,\rho)\notin \mathcal{G}^k$ for any $0>s>s_1$. In this case we have $|y(s)|\leq C|s-s_1|$ for all $s\in[s_1,0]$. Therefore we must have
	$$\left|\frac{d}{ds}\Theta_k(z'(s),\zeta(s))\right|\leq C|s-s_1|.$$
	Combining with $\dot{z_1}\sim -1$, we have
	\begin{equation*}
	\begin{split}
	\Theta_k(z'(s_1),\zeta(s_1))\leq & \Theta_k(z'(0),\zeta(0))+C\int_{s_1}^0|s-s_1|ds\\
	< &z_1(0)+Cs_1^2\\
	=&z_1(s_1)+\int_{s_1}^0\frac{dz_1}{ds}ds+Cs_1^2\\
	\leq &z_1(s_1)-C_1|s_1|+Cs_1^2\\
	\leq &z_1(s_1),
	\end{split}
	\end{equation*}
	provided that $|s_0|$ is chosen to be small enough. This implies that $\gamma(s_1,\rho)\in\Sigma_k^+$, which is a contradiction.
	
	From $(k-1)$-propagation property, we know that $U_{\delta_1}\cap$ supp$(\mu)=\emptyset$. Therefore, applying $k$-pre-propagation property with respect to $\rho_1$ and $U_{\delta_1}$, we deduce that $\rho_0\notin$ supp$(\mu)$, and this completes the proof of Proposition \eqref{8.15}.
\end{proof}

%%%%%%%%%%%%%%%%%%%%%%%%%%%%%%%%%%%%%%%%%%%%%%%%%%%%%%%%%%%%%%%%%%

	\section*{A.Proof of Lemma \ref{hiddenregularitychapter1}}
	%%%%%%%%%%%%%%%%%%%%%%%%%%%%%%%%%%%%%%%%%%%
	\begin{proof}[Proof of Lemma \ref{hiddenregularitychapter1}]
		The first assertion follows from $h\mathrm{div }u=0$ and Dirichlet boundary condition, while we apply a multiplier method to prove the second. From the geometric assumption on $\Omega$, we can find a vector field $L\in C^1(\ov{\Omega})$ such that $L|_{\partial\Omega}=\nu$(see \cite{bookControl}, page 36). In global coordinate system, we write $ \displaystyle{L=L_j(x)\partial_{x_j}}.$ By using the equation, we have
		\begin{equation*}
		\begin{split}
		\int_{\Omega}Lu\cdot fdx=&\int_{\Omega} Lu\cdot (-h^2\Delta u-u+h\nabla q)dx,\\
		-\int_{\Omega}Lu\cdot udx
		=&-\int_{\Omega}L_j(x)\partial_{x_j} u^iu^idx\\
		=&-\int_{\Omega}
		\partial_{x_j}\left(L_j(x)u^i\right)u^idx+\int_{\Omega}\textrm{div }L(x)|u|^2dx\\
		=&\int_{\Omega}L_j(x)u^i(x)\partial_{x_j}u^idx+
		\int_{\Omega}\textrm{div }L(x)|u|^2dx\\
		=&\int_{\Omega}Lu\cdot udx+
		\int_{\Omega}\textrm{div }L(x)|u|^2dx,
		\end{split}
		\end{equation*}
		and thus
		$\displaystyle{\int_{\Omega} Lu\cdot udx=-\frac{1}{2}\int_{\Omega}
			\textrm{div }L(x)|u|^2dx=O(1).}$
		Next we calculate
		\begin{equation*}
		\begin{split}
		h\int_{\Omega}Lu\cdot \nabla qdx
		&=-h\int_{\Omega}u^i \partial_{x_j}\left(L_j\partial_{x_i} q\right)dx\\
		&=-h\int_{\Omega}u\cdot L(\nabla q)dx-h\int_{\Omega}
		(\textrm{div }L(x))u\cdot\nabla qdx\\
		&=-h\int_{\Omega}u\cdot[L,\nabla]qdx-h\int_{\Omega}\textrm{div }L(x)u\cdot\nabla qdx\\
		&=O(1),
		\end{split}
		\end{equation*}
		
		\begin{equation*}
		\begin{split}
		-h^2\int_{\Omega}Lu^i\Delta u^idx=&-h^2\int_{\partial\Omega} \left|\partial_{\nu}u^i\right|^2d\sigma +h^2\int_{\Omega}\nabla L(\nabla u^i,\nabla u^i)dx \\
		&+h^2\int_{\Omega}
		L_j(x)\partial^2_{x_jx_k} u^i\partial_{x_k} u^i\\
		=&-h^2\int_{\partial\Omega} \left|\partial_{\nu}u^i\right|^2d\sigma +h^2\int_{\Omega}\nabla L(x)(\nabla u^i,\nabla u^i)dx
		\\&+h^2\int_{\Omega}\partial_{x_j}\left(L_j\partial_{x_k}u^i\right)\partial_{x_k}u^idx
		-h^2\int_{\Omega}\textrm{div }L(x)\nabla u^i\cdot\nabla u^i(x)dx,
		\end{split}
		\end{equation*}
		\begin{equation*}
		\begin{split}	h^2\int_{\Omega}\partial_{x_j}\left(L_j\partial_{x_k}u^i\right)\partial_{x_k}u^idx&=
		h^2\int_{\partial\Omega}L\cdot\nu \left|\partial_{\nu} u^i\right|^2d\sigma -h^2\int_{\Omega}L_j(x)\partial_{x_k} u^i
		\partial^2_{x_jx_k} u^idx,
		\end{split}
		\end{equation*}
		$$-h^2\int_{\Omega}Lu^i\Delta u^idx
		=-\frac{h^2}{2}\int_{\partial\Omega}\left|\partial_{\nu}u^i\right|^2d\sigma+
		\int_{\Omega}\nabla L(x)(h\nabla u^i,h\nabla u^i)dx-
		\frac{h^2}{2}\int_{\Omega}\textrm{div }L(x)|\nabla u^i|^2dx.
		$$
		Observing that
		$ \int_{\Omega}Lu\cdot fdx=o(1),
		$
		we have
		$$ \int_{\partial\Omega}\left|h\partial_{\nu}u\right|^2d\sigma=O(1).
		$$
	\end{proof}
	
	%%%%%%%%%%%%%%%%%%%%%%%%%%%%%%%%%%%%%%%%%%%%%%%%%%%%%%%%%%%%%%%%%%%
	\section*{B.Standard elliptic theory}
	The differential operator is given by
	$$P_h=\mathrm{Op}_h(\eta^2+\lambda(y,x',\xi')^2-1+hm(y,x',\eta,\xi')),$$ where the principal symbol $p=\eta^2+\lambda^2-1$ is scalar while $m$ is matrix valued. When micro-locally near the region $p>0$, we want to construct the parametrix of the inverse of $P$. Denote by $U$ the turbulent neighborhood (two sided) of $\partial\Omega$. Take $\varphi\in C_c^{\infty}(U),\chi_0\in C^{\infty}(\mathbb{R}^{d-1})$ and the support of $\varphi$ is contained in a coordinate patch near the boundary. We put
	$$ E^0:=\mathrm{Op}_h\big(\frac{\chi_0(\xi')\varphi(y,x')}{\eta^2+\lambda(y,x',\xi')^2-1}\big),
	$$
	and we define matrix valued PdO $E^l$, $l\geq 1$ inductively via
	\begin{equation}\label{ellipticconstructionchapter1}
	\begin{split}
	& E^1\times p=-\sum_{|\alpha|=1}\frac{1}{i}\partial^{\alpha}_{\xi',\eta}E^0\times\partial_{x',y}^{\alpha}p-E^0\times m,\\
	&E^n\times p=-\sum_{|\alpha|+k=n,k\neq n}\frac{1}{i^{|\alpha|}}\partial^{\alpha}_{\xi',\eta}E^k\times\partial_{x',y}^{\alpha}p-\sum_{|\alpha|+k=n-1}\frac{1}{i^{|\alpha|}}\partial^{\alpha}_{\xi',\eta}E^k\times\partial_{x',y}^{\alpha}m.
	\end{split}
	\end{equation}
	For any $N\in\mathbb{N}$, we set
	$$ E_N=\sum_{j=0}^Nh^jE^j,
	$$
	and then
	$$ E_N\circ P=\chi_0(hD_{x'})\varphi(y,x')\textrm{Id}+R_N,\quad \|R_N\|_{L^2\rightarrow L^2}=O(h^{N+1}).
	$$
	\begin{proposition}\label{frequencyconcentrate chapter1}
		The sequence of solutions $(u_k)$ is $h_k$-oscillating in the following sense:
		$$ \lim_{R\rightarrow\infty}\limsup_{k\rightarrow\infty}
		\int_{|\xi|\geq Rh_k^{-1}}
		|\widehat{\varphi u_k}(\xi)|^2d\xi=0\quad \forall \psi\in C_c^{\infty}(\Omega),
		$$
		$$ \lim_{R\rightarrow\infty}\limsup_{k\rightarrow\infty}
		\int_0^{\epsilon_0}dy\int_{|\xi'|\geq Rh_k^{-1}}
		|\widehat{\varphi u_k}(y,\xi')|^2d\xi'=0,\quad \forall \psi\in C_c^{\infty}(\ov{\Omega}),
		$$
		where in the second formula, the support of $\varphi$ is contained in some local coordinate patch near the boundary, and the Fourier transform is only taken for the $x'$ direction.
	\end{proposition}
	\begin{proof}
		We drop the subindex $k$ in the proof.	For the first formula, one can use the equation of $u$ to obtain
		$$ (-h^2\Delta-1)(\varphi u)=g=O_{L^2}(1),
		$$
		and 
		$$\int_{|\xi|\geq Rh^{-1}}|\widehat{\varphi u}(\xi)|^2d\xi
		\leq \int_{h|\xi|\geq R}
		\frac{|\widehat{g}(\xi)|^2}{|h^2|\xi|^2-1|^2}d\xi\leq \frac{C}{(R^2-1)^2}. 
		$$
		
		For the second formula, it will be sufficient to show that
		$$ \lim_{R\rightarrow\infty}
		\limsup_{k\rightarrow\infty}\big\|\big(1-\chi\big(\frac{hD_{x'}}{R}\big)\big)(\varphi u)\big\|_{L^2}=0
		$$
		for some $\chi\in C_c^{\infty}(-1,1)$. We write $w=u\mathbf{1}_{y\geq 0},\tilde{g}=g\mathbf{1}_{y\geq 0},
		v=h\partial_y u|_{y=0}=O_{L^2(y=0)}(1)$. We apply the parametrix construction above with $\chi_0(\xi')=1-\chi\big(\frac{\xi'}{R}\big)$. Let $e_N(y,x',\eta,\xi')$ be the symbol of the operator $E_N$, which is meromorphic in $\eta$ with poles $\eta_0^{\pm}=\pm i\sqrt{\lambda^2(y,x',\xi')^2-1}$. 
		Moreover,
		\begin{equation}\label{symbolboundchapter1}
		|\partial^{\alpha}e_N(y,x',\eta,\xi')|\leq \frac{C_{N,\alpha}}{R}.
		\end{equation}
		We take $\widetilde{\varphi}$ to be a slight enlargement of $\varphi$ such that $\widetilde{\varphi}\varphi=\varphi$. Then
		$$ E_N\circ\widetilde{\varphi}Pw=\big(1-\chi\big(\frac{hD_{x'}}{R}\big)\big)(\varphi w)+R_Nw.
		$$ 
		From the jump formula, 
		$ \displaystyle{Pw=\widetilde{g}+hv(x')\otimes \delta_{y=0}.}
		$ We have
		$$ \big(1-\chi\big(\frac{hD_{x'}}{R}\big)\big)(\varphi w)=E_N\big(\widetilde{\varphi}\widetilde{g}+\widetilde{\varphi}hv\otimes\delta_{y=0}\big)+O_{L_{y,x'}^2}(h^{N+1}).
		$$
		From symbolic calculus, 
		$$ \|E_N(\widetilde{\varphi}\widetilde{g})\|_{L_{y,x'}^2}\leq \sum_{|\alpha|\leq Cd}\sup_{(y,x',\eta,\xi')\in\mathbb{R}^{2d}}|\partial^{\alpha}e_N(y,x',\eta,\xi')|+Ch,
		$$
		and it vanishes after the taking the limit $h\rightarrow 0$ and then $R\rightarrow\infty$, thanks to \eqref{symbolboundchapter1}. We next write
		\begin{equation*}
		\begin{split}
		E_N(\widetilde{\varphi}hv\otimes\delta_{y=0})=&\frac{\pi}{(2\pi h)^{d-1}}\int_{\mathbb{R}^{2(d-1)}}e^{\frac{i(x'-z')\cdot\xi'}{h}}\omega_N(y,x',\xi')\widetilde{\varphi}(0,z')v(z')dz'd\xi',
		\end{split}
		\end{equation*}
		with
		$$ \omega_N(y,x',\xi')=\int_{-\infty}^{\infty}e_N(y,x',\eta,\xi')e^{\frac{iy\eta}{h}}d\eta.
		$$
		To calculate $\omega_N$ for $y>0$, we deform the contour of integral in $\eta$ in the half plane $\Im\eta>0$. From the Residue formula, we have
		$$ \omega_N(y,x',\xi')=2\pi i\textrm{Res}(e_N(y,x',\eta,\xi');i\eta_0^{+})e^{\frac{iy\eta_0^{+}}{h}}.
		$$ 
		The principal symbol of $\omega_N$ is given by
		$$ \pi\exp\big(-\frac{yQ(y,x',\xi')}{h}\big)\frac{\varphi(y,x')\big(1-\chi\big(\frac{\xi'}{R}\big)\big)}{2Q(y,x',\xi')},\quad Q(y,x',\xi')=\sqrt{\lambda(y,x',\xi')^2-1}.
		$$
		Therefore
		\begin{equation*}
		\begin{split}
		&\limsup_{h\rightarrow 0}\|E_N(\widetilde{\varphi}2hv\otimes\delta_{y=0})\|_{L_{y,x'}^2}\\ \leq 
		&C_{N,d}\int_0^{\infty}\sum_{|\alpha|\leq Cd}\sup_{(x',\xi')}|\partial^{\alpha}_{x',\xi'}\omega_N(y,x',\xi')|\|v\|_{L_{x'}^2}dy
		\leq \frac{C}{R},
		\end{split}
		\end{equation*} 
		where we have used the point-wise estimate
		$$ |\partial_{x',\xi'}^{\alpha}\omega_N(y,x',\xi')|\leq \frac{C_{\alpha}e^{-\frac{y\sqrt{R^2-1}}{2Rh}}}{\sqrt{R^2-1}}\big(1+\big(\frac{y}{h}\big)^{|\alpha|}\big).
		$$
	\end{proof}
	Given $\chi(y,x',\xi')\in C_c^{\infty}([0,\epsilon_0)\times\mathbb{R}^{2d-2})$, the following proposition can be deduced in the same manner.
	\begin{proposition}\label{etaellipticchapter1}
		Let $w_k=\chi(y,x',hD_{x'})(\varphi u_k),\underline{w_k}=\mathbf{1}_{y\geq 0}u_k$. Then for $\chi_1\in C_c^{\infty}(\mathbb{R}), 0\leq\chi_1\leq 1, \chi_1=1$ near $0$, we have
		$$\lim_{R\rightarrow\infty}\limsup_{k\rightarrow\infty}\big\|\big(1-\chi_1\big(\frac{h_kD_y}{R}\big)\big)\underline{w_k}\big\|_{L^2(\mathbb{R}^d)}=0.
		$$ 
	\end{proposition}
	\begin{proof}
		We have $\displaystyle{P\underline{w}=\underline{g}+hv(x')\otimes\delta_{y=0}}$
		with $\|\underline{g}\|_{L_{y,x'}^2}=O(1),\|v\|_{L_{x'}^2}=O(1).$ Note that the functions $\underline{g},\underline{v}$ here may not coincide with the functions in the proof of Proposition \ref{frequencyconcentratechapter1}. We define 
		$$ E_N=\sum_{j=0}^Nh^jE^j+R_N,\quad \|R_N\|_{L^2\rightarrow L^2}=O(h^{N+1}),
		$$
		with $\displaystyle{E^0=\mathrm{Op}_h\big(\frac{\widetilde{\varphi}(y,x')\big(1-\chi_1\big(\frac{\eta}{R}\big)\big)}{\eta^2+\lambda(y,x',\xi')^2-1}\big)}\textrm{Id}$ and $E^l, l\geq 1$ as in \eqref{ellipticconstructionchapter1}. This implies that
		$$\big(1-\chi_1\big(\frac{hD_y}{R}\big)\big)(\varphi \underline{w})=E_N(\widetilde{\varphi}\underline{g}+2h\widetilde{\varphi}v\otimes\delta_{y=0})-R_N\underline{w}.
		$$
		Consequently, we have 
		$$ \lim_{R\rightarrow\infty}\limsup_{h\rightarrow 0}\|E_N(\widetilde{\varphi}\underline{g})\|_{L_{y,x'}^2}=0.
		$$
		\begin{equation*}
		\begin{split}
		&E_N(h\widetilde{\varphi}v\otimes\delta_{y=0})(y,x')\\
		=&\frac{\pi}{(2\pi h)^{d-1}}\int_{\mathbb{R}^{2(d-1)}}e^{\frac{i(x'-z')\xi'}{h}}a(y,x',\xi')\widetilde{\varphi}(0,z')v(z')dz'd\xi',\\
		&\textrm{ with } a(y,x',\xi')=\int_{-\infty}^{\infty}e^{\frac{iy\eta}{h}}e_N(y,x',\eta,\xi')d\eta.
		\end{split}
		\end{equation*}
		Observe that
		$$ \sup_{(x',\xi')}|\partial_{x',\xi'}^{\alpha}e_N(y,x',\eta,\xi')|\leq \frac{C_{\alpha}\left(1-\chi_1\left(\frac{\eta}{R}\right)\right)}{(1+\eta^2)(1+y^2)},
		$$
		and this implies that 
		$$ \lim_{R\rightarrow\infty}\limsup_{h\rightarrow 0}\|E_N(h\widetilde{\varphi}v\otimes\delta_{y=0})\|_{L_{y,x'}^2}=0.
		$$
	\end{proof}

	%%%%%%%%%%%%%%%%%%%%%%%%%%%%%%%%%%%%%%%%%%%%%%%%%%%%%%%%%%%%%%%%%%%%%%%%%%%%%%%%%%%%%%%%%%%
	\section*{C.Proof of technical results in section 3}
	
	\begin{proof}[Proof of Lemma \ref{exponentialsymbolchapter1}]
		The proof can be reduced to the point-wise estimate of the solution $F(y)$ of the ODE:
		$$ -h^2\frac{d^2F}{dy^2}+\lambda(y)^2F(y)=G(y),\quad F(0)=F'(0)=0,
		$$
		with $0<c_1\leq \lambda(y)^2\leq c_2,$ $G\in C^{\infty}([0,\infty)),$ and $ |G(y)|\leq Ce^{-\frac{cy}{h}}$ for all $y\geq 0$. By rescaling $z=\frac{y}{h}$,  it reduces to prove the exponential decay of the solution $W$ of the ODE:
		$$ -\frac{d^2W}{dz^2}+V(z)W(z)=g(z),\quad  W(0)=W'(0)=0,
		$$
		with $0<c_1\leq V(z)\leq c_2$, $g\in C^{\infty}([0,\infty))$ and $|g(z)|\leq Ce^{-cz}$ for all $z\geq 0$. 
		
		For this, we first notice that $W$ is smooth and in $H^{s}(\mathbb{R}_+)$ for all $s\geq 0$. To prove the exponential decay, we pick $\theta_{\epsilon}(z)=e^{\frac{2\delta_0z}{1+\epsilon z}}$ with $\delta_0>0$ to be chosen later. One observe easily that $0<\theta_{\epsilon}'(z)\leq  2\delta_0\theta_{\epsilon}(z)$ for $z\geq 0$. We multiply by $\theta_{\epsilon}W$ to the both sides of the equation and integrate it, then
		$$ \int_0^{\infty}\left((\theta_{\epsilon}W)'W'+\theta_{\epsilon}VW^2\right)dz=\int_0^{\infty}\theta_{\epsilon}WGdz.
		$$
		Notice that $\theta_{\epsilon}'W'W\geq -2\delta_0\theta_{\epsilon}|W||W'|\geq -\delta_0\theta_{\epsilon}(|W|^2+|W'|^2)$, by choosing $$\delta_0<\min\big\{\frac{c_1}{4},\frac{1}{4},\frac{c}{2}\big\},$$
		we have that
		$$ \int_0^{\infty}\theta_{\epsilon}\left(|W'|^2+|W|^2\right)dz\leq 2\|Ge^{2\delta_0z}\|_{L^2(\mathbb{R}_+)}\|W\|_{L^2(\mathbb{R}_+)},
		$$
		thanks to Cauchy-Schwartz and the fact that $\theta_{\epsilon}\leq e^{2\delta_0z}$, uniformly in $\epsilon$. From the dominating convergence theorem, we have 
		$We^{\delta_0z}\in L^2(\mathbb{R}_+)$ and $W'e^{\delta_0z}\in L^2(\mathbb{R}_+)$. Finally, by elliptic regularity, we have that $We^{\delta_0z}\in L^{\infty}(\mathbb{R}_+)$,$W'e^{\delta_0z}\in L^{\infty}(\mathbb{R}_+)$.
	\end{proof}

	\begin{proof}[Proof of Proposition \ref{parametrixpressurechapter1}]
		We choose $\varphi_2\in C_c^{\infty}(Y_+)$ such that
		$$ \varphi_1|_{\textrm{supp}(\varphi_2)}=1,\varphi_2|_{\textrm{supp}(\varphi)}=1.
		$$
		We first claim that 
		\begin{equation}\label{bdchapter1}
		\varphi_2\mathrm{Op}_h(\chi_{\delta_0}A_j)(\varphi_{1,0}q_0)=O_{L^2(\mathbb{R}_+^{d})}(1).
		\end{equation}
		Indeed, we can write
		$$ \varphi_2\mathrm{Op}_h(A_j\chi_{\delta_0})(\varphi_{1,0}q_0)=\varphi_2\mathrm{Op}_h(A_j)\varphi_1\varphi\mathrm{Op}_h(\chi_{\delta_0})(\varphi_{1,0}q_0)+h\varphi_2\mathrm{Op}_h(B_j\widetilde{\chi}_{\delta_0})(\varphi_{1,0}q_0)
		$$
		with $B_j\in \mathcal{E}_{\partial}^{-j}$ and $\widetilde{\chi}_{\delta_0}$ has similar support property as $\chi_{\delta_0}$. Thus from symbolic calculus, we have for each fixed $y>0$, 
		$$ \|\varphi_2\mathrm{Op}_h(A_j(y)\chi_{\delta_0}(y))(\varphi_{1,0}q_0)\|_{L^2(\mathbb{R}^{d-1})}^2\leq C_je^{-\frac{c_jy}{h}}h^{-1}\big(1+\big(\frac{y}{h}\big)^{n_j}\big),
		$$
		thanks to Lemma \ref{pressure L^2 boundchapter1}. Integrating over $y> 0$ yields \eqref{bdchapter1}.

		By taking $\mathrm{supp}(\chi_{\delta_0})$ small such that $\varphi_2\chi_{\delta_0}=\chi_{\delta_0}$, we have that 
		$$ \varphi_2\left(\mathrm{Op}_h(\chi_{\delta_0})(\varphi_1q)-\mathrm{Op}_h(\chi_{\delta_0}A)(\varphi_{1,0}q_0)\right)=w+O_{L^2(\mathbb{R}_+^d)}(h)
		$$
		with $$w=\varphi_2\left(\mathrm{Op}_h(\chi_{\delta_0})(\varphi_1q)-\mathrm{Op}_h(\chi_{\delta_0})\varphi_1\varphi_2\mathrm{Op}_h(A)(\varphi_{1,0}q_0)\right).$$ From Lemma \ref{pressure L^2 boundchapter1}, we have $w=O_{L^2(\mathbb{R}^{d}_+)}(1)$ and $hD_yw=O_{L^2(\mathbb{R}_+^d)}(1)$. The trace of $w$ satisfies
		$$ w|_{y=0}=\varphi_{2,0}\mathrm{Op}_h(\chi_{\delta_0}(1-\psi_{\delta_0}(\lambda_0)))(\varphi_{1,0}q_0)=O_{H^{\infty}(\mathbb{R}^{d-1})}(h^{\infty}),
		$$ 
		and $w$ satisfies the equation (we use $\varphi_2=\varphi_1\varphi_2$ here)
		\begin{equation}\label{approximateequationchapter1} P_0w=\varphi_1\left[P_0,\varphi_2\mathrm{Op}_h(\chi_{\delta_0})\right]\left(\varphi_1 q-\varphi_2\mathrm{Op}_h(A)(\varphi_{1,0}q_0)\right)+O_{H^{\infty}(\mathbb{R}^d)}(h^{\infty}).
		\end{equation}
		Notice that $\varphi_1q-\varphi_2\mathrm{Op}_h(A)(\varphi_{1,0}q_0)=O_{L^2(\mathbb{R}_+^d)}(1)$, micro-locally for $\lambda\geq \frac{\delta_0}{2}$, the right hand side of \eqref{approximateequationchapter1} is equal to $O_{L^2(\mathbb{R}_+^d)}(h)$ as well as $O_{H^1(\mathbb{R}_+^d)}(1)$. Multiply by $w=\varphi_1w$ to the both sides of \eqref{approximateequationchapter1} and integrate it, we have
		$$\int_0^{\infty}\int_{\mathbb{R}^{d-1}}w(y,x')P_0wdx'dy=O(h).
		$$
		Doing integration by part for the left hand side, we have
		\begin{equation*}
		\begin{split}
		\int_0^{\infty}\int_{\mathbb{R}^{d-1}}wP_0wdx'dy=&-\int_0^{\infty}\int_{\mathbb{R}^{d-1}}|h\partial_yw|^2dx'dy-\int_{\mathbb{R}^{d-1}}h^2(w\partial_yw)|_{y=0}dx'\\
		-&\int_0^{\infty}\int_{\mathbb{R}^{d-1}}\sum_{1\leq j,k\leq d-1}g^{jk}h\partial_jwh\partial_kwdx'dy+O(h).
		\end{split}
		\end{equation*}
		This implies that
		$$ \|hD_{y,x'}w\|_{L^2(\mathbb{R}_+^d)}+\|w\|_{L^2(\mathbb{R}_+^d)}=O(h^{1/2}).
		$$
		Using this smallness and redoing the integration by part, we can improve each bound in the procedure above and obtain that
		$$ \|hD_{y,x'}w\|_{L^2(\mathbb{R}_+^d)}+\|w\|_{L^2(\mathbb{R}_+^d)}=O(h^{3/4}).
		$$
		To conclude, we observe that
		\begin{equation}\label{calcullocalchapter1}
		\begin{split}
		\varphi h\partial_yw=&\varphi h\partial_y\mathrm{Op}_h(\chi_{\delta_0})(\varphi_1q)-\varphi h\partial_y\mathrm{Op}_h(\chi_{\delta_0})\varphi_2\mathrm{Op}_h(A)(\varphi_{1,0}q_0)\\
		=&\varphi\mathrm{Op}_h(\chi_{\delta_0})h\partial_y(\varphi_1 q)-\varphi\mathrm{Op}_h(\chi_{\delta_0})\varphi_2\mathrm{Op}_h(h\partial_y A)(\varphi_{1,0}q_0)\\
		+&O_{L^2(\mathbb{R}_{+}^d)}(h)\\
		=&\varphi\mathrm{Op}_h(\chi_{\delta_0})h\partial_y(\varphi_1 q)-\varphi\mathrm{Op}_h(\chi_{\delta_0}\lambda A)(\varphi_{1,0}q_0)\\
		+&O_{L^2(\mathbb{R}_+^d)}(h),
		\end{split}
		\end{equation}
		where we have used symbolic calculus and Lemma \ref{pressure L^2 boundchapter1} several times. Plugging into \eqref{approximateequationchapter1}, we have that $P_0w=O_{L^2(\mathbb{R}_+^d)}(h)$, $w|_{y=0}=O_{H^{\infty}(\mathbb{R}^{d-1})}(h^{\infty})$. We decompose $w=w_1+w_2$ with $P_0w_1=P_0w,w_1|_{y=0}=0$ and $P_0w_2=0, w_2|_{y=0}=w|_{y=0}$. From elliptic regularity of boundary value problem, we have $h^2w_1=O_{H^2(\mathbb{R}_+^d)}(h)$ and $h^2w_2=O_{H^2(\mathbb{R}_+^d)}(h^{\infty})$ and thus $h\partial_yw=O_{H^1(\mathbb{R}_+^d)}(1)$. From interpolation, we deduce that $h\partial_yw=O_{H^{\frac{2}{3}}(\mathbb{R}_+^d)}(h^{\frac{1}{4}})$. Observe that the error terms on the right hand side of \eqref{calcullocalchapter1} can be also bounded by $O_{H^1(\mathbb{R}_+^d)}(1)$ and thus $O_{H^{\frac{2}{3}}(\mathbb{R}_+^d)}(h^{\frac{1}{4}})$ by interpolation. This completes the proof.

	\end{proof}

	%%%%%%%%%%%%%%%%%%%%%%%%%%%%%%%%%%%%%%%%%%%%%%%%%%%%%%%%%%%%%%%%%%%%%%%%%%%%%%%%%%
	\section*{D.Construction of test functions}
	We first give the detailed construction of $a=a_0+a_1\eta$ used in the first step of the proof of Proposition\ref{propagation of supportchapter1}, which follows closely to \cite{bookHormander(III)}.
	\begin{proof}[Proof of Lemma \ref{test1chapter1} ]
		Given $\chi_1\in C_c^{\infty}(-2,2)$ with $\chi_1|_{(-1,1)}=1$ and
		$\chi_2\in C_c^{\infty}(-3,3)$ such that $\chi_2|_{(-2,2)}=1$. Consider the smooth functiiton $\chi_0(t)=e^{-1/t}\mathbf{1}_{t>0}$.  We work in the local coordinate $(y,x,\xi)$, and assume that $(0,x_0,\xi_0)\in \mathcal{G}^{2,+}$ with $|\xi_0|\sim 1$. Set $\phi=\phi_0+\phi_1\eta$ with
		$$ \phi_1(y,x,\xi)=\frac{1}{|\xi|}, \phi_0(y,x,\xi)=y^2+|x-x_0|^2+|\xi-\xi_0|^2.
		$$
		We calculate
		$$ H_p\phi=\eta\left(2\partial_y\phi_0-\{r,\phi_1\}\right)+\phi_1\partial_yr-\{r,\phi_0\}\geq 2c>0,
		$$
		provided that $|\eta|\leq c_0$ for some $c_0>0$ and $W_0$ is chosen small enough such that $\frac{\partial r}{\partial y}\geq 4c$ on it. The positivity then follows from the direct calculation:
		\begin{equation*}
		\begin{split}
		\{r,\phi_0\}=&2\partial_{\xi} r\cdot(x-x_0)-2\partial_xr\cdot(\xi-\xi_0),\\
		\partial_y\phi_0=&2y,\{r,\phi_1\}=\partial_xr\cdot\frac{\xi}{|\xi|^3}.
		\end{split}
		\end{equation*}

		We next take
		$$ f(y,x,\eta,\xi):=\chi_2\big(\frac{\phi_0}{\delta}\big)^2\chi_0\big(1-\frac{\phi}{\delta}\big).
		$$
		The desired functions $a_0,a_1$ are chosen to be the remainders when $f$ is divided by $p=\eta^2-r(y,x,\xi)$, thanks to the Malgrange preparation theorem:
		$$ f(y,x,\eta,\xi)=(\eta^2-r(y,x,\xi))g(y,x,\eta,\xi)+a_1(y,x,\xi)\eta+a_0(y,x,\xi).
		$$
		On the support of $f$, we observe that
		$$ \phi_0(y,x,\xi)=|(y,x,\xi)-(0,x_0,\xi_0)|^2\leq 3\delta, \quad \frac{\eta}{|\xi|}+\phi_0\leq \delta,
		$$
		which implies $\eta\leq \delta|\xi|$. Moreover, on supp$(f)\cap \text{supp}(\partial\chi_2(\delta^{-1}\phi_0))$, we have $\phi_0\geq 2\delta, \phi_0+\phi_1\eta\leq\delta$, and these imply $\eta\leq -\delta|\xi|$, hence $r(y,x,\xi)=\eta^2\geq \delta^2|\xi|^2$, when $p=\eta^2-r=0$.
		
		Direct calculation yields
		$$ H_pf+fM|\xi|+\psi^2=\chi_0\big(1-\frac{\phi}{\delta}\big)H_p\big(\chi_2\big(\frac{\phi_0}{\delta}\big)^2\big)
		-\big(1-\chi_1\big(\frac{\eta}{\delta|\xi|}\big)^2\big)N,
		$$
		with
		$$ N=\chi_2\big(\frac{\phi_0}{\delta}\big)^2\big(\chi_0'\big(1-\frac{\phi}{\delta}\big)
		\frac{H_p\phi}{\delta}-\chi_0\big(1-\frac{\phi}{\delta}\big)M|\xi|\big)\in C^{\infty},\;
		\psi=\chi_1\big(\frac{\eta}{\delta|\xi|}\big)N^{1/2}.
		$$
		Here $N\geq 0$ on supp$(\psi)$ if we choose $\delta>0$ small enough. Observe that when $\eta=r^{1/2}\geq 0$, we have $\chi_0(1-\delta^{-1}\phi)H_p\big(\chi_2\left(\delta^{-1}\phi_0\right)^2\big)=0, \big(1-\chi_1\big(\frac{\eta}{\delta|\xi|}\big)^2\big)N=0$. We then define a function
		$$ \varphi(y,x,\xi)=-\frac{\chi_0\big(1-\frac{\phi}{\delta}\big)H_p\big(\chi_2\big(\frac{\phi_0}{\delta}\big)^2\big)
			-\big(1-\chi_1\big(\frac{\eta}{\delta|\xi|}\big)^2\big)N}{2r^{1/2}}|_{\eta=-r^{1/2}}\mathbf{1}_{r(y,x,\xi)>0}
		$$
		and then
		$$ H_pf+fM|\xi|+\psi^2=\varphi(\eta-r^{1/2}), \quad \textrm{when }p=\eta^2-r=0.
		$$
		Therefore, on $p=0$, we have
		$$ H_pa+aM|\xi|+\psi^2=\varphi(\eta-r^{1/2}).
		$$
		It is left to check the smoothness of functions $\varphi, \psi$ and $\rho$. Indeed, on the support of $\psi$,$|\phi_1\eta|\leq 2\delta, \phi_0\leq 3\delta$, and then $ 1-\frac{\phi}{\delta}\leq 3$. Notice that $\frac{\chi_0(t)}{\chi_0'(t)}=t^2$, we have
		$$ N^{1/2}=\chi_2\big(\frac{\phi_0}{\delta}\big)\sqrt{\chi_0'\big(1-\frac{\phi}{\delta}\big)\frac{H_p\phi}{\delta}}
		G\big(\frac{M|\xi|\delta}{H_p\phi},1-\frac{\phi}{\delta}\big)\in C^{\infty},
		$$
		since the function $G(a,t)=\sqrt{1-at^2}\in C^{\infty}$ for $t\leq 3, |a|\ll 1$. This implies that $\psi\in C^{\infty}$, provided that $\delta$ is chosen small enough.
		
		For $\varphi$, the smoothness comes from the fact that on the support of  $$\chi_0\big(1-\frac{\phi}{\delta}\big)H_p\big(\chi_2\big(\frac{\phi_0}{\delta}\big)^2\big)|_{\eta^2=r}
		-\big(1-\chi_1\big(\frac{\eta}{\delta|\xi|}\big)^2\big)N|_{\eta^2=r} ,$$ we have $r\geq \delta^2|\xi|^2$. Moreover, $\varphi$ has compact support.
		
		Finally, from the definition of $a$, we have
		$$ a_1(y,x,\xi)=\frac{f(y,x,\eta,\xi)-f(y,x,-\eta,\xi)}{2\eta}|_{\eta=\sqrt{r(y,x,\xi)}}\in C_c^{\infty},
		$$
		$$a_0(y,x,\xi)=\frac{f(y,x,\eta,\xi)+f(y,x,-\eta,\xi)}{2}|_{\eta=\pm\sqrt{r(y,x,\xi)}}\in C_c^{\infty}.
		$$
		we deduce that $\frac{\partial f}{\partial\eta}=-\frac{\phi_1}{\delta}\chi_2\big(\frac{\phi_0}{\delta}\big)^2\chi_0'\big(
		1-\frac{\phi}{\delta}\big)<0$, hence
		\begin{equation*}
		\begin{split}
		\frac{f(y,x,\eta,\xi)-f(y,x,-\eta,\xi)}{2\eta}=&\frac{1}{2}\int_{-1}^{1}\frac{\partial f}{\partial \eta}(y,x,s\eta,\xi)ds<0.
		\end{split}
		\end{equation*}
		Define
		$$ t(y,x,\xi)=\big(-\frac{1}{2}\int_{-1}^{1}\frac{\partial f}{\partial \eta}(y,x,s\eta,\xi)ds\big|_{\eta=\sqrt{r(y,x,\xi)}}\big)^{1/2},
		$$
		one can show that $t$ is a smooth function with compact support.
		
		The last observation is that $a=f>0$ on $p=0$, hence $s=f^{1/2}|_{p=0}\in C_c^{\infty}$.
		
		We give some more calculations: Let $\psi_0=\psi|_{y=0}, \psi_1=\frac{\partial\psi}{\partial\eta}|_{y=0}$, when $\eta=r=0$. Thus at $(x_0,\xi_0)$,
		\begin{equation*}
		\begin{split}
		t(x_0,\xi_0)=&\sqrt{\frac{\chi_0'(1)\phi_1(x_0,\xi_0)}{\delta}},\\
		\psi_0(x_0,\xi_0)^2=&\chi_0'(1)\frac{\partial r}{\partial y}(0,x_0,\xi_0)\frac{\phi_1(0,x_0,\xi_0)}{\delta}-\chi_0(1)M|\xi|>0,\\
		2\psi_0\psi_1(x_0,\xi_0)=&-\chi_0''(1)\frac{\partial r}{\partial y}(x_0,\xi_0)\frac{\phi_1(x_0,\xi_0)^2}{\delta^2}-\chi_0'(1)\{r,\phi_1\}(0,x_0,\xi_0)\frac{1}{\delta}\\+&\frac{\chi_0'(1)M\phi_1(0,x_0,\xi_0)}{\delta}>0,\textrm{for $\delta$ small enough since $\chi_0''(1)<0 $.}
		\end{split}
		\end{equation*} 
		Observe that near $(x_0,\xi_0)$, we have $\frac{\psi_1}{\psi_0}\sim -\frac{\phi_1(x_0,\xi_0)\chi_0''(1)}{2\chi_0'(1)\delta}$, provided that $\delta$ is small enough. Now if we make a different choice of $\widetilde{\delta}>0$, the difference between two ratios $\frac{\psi_1}{\psi_0}$ and $\frac{\widetilde{\psi}_1}{\widetilde{\psi}_0}$ is non-zero. This implies that we can choose a further cut-off $\chi$ near $(0,x_0,\xi_0)$ such that $\|\varphi\textrm{Op}_h(\chi)\varphi_1u\|_{L^2(Y_+)}=o(1)$ and $\|\varphi\textrm{Op}_h(\chi)hD_y\varphi_1u\|_{L^2(Y_+)}=o(1)$ from $\|\varphi\textrm{Op}_h(\psi_0+\psi_1\eta)\varphi_1u\|_{L^2(Y_+)}=o(1)$ and $\|\varphi\textrm{Op}_h(\widetilde{\psi}_0+\widetilde{\psi}_1\eta)\varphi_1u\|_{L^2(Y_+)}=o(1)$.
	\end{proof}
	Next we recall the proof of Lemma \ref{MS1chapter1}, which is essentially given in \cite{Melrose-SjostrandI}.
	
	\begin{proof}[Proof of Lemma \ref{MS1chapter1}]:
		From the transversality, we can choose a new coordinate $(s,t)$ in $U$ such that $\rho_0=(0,0)$ and $H_{-r_0}=\partial_t$ in this coordinate.\\
		Step 1. Consider the function
		$ \displaystyle{\chi(u)=e^{\frac{1}{u-3/4}}\mathbf{1}_{u<3/4}.}
		$
		It is easy to check that $\chi$ is smooth and non-increasing with the property:
		$$ \partial^{N}\chi(u)=O((-\chi')^{1/{m}}),\quad \forall N\in\mathbb{N},m>1,\textrm{locally uniformly.}
		$$
		Step 2. Next we choose $\beta\in C^{\infty}(\mathbb{R})$ such that $\beta\geq 0$ vanishing on $(-\infty,-1)$ and strictly increasing on $(-1,-\frac{1}{2}),$ equaling to $1$ on $(-\frac{1}{2},\infty)$. We modify $\beta$ such that
		$$ \partial^{N}\beta=O(\beta^{1/m}),\quad \forall N\in\mathbb{N},m>1,\textrm{locally uniformly.}
		$$
		Step 3. Choose $f\in C^{\infty}(\mathbb{R})$ so that $f$ vanishes on $(-\infty,1/2)$ and is strictly increasing and convex on $(1/2,\infty)$ with $f(1)>1$.
		
		Now we set
		$$ a_{\delta}=\beta\big(\frac{3t}{4\delta^2}\big)\chi\big(\frac{t}{\sigma\delta}+\frac{|s|^2}{\delta^4}+f\big(\frac{y}{\delta^2}\big)\big),
		$$
		and 
		$$ g_{\delta}=-\beta\big(\frac{3t}{4\delta^2}\big)H_p(\chi(u))$$
		with $\displaystyle{u=\frac{3t}{4\sigma\delta}+\frac{|s|^2}{\delta^4}+f\big(\frac{y}{\delta^2}\big)}.$
		Finally we define 
		$ h_{\delta}=-H_pa_{\delta}-g_{\delta}.
		$
		Note that
		$$ \textrm{supp }(a_{\delta})=\big\{(y,s,t):-\delta^2\leq t,\frac{t}{\sigma\delta}+\frac{|s|^2}{\delta^4}+f\big(\frac{y}{\delta^2}\big)\leq \frac{3}{4}\big\},
		$$
		hence it is clear that (1)(2)(3)(4) in Lemma \ref{MS1chapter1} are satisfied.
		
		Since $r=r_0+O(y)$,
		when $H_{-r}$ acts on functions independent of $\eta$ we have
		$$ H_p=\partial_t+O(y)\partial_s+O(y)\partial_t+O(\delta)\partial_y,
		$$ due to the bound
	 $|\eta|=O(\delta)$. Therefore, we have
		\begin{equation*}
		\begin{split}
		-g_{\delta}&=\beta\big(\frac{t}{\delta^2}\big)\big(\big(\frac{1}{\delta\sigma}+O(\delta^2)\frac{|s|}{\delta^4}+O(\delta^2)\frac{1}{\delta\sigma}+O(\delta)\frac{1}{\delta^2}\big)\chi'(u)+O(1)\chi(u)\big)\\
		&=\beta\big(\frac{t}{\delta^2}\big)
		\big(\frac{3}{4\delta\sigma}+O\big(\frac{\delta}{\sigma}\big)+O\big(\frac{1}{\delta}\big)+O(1)))\chi'(u)\big)\\
		&\sim \beta\big(\frac{t}{\delta^2}\big)\frac{1}{\delta\sigma}\chi'(u),
		\end{split}
		\end{equation*}
		provided that $\delta,\sigma\ll 1$. In the calculation above, we have used the fact that $\chi(u)=O(\chi'(u))$ on the support of $g_{\delta}$. Thus (5) in Lemma \ref{MS1chapter1} follows. 
		
		(6) follows from the construction of $\chi$. 
		
		To check (7), we observe that
		supp$(g_{\delta})\cup $supp$(h_{\delta})\subset $supp$(a_{\delta})$. Moreover, from the construction, $g_{\delta},h_{\delta}$ are independent of $\eta$ whenever $0\leq y<\frac{\delta^2}{2}$. Finally, to check the support of $h_{\delta}$, we write
		$ h_{\delta}=-H_{p}\big(\beta\big(\frac{t}{\delta^2}\big)\big)\chi(u)$.  Since $\beta$ is independent of $y,\eta$, we have
		$ H_{p}\left(\beta\left(\frac{t}{\delta^2}\right)\right)=H_{-r}\left(\beta\left(\frac{t}{\delta^2}\right)\right),
		$ which is supported on $I\times L^-(\delta,\delta^2)\times\mathbb{R}_{\eta}$, thanks to supp $\beta'\subset [-1,-\frac{1}{2}]$.
	\end{proof}
	
	\section*{E.Proof of Lemma \ref{HodgeLaplacechapter1}}
	\begin{lemma}\label{HodgeLaplace1chapter1}
		In local coordinate $Y_+$, we have
		$$P_h=-h^2\frac{\ov{g}}{\sqrt{G}}\frac{\partial}{\partial y}\big(\sqrt{G}\ov{g}^{-1}\partial_y\big)+R_h=h^2D_y^2+\mathrm{Op}_h(r)+O_{L^2\rightarrow L^2}(h).
		$$	
		Moreover, $R_h$ is a matrix-valued second order differential operator in $x$ with scalar principal symbol $r(y,x,\xi)=1-\lambda(y,x,\xi)^2$, which is self-ajoint with respect to the $(\cdot|\cdot)_{L^2(Y_+)}$.  
	\end{lemma}
	\begin{proof}
		Denote by $y=x^0$, $\partial_0=\partial_y,\partial_j=\partial_{x_j},j=1,2,\cdots d-1.$ Let $u\in \Lambda^1(Y_+)$ and $w\in \Lambda^2(Y_+)$ written in the form $$u=u_0\mathrm{d}x^0+u_j\mathrm{d}x^j,\quad w=w_{0j}\mathrm{d}x^0\wedge \mathrm{d}x^j+w_{jk}\mathrm{d}x^j\wedge \mathrm{d}x^k.$$
		We have from direct calculation that
		\begin{equation*}
		\begin{split}
		\mathrm{d}u=&(\partial_{0}u_j-\partial_ju_0)\mathrm{d}x^0\wedge \mathrm{d}x^j+\partial_ku_j\mathrm{d}x^j\wedge \mathrm{d}x^k,\\
		\mathrm{d}^*u=&-\frac{1}{\sqrt{G}}\partial_0(u_0\sqrt{G})-\frac{1}{\sqrt{G}}\partial_j(g^{jk}u_k\sqrt{G}),\\
		\mathrm{d}^*w=&\frac{1}{\sqrt{G}}\partial_k(w_{0j}g^{jk}\sqrt{G})\mathrm{d}x^0\\-&\frac{1}{\sqrt{G}}g_{jl}\left(\partial_0(w_{0k}g^{kl}\sqrt{G})+\partial_m(\sqrt{G}w_{pk}(g^{pl}g^{km}-g^{pm}g^{kl})\right)\mathrm{d}x^j
		\end{split}
		\end{equation*} 
		From direct calculation, $h^2\Delta_Hu=h^2(\mathrm{d}\mathrm{d}^*+\mathrm{d}^*\mathrm{d})u=v_0\mathrm{d}x^0+v_j\mathrm{d}x^j+R_hu$
		with
		\begin{equation*}
		\begin{split}
		v_0=&-h^2\partial_0^2u_0-h^2\frac{\partial_0(\sqrt{G})}{\sqrt{G}}\partial_0u_0,\\
		v_j=&-h^2\partial_0^2u_l-\frac{h^2}{\sqrt{G}}g_{jl}\partial_0(g^{kl}\sqrt{G})\partial_0 u_k
		\end{split}
		\end{equation*}
		and the $R_hu$ consists only the tangential derivatives $\partial_{j}$. Hence in the matrix form, 
		$$ v=L_hu:=-h^2\frac{\ov{g}}{\sqrt{G}}\frac{\partial}{\partial y}\big(\sqrt{G}\ov{g}^{-1}\frac{\partial u}{\partial y}\big).
		$$ 
		Moreover, one easily verified that $L_h^*=L_h$, thus $R_h^*=R_h$.
	\end{proof}

	\section*{F.Proof of Lemma \ref{compositionlawchapter1}}
	\begin{proof}[Proof of Lemma \ref{compositionlawchapter1}]
		For our need, it suffices to prove the last assertion. We first let $A_h=a(y,x,hD_y,hD_x)$ and $B_h=b(y,x,hD_x)$, then
		$$ A_hB_hu(y,x)=
		\frac{1}{(2\pi h)^d}\iint e^{\frac{i(x-x')\xi+i(y-y')\eta}{h}}\varphi(y',y,x,\eta,\xi)u(y',x')dy'dx'd\xi d\eta,
		$$
		where
		$$ \varphi(y',y,x,\eta,\xi)=\frac{1}{(2\pi h)^{d-1}}\iint e^{\frac{i(x-z)(\xi'-\xi)}{h}}a(y,x,\eta,\xi')b(y',z,\xi)d\xi'dz.
		$$
		Talor expansion gives
		$$\varphi(y',y,x,\eta,\xi)=\varphi(y,y,x,\eta,\xi)+(y'-y)\int_0^1\partial_{y'}\varphi(ty'+(1-t)y,y,x,\eta,\xi)dt.
		$$
		Denote by $c(y,x,\eta,\xi)=\varphi(y,y,x,\eta,\xi)$, it is obvious that $c$ is an interior symbol, since it can be viewed as a tangential symbol for fixed $\eta$, and we have
		$ (1+|\xi|)^m\lesssim (1+|\xi|+|\eta|)^m$ for all $m\in\mathbb{R}$ on the support of $c$, thanks to the support property of $a$. Now we note $C_h=c(y,x,hD_y,hD_x)$, and write
		$ A_hB_hu=C_hu+R'_hu,
		$
		where
		\begin{align*}
		R'_hu(y,x)&=\int_0^1dt\frac{1}{(2\pi h)^{d}}\iint e^{\frac{i(x-x')\xi+i(y-y')\eta}{h}}(y-y')\partial_{y'}c_t(y',y,x,\eta,\xi)u(y',x')dy'dx' d\xi d\eta\\
		&=ih\int_0^1dt\frac{1}{(2\pi h)^{d}}\iint
		e^{\frac{i(x-x')\xi+i(y-y')\eta}{h}}\partial_{\eta}\partial_{y'}c_t(y',y,x,\eta,\xi)u(y',x')dy'dx' d\xi d\eta\\
		&=:ih\int_0^1C_tu(y,x)dt,
		\end{align*}
		with
		$ c_t(y',y,x,\eta,\xi)=\varphi(ty'+(1-t)y,y,x,\eta,\xi).
		$
		Notice that
		$$ \partial_{\eta}\partial_{y'}c_t(y',y,x,\eta,\xi)=\frac{1}{(2\pi h)^{d-1}}\iint e^{\frac{i(x-z)(\xi'-\xi)}{h}}\partial_{\eta}a(y,x,\eta,\xi')(\partial_{y}b)(ty'+(1-t)y,z,\xi)d\xi'dz.
		$$
		We need to be careful here since $\partial_y b$ only exists for $y>0$ and at the point $y=0$, the right derivative $\displaystyle{(\partial_y^m)^+b(0):=\lim_{y\rightarrow 0^+}\partial^mb(y)}$ exists for any order $m$. Since we are dealing with Dirichlet boundary condition, we always apply a tangential operator $B(y,x,hD_x)$ to functions $u(y,x)$ with $u|_{y=0}=0$ in the trace sense. We could thus extend $u(y',x')$ by $u(y',x')\mathbf{1}_{y'\geq 0}$in $y'$ in the expression of the form
		$$\frac{1}{(2\pi h)^d}\iint e^{\frac{i(x-x')\xi'+i(y-y')\eta}{h}}\varphi(y',y,x,\eta,\xi)u(y',x')dy'dx'd\xi d\eta.
		$$
		Therefore, we have
		$$ \sup_{y,y'\geq 0,0<t<1,z,\xi}|\partial_{z}^{\alpha}\partial_{\xi}^{\beta}b(ty'+(1-t)y,z,\xi)|\leq C_{m,\alpha,\beta},\forall m\in\mathbb{N},\alpha,\beta\in\mathbb{N}^{d-1}.
		$$

		Now it is reduced to prove the uniform $L^2$ boundness of the operator
		$$T_hu=\int_{\mathbb{R}^d}K_h(y',x',y,x)u(y'
		,x')dy'dx',
		$$
		with kernel
		$$ K_h(y',x',y,x)=\frac{1}{(2\pi h)^d}\int_{\mathbb{R}^d}e^{\frac{i(x-x')\xi+i(y-y')\eta}{h}}H_t(y',y,x,\eta,\xi)d\eta d\xi,$$
		where
		$$ H_t(y',y,x,\eta,\xi)=\mathbf{1}_{y',y\geq 0}\frac{1}{(2\pi h)^{d-1}}\int e^{\frac{iz\zeta}{h}}a_1(y,x,\eta,\xi+\zeta)b_1(ty'+(1-t)y,x-z,\zeta)dzd\zeta.
		$$
		From Schur's test, we need to show
		\begin{gather*}
		\sup_{(y,x)\in\mathbb{R}_+^d}\int_{\mathbb{R}_+^d} |K_h(y',x',y,x)|dy'dx'\leq C_1<\infty,\\
		\sup_{(y',x')\in\mathbb{R}_+^d}\int_{\mathbb{R}_+^d} |K_h(y',x',y,x)|dydx\leq C_2<\infty,
		\end{gather*}
		with $C_1,C_2$ independent of $h$ and $t$.
		
		To this end, we   define
		$$ k_h(y,x,w,v):=\frac{1}{(2\pi)^d}\int_{\mathbb{R}^d} e^{iv\xi+iw\eta}H_t(y-hw,y,x,\eta,\xi)d\eta d\xi,
		$$
		hence,
		$$ T_hu(y,x)=\frac{1}{h^d}\int_{\mathbb{R}_+^d}k_h\big(y,x,\frac{y-y'}{h},\frac{x-x'}{h}\big)u(y',x')dy'dx'.
		$$
		Notice that $H_t(y',y,x,\eta,\xi)$ is a tangential symbol, parametrized by $(y',y,\eta)$. Moreover, it is compactly supported in $(y,x,\eta,\xi)$ variables, uniformly in the first variable $y'$. Thus, $\partial_{\eta}^m\partial_{\xi}^{\alpha}H_t(y-hw,y,x,\eta,\xi)$ has compact support in $(\eta,\xi)$ and $$|\partial_{\eta}^m\partial_{\xi}^{\alpha}H_t(y-hw,y,x,\eta,\xi)|\leq C_{m,\alpha}$$ for any $m\in\mathbb{N}$ and $\alpha\in\mathbb{N}^{d-1}$.
		Thus, doing integration by part in the expression of $k_h$, we have
		$$ \sup_{(y,x)}|k_h(y,x,w,v)|\leq C (1+|w|+|v|)^{-(d+1)}.
		$$
		Therefore, we obtain
		\begin{align*}
		\int_{\mathbb{R}_+^d}\left|K_h(y',x',y,x)\right|dy'dx'&=\frac{1}{h^d}\int_{\mathbb{R}_+^d}\big|k_h\big(
		y,x,\frac{y-y'}{h},\frac{x-x'}{h}\big)\big|dy'dx'\\
		&=\int_{\mathbb{R}^d}|k_h(y,x,w,v)|dwdv\\
		&\leq C_1,
		\end{align*}
		and
		\begin{align*} 
		\int_{\mathbb{R}_+^d}|K_h(y,x,y',x')|dydx&=\frac{1}{h^d}\int_{\mathbb{R}_+^d}\big|k_h\big(
		y,x,\frac{y-y'}{h},\frac{x-x'}{h}\big)\big|dydx\\
		&\leq
		\int_{\mathbb{R}^d}\sup_{(y,x)}|k_h(y,x,w,v)|dwdv\\
		&\leq C_2.
		\end{align*}
	\end{proof}
\subsection*{Acknowledgement}This work is supported by the European Research Council, ERC-2012-ADG, project number 320845: Semi classical Analysis and Partial Differential Equations. It is a part of Phd thesis of the author. He would like to thank his advisor, Gilles Lebeau, for his encouragement and many fruitful discussions.

\renewcommand\refname{References}
\bibliographystyle{plain}
\bibliography{propagationfinal}
\end{document}